\documentclass[12pt,reqno]{article}
\usepackage{amsmath,amsfonts,amssymb,amscd,amsthm,amsbsy, epsf,calc,enumerate,cite}
\usepackage{color}
\usepackage{graphicx,gensymb}
\usepackage{epsfig,esint}
\usepackage{verbatim}
\usepackage{soul}
\definecolor{purple}{rgb}{0.65, 0, 1}
\definecolor{pink}{rgb}{1, 0, 0.9}
\definecolor{orange}{rgb}{1,.5,0}

\numberwithin{equation}{section}
\newtheorem{theorem}{Theorem}[section]

\newtheorem{lemma}[theorem]{Lemma}
\newtheorem{definition}[theorem]{Definition}
\newtheorem{proposition}[theorem]{Proposition}
\newtheorem{corollary}[theorem]{Corollary}

\newcommand{\comments}[1]{}

\newcounter{DTheoremCounter}

\newcounter{DAppCounter}
\newcounter{claimCounter}
\definecolor{purple}{rgb}{0.65, 0, 1}
\definecolor{green}{rgb}{0, 1, 0}
\definecolor{orange}{rgb}{1,.5,0}

\def \Rm {\mathbb R}
\def \Tm {\mathbb T}
\def \z {z}
\def \Z {Z}

\def \eps {\epsilon}

\newcommand{\vertiii}[1]{{\vert\kern-0.25ex\vert\kern-0.25ex\vert #1
    \vert\kern-0.25ex\vert\kern-0.25ex\vert}}

\begin{document}

\title{Local regularity for the modified SQG patch equation}
\author{Alexander Kiselev\footnote{Department of Mathematics, Rice
    University, 6100 S. Main St.,
Houston, TX 77005-1892; kiselev@rice.edu}
\and Yao Yao\footnote{School of Mathematics, Georgia Institute of Technology, 686 Cherry Street, Atlanta, GA 30332-0160; yaoyao@math.gatech.edu}
\and Andrej Zlato\v s\footnote{Department of Mathematics, University of
  Wisconsin, 480 Lincoln Dr.
Madison, WI 53706-1325; zlatos@math.wisc.edu}}

\maketitle

\begin{abstract}
We study the patch dynamics on the whole plane and on the half-plane for a family of active scalars called modified SQG equations.
These involve a parameter $\alpha$ which appears in the power of the kernel in their Biot-Savart laws and
describes the degree of regularity of the equation.  The values $\alpha=0$ and $\alpha=\frac 12$ correspond to the 2D Euler and SQG equations, respectively.
We establish here local-in-time regularity for these models, for all $\alpha\in(0,\frac 12)$ on the whole plane and for all small $\alpha>0$ on the half-plane.
We use the latter result in \cite{KRYZ2}, where we show existence of regular initial data on the half-plane which lead to a finite time singularity.
\end{abstract}

\section{Introduction}

Two of the most important models in two-dimensional fluid dynamics are the (incompressible) 2D Euler equation, modeling motion of inviscid fluids, and the surface quasi-geostrophic (SQG) equation, which is used in atmospheric science models, appearing for instance in Pedlosky \cite{Ped}. In the mathematical literature, the SQG equation was first discussed in the work of Constantin, Majda, and Tabak \cite{CMT}.   Both these equations (the former in the vorticity formulation) can be written in the  form
\begin{equation}\label{sqg}
\partial_t \omega + (u \cdot \nabla) \omega =0,
\end{equation}
along with initial condition $\omega(\cdot,0)=\omega_0$ and the Biot-Savart law $u := \nabla^\perp (-\Delta)^{-1+\alpha} \omega$. Here $\nabla^\perp := (\partial_{x_2}, -\partial_{x_1})$, and the Euler and SQG cases are obtained by taking $\alpha=0$ and $\alpha=\frac 12$, respectively.  Note that the Biot-Savart law for the 2D Euler equation is therefore more regular (by one derivative) than that of the SQG equation.

Global regularity of solutions to the 2D Euler equation  has been known since the works of Wolibner \cite{Wolibner} and H\"older \cite{Holder}.
The necessary estimates  barely close, and the upper bound on the growth of the derivatives of the vorticity is double exponential
in time. Recently,  Kiselev and \v Sver\' ak showed that this upper bound is sharp by constructing an example  of a solution
to the 2D Euler equation on a disk whose gradient indeed grows double exponentially in time \cite{KS}. Some earlier examples of unbounded growth are due to Yudovich \cite{Yud1,Yud2},
Nadirashvili \cite{Nad}, and Denisov \cite{Den1,Den2}, and exponential growth on a domain without a boundary (the torus $\Tm^2$) was recently shown to be possible by Zlato\v s \cite{ZlaEuler}.
On the other hand, while existence of global weak solutions for the SQG equation (which shares many of its features with the 3D Euler equation --- see, e.g., \cite{CMT,mb,Rodrigo}) was proved by Resnick \cite{Resnick}, the global regularity vs finite time blow-up question for it is a major open problem.

Both the 2D Euler and SQG equations belong to the class of active scalars,  equations of the form \eqref{sqg} where the fluid velocity is determined from the advected scalar $\omega$ itself.
A natural family of active scalars which interpolates between the 2D Euler and SQG equations is given by \eqref{sqg} with $\alpha\in(0,\frac 12)$ in the above Biot-Savart law.  
This family  has been called modified or generalized SQG equations in the literature
(see, e.g., \cite{CIW}, or the paper \cite{PHS} by Pierrehumbert, Held, and Swanson  for a geophysical literature reference).
The global regularity vs finite time blow-up question is still open for all $\alpha>0$.

While the above works studied active scalars with sufficiently smooth initial data, an important class of solutions to  these equations arises from rougher initial data.  Of particular interest is the case of characteristic functions of domains with smooth boundaries, or more generally,  sums of characteristic functions of such domains multiplied by some coupling constants.
Solutions originating from such initial data are called vortex patches, and they model flows with abrupt variations in their vorticity. The latter, including hurricanes and tornados, are common in nature.  Existence and uniqueness of vortex patch solutions to the 2D Euler equation on the whole plane  goes back to the work of  Yudovich \cite{Yudth}, and
regularity in this setting refers to a sufficient smoothness of the patch boundaries as well as to a lack of both self-intersections of each patch boundary and touching of different patches.

The vortex patch problem can be viewed as an interface evolution problem, and singularity formation for 2D Euler patches had initially been conjectured by Majda \cite{Majda} based on relevant numerical simulations by Buttke \cite{Butt}.  Later, simulations by Dritschel, McIntyre, and Zabusky \cite{DM,DZ}
questioning the singularity formation prediction appeared, and we refer to \cite{Pulli} for a review of these and related works.  This controversy was settled in 1993, when  Chemin \cite{c} proved that the boundary of a 2D Euler patch remains regular for all times, with a double exponential upper bound on the temporal growth of its curvature (see also the work by Bertozzi and Constantin \cite{bc} for a different proof).

The patch problem for the SQG equation is comparatively more involved to set up rigorously. Local existence and uniqueness in the class of weak solutions of the special type $\omega(x,t)=\chi_{\{x_2<\varphi(x_1,t)\}}$ with $\varphi\in C^\infty$ and periodic in $x_1$, corresponding to (single patch) initial data with the same property, was proved by Rodrigo \cite{Rodrigo}.
For SQG and modified SQG patches with boundaries which are simple closed $H^3$ curves, local existence was obtained by Gancedo \cite{g} via solving a related contour equation whose solutions are some parametrizations of the patch boundary (uniqueness of solutions was also proved for the contour equation when $\alpha\in(0,\frac 12)$, although not for the original modified SQG patch equation).  Local existence of such contour solutions in the more singular case $\alpha\in(\frac 12,1]$ was obtained by Chae, Constantin, C\' ordoba, Gancedo, and Wu \cite{CCCGW}.  Finally, existence of splash singularities (touch of exactly  two segments of a patch boundary, which remains uniformly $H^3$) for the SQG equation was ruled out by Gancedo and Strain \cite{GS}.

A computational study of the SQG and modified SQG patches by C\' ordoba, Fontelos, Mancho, and Rodrigo \cite{CFMR} (where the patch problem for the modified SQG equation first appeared) suggested finite time singularity formation, with two patches touching each other and simultaneously developing corners at the point of touch. A more careful numerical study
by Mancho \cite{Mancho} suggests self-similar elements involved in this singularity formation
process, but its rigorous confirmation and understanding is still lacking.

In this paper, we consider the patch evolution for the modified SQG equations, both on the whole plane and on the half-plane, and prove local-in-time regularity for these models (for all $\alpha\in(0,\frac 12)$ on the plane and for all sufficiently small $\alpha>0$ on the half-plane).  Our motivation is, in fact, primarily the half-plane case because in the companion paper \cite{KRYZ2} we show existence of finite time blow-up for patch solutions to the modified SQG equation with small $\alpha>0$ on the half-plane.
To the best of our knowledge, this is the first rigorous proof of finite time blow-up in this type of fluid dynamics models.

Let us now turn to the specifics of the model we will study.  We only consider here the case $\alpha\in(0,\frac 12)$, and we concentrate on the half-plane case $D := \Rm\times\Rm^+$.  This is both because this case is our main motivation, and because the proofs are more involved (in fact, the whole-plane proofs are essentially contained in the half-plane ones).
The corresponding patch evolution can then be formally defined via the Biot-Savart law
\begin{equation}
u(x, t) :=  \int_D \left( \frac{(x-y)^\perp}{|x-y|^{2+2\alpha}} -
\frac{(x-\bar y)^\perp}{|x-\bar y|^{2+2\alpha}} \right) \omega(y,t) dy
\label{eq:velocity_law}
\end{equation}
for $x \in \bar D$, along with the requirement that $\omega$ is advected by the flow given by $u$, that is,
\begin{equation}\label{1.31}
\omega(x,t) = \omega \left(\Phi^{-1}_t(x),0\right),
\end{equation}
where
\begin{equation}\label{eq:alpha}
\frac{d}{dt}\Phi_t(x) = u\left(\Phi_t(x),t \right) \qquad \text{and} \qquad \Phi_0(x)=x.
\end{equation}
Here $v^{\perp}:=(v_2, -v_1)$ and $\bar v:=(v_1, -v_2)$ for $v=(v_1,v_2)$, and we note that the integral in \eqref{eq:velocity_law} equals  $\nabla^\perp (-\Delta)^{-1+\alpha}\omega$  (up to a positive pre-factor, which can be dropped without loss due to scaling), with the Dirichlet Laplacian on $D$.  The vector field $u$ is then divergence free and tangential to the boundary $\partial D$ (i.e., $u_2(x,t)=0$ when $x_2=0$).

We have to be careful, however, with the rigorous definition of the evolution because the low regularity of the fluid velocity $u$ need not allow for a unique definition of trajectories from \eqref{eq:alpha} when $\alpha>0$ (existence will not be an issue here because $u$ is continuous for $\alpha<\frac 12$).
We introduce here the following Definition \ref{D.1.1} which, as we discuss below, encompasses various previously used definitions.  We start with a definition of some  norms of boundaries of domains in $\Rm^2$, letting here $\mathbb T:=[-\pi,\pi]$ with $\pm\pi$ identified.

\begin{definition} \label{D.1.0}
Let  $\Omega\subseteq \mathbb{R}^2$ be a bounded open set whose boundary $\partial\Omega$ is a simple closed $C^{1}$ curve with arc-length $|\partial\Omega|$.  We call a {\it constant speed parametrization} of $\partial\Omega$ any counter-clockwise parametrization $z:\mathbb{T}\to \mathbb{R}^2$ of $\partial\Omega$ with $|z'|\equiv \frac {|\partial\Omega|}{2\pi}$ on $\mathbb T$ (all such $z$ are translations of each other),
and we define $\|\Omega\|_{C^{m,\gamma}}:=\|z\|_{C^{m,\gamma}}$ and $\|\Omega\|_{H^{m}}:=\|z\|_{H^{m}}$.
\end{definition}

{\it Remark.}
 It is not difficult to see (using Lemma \ref{lem:regularity} below), that an $\Omega$ as above satisfies $\|\Omega\|_{C^{m,\gamma}}<\infty$ (resp.~$\|\Omega\|_{H^{m}}<\infty$) precisely when for some $r>0$, $M<\infty$, and each $x\in\partial\Omega$, the set $\partial\Omega\cap B(x,r)$ is (in the coordinate system centered at $x$ and with axes given by the tangent and normal vectors to $\partial\Omega$ at $x$) the graph of a function with $C^{m,\gamma}$ (resp. $H^m$) norm less than $M$.

Next, let $d_H(\Gamma,\tilde\Gamma)$ be the Hausdorff distance of sets $\Gamma,\tilde\Gamma$, and for a set $\Gamma\subseteq\Rm^2$, vector field $v:\Gamma\to\Rm^2$, and $h\in\Rm$, we let
\[
X_{v}^h[\Gamma]:= \{ x+hv(x)\,:\, x\in\Gamma\}.
\]
Our definition of patch solutions to \eqref{sqg}-\eqref{eq:velocity_law} on the half-plane is now as follows.

\begin{definition}\label{D.1.1}
Let $D:=\Rm\times\Rm^+$, let $\theta_1,\dots,\theta_N\in\Rm\setminus\{0\}$, and for each $t\in[0,T)$, let $\Omega_1(t),\dots,\Omega_N(t)\subseteq D$ be bounded open sets whose boundaries $\partial \Omega_k(t)$ are pairwise disjoint simple closed curves, such that each $\partial \Omega_k(t)$ is also continuous in $t\in[0,T)$ with respect to $d_H$.  Denote $\partial\Omega(t):=\bigcup_{k=1}^N \partial\Omega_k(t)$ and let 
\begin{equation} \label{1.55}
\omega(\cdot,t) := \sum_{k=1}^N \theta_k \chi_{\Omega_k(t)}.
\end{equation}

If for each $t\in(0,T)$ we have
\begin{equation}\label{1.3}
\lim_{h\to 0} \frac{d_H \Big(\partial\Omega(t+h),X_{u(\cdot,t)}^h[\partial\Omega(t)] \Big)}h = 0,
\end{equation}
with $u$ from \eqref{eq:velocity_law},
then $\omega$ is a patch solution to \eqref{sqg}-\eqref{eq:velocity_law} on the time interval $[0,T)$.  If we also have $\sup_{t\in [0,T']} \|\Omega_k(t)\|_{C^{m,\gamma}}<\infty$ (resp.~$\sup_{t\in [0,T']} \|\Omega_k(t)\|_{H^{m}}<\infty$)  for each $k$ and $T'\in(0,T)$,
then $\omega$ is a $C^{m,\gamma}$ (resp.~$H^m$) patch solution to \eqref{sqg}-\eqref{eq:velocity_law} on  $[0,T)$.
\end{definition}

{\it Remarks.}
1.   Continuity of $u$ (which is not hard to show, see the last claim in the elementary Lemma \ref{lemma:uniform_u_bound} below)
and \eqref{1.3} mean that for patch solutions, $\partial\Omega$ is moving with velocity $u(x,t)$ at $t\in[0,T)$ and $x\in \partial\Omega(t)$.

2.  We note that our definition encompasses well-known definitions for the 2D Euler equation  in terms of \eqref{eq:alpha} and in terms of the normal velocity at $\partial\Omega$.  Indeed, if $\omega$ satisfies $\partial\Omega_{k}(t) = \Phi_t(\partial\Omega_{k}(0))$ for each $k$ and $t\in[0,T)$ and the patch boundaries remain pairwise disjoint simple closed curves, then continuity of $u$, compactness of $\partial\Omega(t)$, and \eqref{eq:alpha} show that $\omega$ is a patch solution to  \eqref{sqg}-\eqref{eq:velocity_law} on  $[0,T)$.  Moreover, if $\partial\Omega(t)$ is $C^1$  and $n_{x,t}$ is the outer unit normal vector at $x\in \partial\Omega(t)$, then it is easy to see that \eqref{1.3} is equivalent to motion of $\partial\Omega(t)$ with outer normal velocity $u(x,t)\cdot n_{x,t}$ at each $x\in\partial\Omega(t)$ (which can be defined in a natural way by \eqref{1.3} with $u(\cdot,t)$ replaced by $(u(\cdot,t)\cdot n_{\cdot,t})n_{\cdot,t}$).
However, Definition \ref{D.1.1} can be stated even if $\Phi_t(x)$ cannot be uniquely defined for some $x\in\partial\Omega(0)$ (when $\alpha>0$, this might even be the case for $x\notin\partial\Omega(0)$, as the hypotheses of Proposition~\ref{P.1.3}(a) below suggest) or when $\partial\Omega(t)$ is not $C^1$.

3. As we show at the end of this introduction, $C^1$ patch solutions to \eqref{sqg}-\eqref{eq:velocity_law} are also weak solutions to \eqref{sqg} in the sense that for each $f\in C^1(\bar D)$ we have
\begin{equation} \label{1.6}
\frac d{dt} \int_D \omega(x,t)f(x)dx = \int_D \omega(x,t) [u(x,t)\cdot\nabla f(x)] dx
\end{equation}
for all $t\in(0,T)$, with  both sides continuous in $t$.
Also, weak solutions to \eqref{sqg}-\eqref{eq:velocity_law} which are of the form \eqref{1.55} and have $C^1$ boundaries $\partial\Omega_k(t)$ which move with some continuous velocity $v:\Rm^2\times(0,T)\to\Rm^2$ (in the sense of \eqref{1.3} with $v$ in place of $u$), do satisfy \eqref{1.3} with $u$ (hence they are patch solutions if those boundaries remain pairwise disjoint simple closed curves).
Finally, $|\Omega_k(t)|=|\Omega_k(0)|$ holds for each $k$ and $t\in[0,T)$.

4.  In the 2D Euler case $\alpha=0$, it is not difficult to show using standard results of Yudovich theory that if $\omega(x,0)=\omega_0(x)$ is as in Definition~\ref{D.1.1}, then there exists
a unique global weak solution $\omega$ to \eqref{sqg}, and it is of the form \eqref{1.55}, with $\partial\Omega_{k}(t) = \Phi_t(\partial\Omega_{k}(0)).$ 
Remark 2 then shows that if the patch boundaries remain disjoint simple closed curves,
$\omega$ is also a patch solution to  \eqref{sqg}-\eqref{eq:velocity_law} on  $[0,\infty)$.
Moreover, $\omega$ must be unique in the class of $C^1$ patch solutions (if it belongs there) because these are also weak solutions.
In \cite{KRYZ2} we prove that the $C^{1,\gamma}$ patch solutions in the 2D Euler case are globally regular.
Therefore, since the Euler case is well-understood, we will only consider  $\alpha>0$ here.

Note that Definition \ref{D.1.1} automatically requires patch boundaries to not touch each other or themselves.  If this happens, the solution develops a singularity.
Also note that the definition allows for, e.g., $\Omega_2(t)\subseteq\Omega_1(t)$ and $\theta_2=-\theta_1$, and then $\sum_{k=1}^2 \theta_k \chi_{\Omega_k(t)}$ represents a non-simply connected patch.
Finally, we will say that $\omega$ is a patch solution to \eqref{sqg}-\eqref{eq:velocity_law} on $[0,T]$ if it is a patch solution to \eqref{sqg}-\eqref{eq:velocity_law} on $[0,T')$ for some $T'>T$.

Before we turn to our main results, let us address the relationship of the flow maps  from \eqref{eq:alpha} to the patch solution definition \eqref{1.3}. Note that since $u$ is smooth away from $\partial\Omega$ (see Lemma \ref{lem:du_crude} below), $\Phi_t(x)$ remains unique at least until it hits $\partial\Omega$ (in the Euler case, $\Phi_t(x)$ is always unique because $u$ is log-Lipschitz), after which it still exists but need not be unique.   The following result shows, in particular, that for $\alpha<\frac 14$ and patch solutions with sufficiently smooth boundaries, this remains true for any $x\in \bar D\setminus\partial\Omega(0)$ until time $T$.

\begin{proposition} \label{P.1.3}
Let $\omega$ be as in the first paragraph of Definition \ref{D.1.1}. For $x\in \bar D\setminus \partial\Omega(0)$, let $t_x\in [0,T]$ be the maximal time such that the solution of \eqref{eq:alpha} with $u$ from \eqref{eq:velocity_law} satisfies $\Phi_t(x)\in \bar D\setminus \partial\Omega(t)$ for each $t\in[0,t_x)$.
Then we have the following.
%
\begin{enumerate}[(a)]
\item If $\alpha\in(0,\frac 14)$, $\gamma\in(\frac{2\alpha}{1-2\alpha},1]$, and $\omega$ is a $C^{1,\gamma}$ patch solution to \eqref{sqg}-\eqref{eq:velocity_law} on  $[0,T)$, then $t_x=T$ for each  $x\in \bar D\setminus\partial\Omega(0)$ and $\Phi_t:[\bar D\setminus \partial\Omega(0)]\to [\bar D\setminus \partial\Omega(t)]$ is a bijection for each $t\in[0,T)$.
\item If $\alpha\in(0,\frac 12)$, $t_x=T$ for each  $x\in \bar D\setminus\partial\Omega(0)$, and $\Phi_t:[\bar D\setminus \partial\Omega(0)]\to [\bar D\setminus \partial\Omega(t)]$ is a bijection for each $t\in[0,T)$, then $\omega$ is a patch solution to \eqref{sqg}-\eqref{eq:velocity_law} on  $[0,T)$.  Moreover,  $\Phi_t$ is measure preserving on $\bar D\setminus \partial\Omega(0)$ and it also  preserves the connected components of $\bar D\setminus \partial\Omega$. Finally,  we have
\begin{equation} \label{1.32}
\Phi_t(\partial\Omega_k(0)) = \partial\Omega_k(t)
\end{equation}
for each $k$ and $t\in[0,T)$, in the sense that any solution of \eqref{eq:alpha} with $x\in\partial\Omega_k(0)$ has $\Phi_t(x)\in \partial\Omega_k(t)$, as well as that for each  $y\in\partial\Omega_k(t)$, there is $x\in\partial\Omega_k(0)$ and a solution of \eqref{eq:alpha} such that $\Phi_t(x)=y$.
\end{enumerate}
\end{proposition}

{\it Remarks.}
1.  Since $H^3(\mathbb T)\subseteq C^{1,1}(\mathbb T)$, we see that when $\alpha< \frac 14$, this result applies to the $H^3$ patch solutions from our main result below.

2.  We do not know whether this result holds for $\gamma\le \frac{2\alpha}{1-2\alpha}$.

Let us call the initial data $\omega_0$ for the problem \eqref{sqg}-\eqref{eq:velocity_law} {\it patch-like} if
\[
\omega_0 = \sum_{k=1}^N \theta_k \chi_{\Omega_{0k}},
\]
with $\theta_1,\dots,\theta_N\in\Rm\setminus\{0\}$ and $\Omega_{01},\dots,\Omega_{0N}\subseteq D$ bounded open sets whose boundaries are pairwise disjoint simple closed curves.  That is, $\omega_0$ is as $\omega(\cdot,0)$ in Definition \ref{D.1.1}.  Notice also that if $\omega(\cdot,0) = \sum_{k=1}^{N'} \theta_k' \chi_{\Omega_k(0)}$ is as in Definition \ref{D.1.1} and  $\omega(\cdot,0)=\omega_0$, then $N'=N$, and (up to a permutation) $\theta_k'=\theta_k$ and $\Omega_k(0)=\Omega_{0k}$ for each $k$.

Here is our first main result, local existence and uniqueness of $H^3$ patch solutions to \eqref{sqg}-\eqref{eq:velocity_law} on the half-plane $D=\Rm\times\Rm^+$ for small $\alpha>0$.
  Recall that uniqueness for patch solutions with $\alpha>0$ was previously only proved within a special class of SQG patches on $\Rm^2$ with $C^\infty$ boundaries in \cite{Rodrigo}.

\begin{theorem}\label{T.1.1}
Let  $\alpha\in(0,\frac 1{24})$. Then for each $H^3$ patch-like initial data $\omega_0$,  there exists a unique local $H^3$ patch solution $\omega$ to \eqref{sqg}-\eqref{eq:velocity_law} with $\omega(\cdot,0)=\omega_0$.  Moreover, if the maximal time $T_\omega$ of existence of $\omega$ is finite, then at $T_\omega$ either two patch boundaries touch, or a patch boundary touches itself, or a patch boundary loses $H^3$ regularity.
\end{theorem}

{\it Remarks.}  1.  The last claim means that either $\partial\Omega_k(T_\omega)\cap \partial\Omega_i(T_\omega)\neq\emptyset$ for some $k\neq i$, or $\partial\Omega_k(T_\omega)$ is not a simple closed curve for some $k$, or  $\lim_{t\nearrow T_\omega} \|\Omega_k(t)\|_{H^3}=\infty$ for some $k$.  Note that the sets $\partial\Omega_k(T_\omega):=\lim_{t\nearrow T_\omega}\partial\Omega_k(t)$ (with the limit taken in Hausdorff distance) are well defined if $T_\omega<\infty$ because
$u$ is uniformly bounded (see \eqref{uLinfty}).
In fact, an argument from Lemma \ref{lem:onestep} yields $d_H(\partial\Omega(t),\partial\Omega(s))\le \|u\|_{L^\infty}|t-s|$ for $t,s\in[0,T_\omega)$.

2.  The last claim further justifies  our definition of $H^3$ patch solutions because it shows that a solution cannot stay regular up to (and including) the time $T_\omega$ but stop existing due to some artificial limitation stemming from the definition of solutions.

3. We show (see Corollary \ref{C.5.7}) that  $T_\omega$  is bounded below by a constant depending on $\alpha, N, \|\omega\|_{L^\infty}$, and the quantity $\vertiii{\{\Omega_{0k}\}_{k=1}^N}_{H^3}$ from Definition \ref{D.5.6} (the latter expresses how close the initial patch boundaries are to touching each other or themselves, and how large their $H^3$ norms are).

4. Note that Remark 3 after Definition \ref{D.1.1} shows that the above solution is also the unique weak solution to \eqref{sqg}-\eqref{eq:velocity_law} from the class of functions which are of the form \eqref{1.55} and have $H^3$ boundaries $\partial\Omega_k(t)$ which are disjoint simple closed curves and move with some continuous velocity $v:\Rm^2\times(0,T)\to\Rm^2$ (in the sense of \eqref{1.3} with $v$ in place of $u$).

5. The hypothesis $\alpha<\frac 1{24}$ may well be an artifact of the proof, as it only appears in the application of the technical Lemma \ref{cor:sobolevbis} in the existence part.  The rest of the proof applies to all $\alpha\in(0,\frac 14)$, so it is possible that the result extends to at least this range.

As we mentioned above, our method also works on the whole plane, where non-existence of a boundary allows us to treat all $\alpha\in(0,\frac 12)$.
In this case we again use Definition \ref{D.1.1} but with $D:=\Rm^2$, and the flow is given by
\begin{equation}
u(x, t) =  \int_{\Rm^2}  \frac{(x-y)^\perp}{|x-y|^{2+2\alpha}} \, \omega(y,t) dy
\label{eq:velocity_law_R}
\end{equation}
instead of \eqref{eq:velocity_law}.
Our second main result is the corresponding version of Theorem \ref{T.1.1}.

\begin{theorem} \label{T.1.7}
With $D:=\Rm^2$ and \eqref{eq:velocity_law} replaced by \eqref{eq:velocity_law_R}, Proposition \ref{P.1.3}  holds as stated and Theorem \ref{T.1.1} holds with $\alpha\in(0,\frac 12)$.
\end{theorem}

The paper is organized as follows.  In Section \ref{sec:lwp}
we derive a contour equation corresponding to the patch dynamics for \eqref{sqg}-\eqref{eq:velocity_law} on the half-plane $D=\Rm\times\Rm^+$ and prove that it is locally well-posed for all $\alpha\in(0,\frac 1{24})$.
The proof largely follows that in~\cite{g} for the whole plane, but the presence of the boundary $\partial D$ will require us to introduce several non-trivial new  arguments.  We therefore present it in detail.
In Section~\ref{sec:regularity}, we prove some auxiliary estimates on the fluid velocity and on the geometry of the boundaries of sufficiently regular patches.
These are then used in Section~\ref{sec:patch} to show that the contour solution in fact yields a unique patch solution to \eqref{sqg}-\eqref{eq:velocity_law}, and Theorem \ref{T.1.1} will follow.
Some more delicate estimates on the fluid velocity generated by sufficiently regular patches are obtained in Section \ref{sec:prop13}, and these are then used to prove Proposition~\ref{P.1.3}.  We conclude Section \ref{sec:prop13} and the paper with the proof of Theorem \ref{T.1.7}.

\begin{proof}[Proof of Remark 3 after Definition \ref{D.1.1}]
Since $\nabla\cdot u=0$ in $\bar D\setminus \partial\Omega(t)$ and $u$ is continuous, the right-hand side of \eqref{1.6} is
\[
\sum_{k=1}^N \int_{\partial \Omega_k(t)} \theta_k [u(x,t)\cdot n_{x,t}] f(x) d\sigma(x).
\]
This equals the left-hand side of \eqref{1.6}, which can be seen by noticing that the area of the rectangle with vertices $x,x',x'+hu(x',t),x+hu(x,t)$ is $|x-x'|[u(x,t)\cdot n_{x,t}]h+o(|h||x-x'|)$ if $x,x'\in\partial\Omega_k(t)$ (because $u$ is continuous), and then taking $x,x'$ to be successive points on a progressively finer mesh of $\partial\Omega_k(t)$ (as well as letting $h\to 0$).
Finally, continuity of the right-hand side of \eqref{1.6} in time follows from continuity of $\partial\Omega_k$ in time in Hausdorff distance and from boundedness of $u\cdot \nabla f$, which is due to $f\in C^1(\bar D)$ and \eqref{uLinfty} below.

Next, if a weak solution  of the form \eqref{1.55} satisfies \eqref{1.3} with some continuous $v$ in place of $u$, then the above computation applied to smooth characteristic functions $f$ of successively smaller squares centered at any fixed $x\in\partial\Omega(t)$ yields
\[
|x-x'|[v(x,t)\cdot n_{x,t}]h+o(|h||x-x'|) = |x-x'|[u(x,t)\cdot n_{x,t}]h+o(|h||x-x'|).
\]
Hence $v(x,t)\cdot n_{x,t}=v(x,t)\cdot n_{x,t}$, so $\partial\Omega(t)$ being $C^1$ shows that \eqref{1.3} holds with $u$ as well.

Now fix any $\tau\in [0,T)$ and let $f$ in \eqref{1.6} be 1 on some open neighborhood of $\partial\Omega_k(\tau)$ and with its support having positive distance from $\partial\Omega(\tau)\setminus\partial\Omega_k(\tau)$.  Then continuity of $\partial\Omega$ in $t$ and $\nabla\cdot u=0$ show that for all $t$ close to $\tau$, the right-hand side of \eqref{1.6} equals
\[
\int_{\Omega_k(t)} \theta_k [u(x,t)\cdot\nabla f(x)]dx = \theta_k\int_{\partial \Omega_k(t)}  [u(x,t)\cdot n_{x,t}]d\sigma(x) = 0.
\]
  Thus $ \int_{\Omega_k(t)}f(x) dx$ is constant in all $t$ near $\tau$ because it equals $\theta_k^{-1}\int_D \omega(x,t)f(x)dx$.
  Since also $\int_{\Omega_k(t)} (1-f(x)) dx$ is constant in all $t$ near $\tau$ (because the support of $1-f(x)$ has positive distance from $\partial\Omega_k(\tau)$),
  we obtain that $|\Omega_k(t)|$ is constant on an open interval containing $\tau$.  This holds for all $\tau\in[0,T)$, hence $|\Omega_k(t)|$ is constant on $[0,T)$.

We note that for the 2D Euler equation, the definition of weak solutions via \eqref{1.6} can be found, for instance, in \cite[Theorem 3.2]{MP}, and that it easily implies
\[
\int_D \omega (x,T')g(x,T')dx - \int_D \omega (x,0)g(x,0)dx = \int_{D\times(0,T')} \omega(x,t)[\partial_t g(x,t)+u(x,t)\cdot\nabla g(x,t)]dxdt
\]
for all $T'\in[0,T)$ and  $g\in C^1(\bar D\times [0,T'])$.
\end{proof}

{\bf Acknowledgment.} We thank Peter Guba, Giovanni Russo and Lenya Ryzhik for useful discussions.
We acknowledge partial support by 
NSF-DMS grants  1056327, 1159133, 1411857, 1412023, and 1453199.

 \section{Local regularity for a contour equation and small~$\alpha$}
\label{sec:lwp}

This section is the first step towards the proof of Theorem \ref{T.1.1}.  We will derive a PDE whose solutions are time-dependent $H^3$ curves on the half-plane $\bar D=\Rm\times\Rm^+_0$, and one expects the latter to be some parametrizations of the patch boundaries $\partial\Omega_k(t)$.  We will then prove local regularity for this {\it contour equation} in the main result of this section, Theorem~\ref{thm:local} (which is the half-plane analog of its whole plane version from \cite{g}).  We will later show in Section~\ref{sec:patch}, using some crucial estimates derived in Section~\ref{sec:regularity}, that the solutions of the contour equation indeed yield $H^3$ patch solutions to \eqref{sqg}-\eqref{eq:velocity_law}.



\subsection{Derivation of the contour equation}
\label{sec:derivation}

Let us first derive our contour equation. Assuming that we have an $H^3$ patch solution to \eqref{sqg}-\eqref{eq:velocity_law}, let us parametrize the patch boundaries by
\begin{equation}\label{apr102}
\partial \Omega_k(t) = \{\z_k(\xi, t) = (\z_k^1(\xi, t), \z_k^2(\xi, t)): \xi \in \mathbb{T}\} \subseteq \bar D,
\end{equation}
with each $\z_k(\cdot,t)$ running once counter-clockwise along $\partial \Omega_k(t)$.
We do this so that at $t=0$, the curves $\z_k(\cdot, 0)$ all belong
to $H^3(\mathbb{T})$, and are non-degenerate in the sense of the right-hand side of \eqref{F_def} below being finite for $t=0$.
Of course, even when all $z_k(\cdot,0)$ are given, the choice of $z_k(\cdot,t)$ is not unique.  Hence, we will have to be careful  when choosing our contour equation for the $z_k$.  While our choice is similar to the case of the whole plane in
\cite{g}, the boundary $\partial D$ creates some new terms, so we present the derivation below for the sake of completeness.


Let $x\in \partial \Omega_k(t)$ and let $n(x,t)$
denote the outer unit normal vector for $\Omega_k(t)$ at $x$. We use the Biot-Savart law
to compute the outer normal velocity at $x$ as follows:
\begin{equation}
\begin{split}
\label{derivation_v}
u_n(x,t)&=u(x,t) \cdot n(x,t) \\
&= -\sum_{i=1}^N \theta_i \int_{\Omega_i} \left[
\frac{(x-y)\cdot n(x,t)^{\perp}}{|x-y|^{2+2\alpha}} - \frac{(x-\bar{y})\cdot n(x,t)^{\perp}}{|x-\bar{y}|^{2+2\alpha}}\right]dy~~~
\text{(since $u^\perp\cdot v = -u \cdot v^\perp$)} \\
&= -\sum_{i=1}^N \theta_i \int_{\Omega_i} \left[\frac{(x-y)\cdot n(x,t)^{\perp}}{|x-y|^{2+2\alpha}} - \frac{(\bar{x}-{y})\cdot \overline{{n}(x,t)^{\perp}}}{|\bar{x}-{y}|^{2+2\alpha}}\right]dy~~~
\text{(since $u\cdot \bar{v} = \bar{u} \cdot v$)}\\
&=- \sum_{i=1}^N \frac{\theta_i}{2\alpha} \int_{\Omega_i} \nabla_{y}\cdot\left[\frac{n(x,t)^{\perp}}{|x-y|^{2\alpha}} + \frac{{\bar{n}(x,t)^{\perp}}}{|\bar {x}-{y}|^{2\alpha}}\right]dy ~~~ \text{(since $\bar n^{\perp} = - \overline{n^\perp}$)}\\
&= \sum_{i=1}^N \frac{\theta_i}{2\alpha} \int_{\partial \Omega_i}  \left[\frac{n(y,t)^{\perp}}{|x-y|^{2\alpha}} + \frac{\overline{n(y,t)^{\perp}}}{|{x}-\bar{y}|^{2\alpha}} \right]\cdot n(x,t) d\sigma(y).
\end{split}
\end{equation}
Using  (\ref{apr102}), we conclude that the normal
velocity at $x = \z_k(\xi, t) \in \partial \Omega_k(t)$ is
\begin{equation}
u_n(x,t) = -\sum_{i=1}^N \frac{\theta_i}{2\alpha}\int_{\mathbb{T}}
\left[\frac{\partial_\xi \z_i(\xi-\eta,t)}{|\z_k(\xi,t) -
    \z_i(\xi-\eta, t)|^{2\alpha}}
+ \frac{\partial_\xi \bar \z_i(\xi - \eta,t)}{|\z_k(\xi, t) -
  \bar \z_i(\xi-\eta, t)|^{2\alpha}}
\right]\cdot n(x,t) d\eta .
\label{eq:contour_u}
\end{equation}
Intuitively, one can add any multiple of the tangent vector $\partial_\xi \z_k(\xi,t)$ to the velocity without changing the evolution of the patch itself (this does affect the particular parametrization $\z_k$, though). Hence, we
will use as the {\it contour equation} for  $\partial\Omega_k(t)$ the equation
\begin{equation}
\partial_t \z_k(\xi,t) =  \sum_{i=1}^N \frac{\theta_i}{2\alpha} \int_{\mathbb{T}} \left[
\frac{\partial_\xi \z_k(\xi,t) - \partial_\xi \z_i(\xi -
  \eta, t)}{|\z_k(\xi,t)
- \z_i(\xi-\eta, t)|^{2\alpha}} + \frac{\partial_\xi
\z_k(\xi,t)
- \partial_\xi \bar \z_i(\xi - \eta, t)}{|\z_k(\xi,t)
- \bar \z_i(\xi-\eta, t)|^{2\alpha}}\right] d\eta.
\label{eq:contour}
\end{equation}
(This particular choice of the tangential component of $\partial_t z_k$ will allow us to derive \eqref{eq:third_derivative} below.)
To simplify the notation, we let $y_i^1 := \z_i$ and $y_i^2 := \bar
\z_i$, so \eqref{eq:contour} becomes
\begin{equation}
\partial_t \z_k(\xi,t) = \sum_{i=1}^N \sum_{m=1}^2 \frac{\theta_i}{2\alpha}
\int_{\mathbb{T}}
 \frac{\partial_\xi \z_k(\xi,t)
- \partial_\xi y_i^m(\xi - \eta, t)}{|\z_k(\xi,t)  - y_i^m(\xi-\eta, t)|^{2\alpha}}  d\eta.
\label{eq:contour2}
\end{equation}
Note that while our $v^\perp$ is the negative of $v^\perp$ from \cite{g},
we have parametrized the curve~$\partial \Omega_i$ in the opposite
direction as well. Therefore our half-plane contour equation
\eqref{eq:contour2} is similar to that in the
whole space $D=\Rm^2$  \cite{g}, which however  only contains the $m=1$ terms.

\subsection{A priori estimates for the contour equation and small $\alpha$}
\label{sec:apriori}
Let us first define some norms and functionals  of the patch
boundaries $\Z(t) = \{\z_k(\cdot,t)\}_{k=1}^N$ which we will need in order to establish local well-posedness
for the contour equation:
\begin{eqnarray}
&&\|\Z(t)\|_{H^3}^2 := \sum_{k=1}^N \Big(\|\z_k(\cdot,t)\|_{L^\infty}^2 + \|\partial_\xi^3 \z_k(\cdot,t)\|_{L^2}^2 \Big).
\label{H3_def}
\\
&&\|\Z(t)\|_{C^2} :=\displaystyle\max_{1\leq k \leq N} \displaystyle\max_{0\le j\leq 2} \|\partial_\xi^j \z_k (\cdot,t)\|_{L^\infty},
\notag
\\
&&
\delta[\Z(t)] := \displaystyle \min \left\{\min_{i\neq k} \min_{\xi, \eta\in\mathbb{T}} |\z_i(\xi, t) - \z_k(\eta, t)|, 1 \right\},
\label{del_def}
\\
&&
F[\Z(t)] := \max \left\{ \max_{1\leq k\leq N}\,\sup_{\xi,\eta\in \mathbb{T}, \eta \neq 0}\dfrac{|\eta|}{|\z_k(\xi, t) - \z_k(\xi-\eta, t)|}, 1 \right\}. 
\label{F_def}
\end{eqnarray}
Note that the $H^3(\mathbb T)$-norm above is equivalent to the usual definition, where
$\|\z_k(\cdot,t)\|_{L^\infty}^2$ is replaced by~$\|\z_k(\cdot,t)\|_{L^2}^2$. We also let $\delta[\Z(t)]=1$ if $N=1$.  Finally, let
us define
\begin{equation}\vertiii{\Z(t)} := \|\Z( t)\|_{H^3} +  \delta[\Z(t)]^{-1} + F[\Z(t)].
\label{eq:def_norm}
\end{equation}
Note that $\vertiii{\cdot}$ is not a norm, but this will not affect our arguments.
Our goal is to obtain an a priori control on the growth of
$\vertiii{\Z(t)}$ for smooth solutions. We will show that if $\alpha\in (0,\tfrac 1{24})$ and  $\Z( t)$ solves \eqref{eq:contour2} with
$\vertiii{\Z(0)}<\infty$, then $\vertiii{\Z(t)}$ will remain finite for a
short time.  This follows from the main result here, the estimate \eqref{eq:apriori} below.  To prove it, we will now obtain bounds on the growth of the terms constituting $\vertiii{\Z(t)}$.

\subsubsection*{The evolution of  $\|\z_k(\cdot, t)\|_{L^\infty}$ and $\delta[\Z(t)]^{-1}$}

The evolution of these terms is controlled via the following lemma.  (A better bound, on the velocity $u$ rather than on $\partial_t z_k$, will be provided in Lemma \ref{lemma:uniform_u_bound} below.  However, since we will also need to work with regularizations of \eqref{eq:contour2}, for which we do not have \eqref{eq:velocity_law}, Lemma \ref{lemma:uniform_u_bound} will not be sufficient here.)
Let us denote $\Theta:=\sum_{k=1}^N |\theta_k|$.


\begin{lemma}\label{lemma:S}
For $\alpha\in(0,\tfrac 12)$ and $S_k[Z(t)](\xi)$ the right-hand side of \eqref{eq:contour2} we have
\[
\|S_k[Z(t)]\|_{L^\infty} \leq  \frac{4\pi  \Theta}{\alpha(1-2\alpha)}  \big( \delta[\Z(t)]^{-1}+F[\Z(t)] \big)^{2\alpha} \|Z(t)\|_{C^2}.
\]
\end{lemma}

\begin{proof}
In the integrands of $S_k[Z(t)](\xi)$, the numerators are bounded above by $2\|Z(t)\|_{C^2}$, and the denominators are bounded below by either $\delta[\Z(t)]^{2\alpha}$ or $F[\Z(t)]^{-2\alpha} |\eta|^{2\alpha}$. The claim now follows by a simple computation.
\end{proof}

Thus
\begin{equation} \label{3.11a}
\left|\frac d{dt}\max_k \|\z_k(\cdot, t)\|_{L^\infty}\right|  \le \frac{4\pi  \Theta}{\alpha(1-2\alpha)}  \big( \delta[\Z(t)]^{-1}+F[\Z(t)] \big)^{2\alpha} \|Z(t)\|_{C^2},
\end{equation}
 while $| \frac{d}{dt} \delta[\Z(t)]  |$ is bounded by twice that, so we have
\begin{equation} \label{3.51}
\left| \frac{d}{dt} \delta[Z(t)]^{-1} \right| \leq \frac{8\pi  \Theta}{\alpha(1-2\alpha)}  \left( \delta[Z(t)]^{-1}+F[Z(t)] \right)^{2+2\alpha} \|Z(t)\|_{C^2}.
\end{equation}

\subsubsection*{The evolution of $\|\partial_\xi^3 \z_k(\cdot, t)\|_{L^2}$}

In the following computation, let us assume that each $\z_k\in C^{4,1}(\mathbb{T}\times [0,T])$ for each $T<\infty$ (by which we mean that $\partial^4_{abcd} \z_k\in C(\mathbb{T}\times [0,T])$ whenever $a,b,c,d\in\{\xi,t\}$ and at most one of them is $t$).  This will be sufficient because we will eventually apply the obtained  estimates to a family of regularized solutions which do possess this regularity.  Then for $1\leq k\leq N$ we have
\begin{equation}
\frac{d}{dt}\|\partial_\xi^3 \z_k(\cdot, t)\|_{L^2}^2 =  \sum_{i=1}^N \sum_{m=1}^2 \frac{\theta_i}{\alpha} \int_{\mathbb{T}^2} \partial_\xi^3 \z_k(\xi,t) \cdot  \partial_\xi^3 
\left(\frac{\partial_\xi \z_k(\xi,t) - \partial_\xi y_i^m( \xi-\eta, t)}{|\z_k(\xi,t)  - y_i^m(\xi-\eta, t)|^{2\alpha}}
\right) d\eta d\xi.
\label{eq:third_derivative}
\end{equation}
  Here we used that
\[
\partial_\xi^3 \int_{\mathbb{T}}
\frac{\partial_\xi \z_k(\xi,t) - \partial_\xi y_i^m( \xi-\eta, t)}{|\z_k(\xi,t)  - y_i^m(\xi-\eta, t)|^{2\alpha}} d\eta =
 \int_{\mathbb{T}} \partial_\xi^3
 \left( \frac{\partial_\xi \z_k(\xi,t) - \partial_\xi y_i^m( \xi-\eta, t)}{|\z_k(\xi,t)  - y_i^m(\xi-\eta, t)|^{2\alpha}} \right) d\eta,
\]
which is obvious for $i\neq k$ as long as $\delta[\Z(t)]>0$, and for $i=k$ it follows from $z_k\in C^{4,1}$ and $2\alpha<1$ via a triple application of the Leibnitz integral rule.

We will analyze the right hand side of \eqref{eq:third_derivative} separately for $i\neq k$ and $i=k$. For the sake of clarity we will omit the $t$ dependence in the rest of this argument, since all the estimates of the right hand side of \eqref{eq:third_derivative} are done at a fixed time $t$.  We will also omit $Z$ in $\delta[\Z(t)]$ and $F[\Z(t)]$, instead writing just $\delta$ and $F$.

\noindent\textbf{Step 1: Contribution to \eqref{eq:third_derivative}  from the $i\neq k$ terms.}
For $i\neq k$ and $m=1,2$, the integral on the right hand side
of \eqref{eq:third_derivative} can be written as $\sum_{j=0}^3 \binom{3}{j} I_{k,i,j}^m$, where
\[
I_{k,i,j}^m :=  \int_{\mathbb{T}^2}
\partial_\xi^3 \z_k(\xi)\cdot
\underbrace{\partial_\xi^{3-j} \Big(\partial_\xi \z_k(\xi) - \partial_\xi y_i^m(\xi-\eta)\Big)}_{T_1(\xi,\eta)} \underbrace{\partial_\xi^{j} \Big( \big| \z_k(\xi) -  y_i^m (\xi-\eta)\big|^{-2\alpha}\Big)}_{T_2(\xi,\eta)} d\eta d\xi.
\]

The $j=2$ term is the easiest to control, where we directly use
\[
|T_1| \leq 2 \|\Z\|_{C^2} \le C \|\Z\|_{H^3} \qquad \text{and } \qquad |T_2| \leq C(\alpha) \delta^{-2-2\alpha} (\|\Z\|_{C^2}+1)^2
\]
to obtain
\[
|I_{k,i,2}^m| \leq C(\alpha)  \delta^{-2-2\alpha}  (\|\Z\|_{C^2}+1)^2 \|\Z\|_{H^3}^{2}.
\]

For $j=1$ we have a similar estimate for $T_2$, so
\begin{equation*}
\begin{split}
|I_{k,i,1}^m| &\le C(\alpha) \delta^{-2-2\alpha} (\|\Z\|_{C^2}+1)^2  \int_{\mathbb{T}^2} \left[  \left|\partial_\xi^3 z_k(\xi)\right|^2 + |\partial_\xi^3 z_k(\xi) \cdot \partial_\xi^3 y_i^m(\xi-\eta)| \right] d\eta d\xi\\
&\le  C(\alpha)\delta^{-2-2\alpha} (\|\Z\|_{C^2}+1)^2  \|Z\|_{H^3}^2.
\end{split}
\end{equation*}

For $j=3$ we have $|T_1| \leq 2\|\Z\|_{C^2}$ and
\[
|T_2| \leq C(\alpha)  \delta^{-1-2\alpha} |\partial_\xi^3 \z_k(\xi) -  \partial_\xi^3 y_i^m(\xi-\eta) |  + C(\alpha) \delta^{-3-2\alpha} (\|\Z\|_{C^2}+1)^3 ,
\]
so that we obtain (also using $\|\Z\|_{C^2} \le C \|\Z\|_{H^3}$)
\[
\begin{split}
|I_{k,i,3}^m| &\leq C(\alpha)   \|Z\|_{C^2} \left( \delta^{-1-2\alpha} \|Z\|_{H^3}^2 + \delta^{-3-2\alpha} (\|\Z\|_{C^2}+1)^3  \|Z\|_{H^3} \right)\\
&\leq C(\alpha) \delta^{-3-2\alpha}(\|\Z\|_{C^2}+1)^3  \|Z\|_{H^3}^2.
\end{split}
\]

For $j=0$, we split the integral as follows:
\begin{equation*}
\begin{split}
I_{k,i,0}^m &= \underbrace{\int_{\mathbb{T}^2} \partial_\xi^3 \z_k(\xi) \cdot
\frac{\partial_\xi^4 \z_k(\xi) }{|\z_k(\xi)  -
y_i^m(\xi-\eta)|^{2\alpha}}d\eta d\xi }_{I_{01}}
- \underbrace{\int_{\mathbb{T}^2} \partial_\xi^3 \z_k(\xi) \cdot
\frac{ \partial_\xi^4  y_i^m(\xi-\eta)}{|\z_k(\xi)  -
y_i^m(\xi-\eta)|^{2\alpha}}d\eta d\xi}_{I_{02}}.
\end{split}
\end{equation*}
Integration by parts in $\xi$ yields
\begin{equation}
\label{eq:I01}
\begin{split}
|I_{01}| &= \frac{1}{2} \left| \int_{\mathbb{T}^2} \left|\partial_\xi^3 \z_k(\xi)\right|^2 \partial_\xi \left(|\z_k(\xi)  -
y_i^m(\xi-\eta)|^{-2\alpha}\right) d\eta d\xi \right|
\leq C(\alpha) \delta^{-1-2\alpha} \|\Z\|_{C_2} \|\Z\|_{H^3}^2.
\end{split}
\end{equation}
Since $\partial_\xi^4  y_i^m(\xi-\eta)=\frac{d^4}{d\eta^4} y_i^m(\xi-\eta)$,  integration by parts in $\eta$ also yields
\begin{equation}
\label{eq:I02}
\begin{split}
|I_{02}| &= \left|\int_{\mathbb{T}^2} \partial_\xi^3 \z_k(\xi) \cdot
\frac{ \frac{d^4}{d\eta^4}  y_i^m(\xi-\eta)}{|\z_k(\xi)  -
y_i^m(\xi-\eta)|^{2\alpha}}d\eta d\xi \right|\\
&=\left|  \int_{\mathbb{T}^2} \partial_\xi^3 \z_k(\xi) \cdot \frac{d^3}{d\eta^3}  y_i^m(\xi-\eta) \frac{d}{d\eta} \left(|\z_k(\xi)  -
y_i^m(\xi-\eta)|^{-2\alpha}\right)d\eta d\xi\right|\\
&\leq C(\alpha) \delta^{-1-2\alpha} \|\Z\|_{C_2} \|\Z\|_{H^3}^2.
\end{split}
\end{equation}

Thus for $i\neq k$, the integral on the right-hand side of \eqref{eq:third_derivative}  is bounded by
\[
C(\alpha) \delta^{-3-2\alpha} (\|\Z\|_{C^2}+1)^3 \|\Z\|_{H^3}^2.
\]

\noindent \textbf{Step 2: Contribution to \eqref{eq:third_derivative}  from the $i=k$ terms.}
An argument as in \cite{g} (see the
bound just below (24) in \cite{g}) shows that the
integral on the right-hand side of \eqref{eq:third_derivative} with $i=k$ and $m=1$ (that is, $y_i^m=\z_k$), is bounded by
\[
C(\alpha) F^{3+2\alpha}  (\|\Z\|_{C^2}+1)^{3} \|\Z\|_{H^3}^2.
\]
However, the term with $i=k$ and $m=2$ (that is, $y_i^m = \bar{\z}_k$), creates some
new difficulties.  Nevertheless, we will be able to obtain for it the almost identical bound \eqref{3.1t} below. (Also, as the reader can easily check,  the argument in the case $m=1$ is essentially a subset of the argument below for $m=2$.)

Using the notation from \cite{g} and writing $z=z_k$, the integral in \eqref{eq:third_derivative} with $i=k$ and $m=2$ becomes $I_0+3I_1+3I_2 + I_3$,
where
\[
I_j := I_{k,k,j}^2 = \int_{\mathbb{T}^2} \partial_\xi^3 \z(\xi)
\cdot  \partial_\xi^{3-j}
\left(\partial_\xi \z(\xi) - \partial_\xi \bar{\z}(\xi-\eta)\Big)
\partial_\xi^{j} \Big(|\z(\xi)  - \bar{\z}(\xi-\eta)|^{-2\alpha}\right) d\eta d\xi.
\]
For $j=0$ we have using $u\cdot v = \bar u\cdot \bar v$ and a change of variables,
\begin{equation*}
\begin{split}
I_0 =  \int_{\mathbb{T}^2} \partial_\xi^3 \z(\xi) \cdot
\frac{\partial_\xi^4 \z(\xi) -
\partial_\xi^4 \bar{\z}(\xi-\eta)}{|\z(\xi)  -
\bar{\z}(\xi-\eta)|^{2\alpha}}d\eta d\xi
&=\int_{\mathbb{T}^2} \partial_\xi^3\bar \z(\xi) \cdot
\frac{\partial_\xi^4 \bar \z(\xi) -
\partial_\xi^4 {\z}(\xi-\eta)}{|\bar \z(\xi)  - {\z}(\xi-\eta)|^{2\alpha}}d\eta d\xi\\
&= \int_{\mathbb{T}^2} \partial_\xi^3\bar \z(\xi-\eta) \cdot
\frac{\partial_\xi^4 \bar \z(\xi-\eta) -
\partial_\xi^4 {\z}(\xi)}{|\bar \z(\xi-\eta)  - {\z}(\xi)|^{2\alpha}}d\eta d\xi.
\end{split}
\end{equation*}
This, an integration by parts in $\xi$, and a change of variables now yield
\begin{equation}
\begin{split}
|I_0|
&= \frac{1}{2} \left| \int_{\mathbb{T}^2} \Big(\partial_\xi^3 \z(\xi) - \partial_\xi^3 \bar{\z}(\xi-\eta)\Big) \cdot \frac{\partial_\xi(\partial_\xi^3 \z(\xi) - \partial_\xi^3 \bar{\z}( \xi-\eta))}{|\z(\xi)  - \bar{\z}(\xi-\eta)|^{2\alpha}}d\eta d\xi\right|\\
&\leq \frac{2\alpha}{4}  \int_{\mathbb{T}^2}
\Big|\partial_\xi^3 \z(\xi) - \partial_\xi^3 \bar{\z}(\xi-\eta)\Big|^2 \frac{|\partial_\xi \z(\xi) - \partial_\xi \bar{\z}( \xi-\eta)|}{|\z(\xi)  - \bar{\z}(\xi-\eta)|^{1+2\alpha}}d\eta d\xi\\
&\leq \alpha  \int_{\mathbb{T}^2}
\left(\left|\partial_\xi^3 \z(\xi)\right|^2 +\left| \partial_\xi^3 \bar{\z}(\eta)\right|^2\right)  \underbrace{\frac{|\partial_\xi \z(\xi) - \partial_\xi \bar{\z}(\eta)|}{|\z(\xi)  - \bar{\z}(\eta)|^{1+2\alpha}}}_{T_3(\xi,\eta)}d\eta d\xi\\[-0.5cm]
&\leq 2\alpha \|\Z\|_{H^3}^2 \max_{\xi\in\mathbb{T}} \int_{\mathbb{T}} T_3(\xi,\eta) d\eta.
\label{eq:temp_i0}
\end{split}
\end{equation}
The above computation is similar to that in \cite{g}, with the latter having $\z$ in place of $\bar \z$ (which is our case $m=1$).
In that case the numerator of $T_3$ is bounded above by
$\|\Z\|_{C^2} |\xi-\eta|$, and the denominator is bounded below by
$F^{-1-2\alpha}|\xi-\eta|^{1+2\alpha}$, giving $|T_3(\xi,\eta)|\leq  F^{1+2\alpha} \|\Z\|_{C^2} |\xi-\eta|^{-2\alpha}$.
Since $2\alpha<1$, this now yields a bound on $I_0$.

In our case $m=2$, the lower bound for the denominator continues to hold
because
\begin{equation}
|\z(\xi)-\bar{\z}(\eta)| \ge  |\z(\xi)-\z(\eta)| \ge F^{-1} |\xi-\eta|.
\label{eq:temp}
\end{equation}
However, we no longer have the same estimate for the numerator. With the notation $\z(\xi)=(\z^1(\xi), \z^2(\xi))$, the second component of the numerator becomes $\partial_\xi \z^2(\xi) + \partial_\xi \z^2(\eta)$, which need not converge to 0 as $\xi-\eta\to 0$.  The following lemma will help us instead.

\begin{lemma}\label{lemma:f'bis}
If $\gamma\in[0,1]$ and $0\le f\in C^{1,\gamma}(\mathbb{T})$, then for any $\xi\in\mathbb{T}$ we have
\begin{equation}
|f'(\xi)| \leq 2 \|f\|_{C^{1,\gamma}}^{1/(1+\gamma)}  f(\xi)^{\gamma/(1+\gamma)}.
\label{eq:f'bis}
\end{equation}
\end{lemma}

We present the proof in Section~\ref{sec:nonnegative}.  Note that the  power of $f(\xi)$ is sharp,
and that Sobolev embedding and
\eqref{eq:f'bis} show for $0\le f\in H^3(\mathbb{T})$, $\xi\in\mathbb{T}$, and a universal $C<\infty$,
\begin{equation}
|f'(\xi)| \leq C \|f\|_{H^3}^{1/(1+\gamma)} f(\xi)^{\gamma/(1+\gamma)}.
\label{eq:f'_sqrtfbis}
\end{equation}

Lemma \ref{lemma:f'bis} with $f(\eta)=\z^2(\eta)\ge 0$  (together with $|\xi-\eta|\le\pi$) now yields
\begin{equation}
\begin{split}
\label{eq:diff_derivative}
|\partial_\xi \z(\xi) - \partial_\xi \bar{\z}(\eta)|
&\leq  |\partial_\xi \z(\xi) - \partial_\xi \z(\eta)| +  |\partial_\xi \z(\eta) - \partial_\xi \bar\z(\eta)|\\
&= |\partial_\xi \z(\xi) - \partial_\xi \z(\eta)| + 2|\partial_\xi \z^2(\eta)|\\
&\leq \|Z\|_{C^2}\, |\xi-\eta| + 2\|Z\|_{C^2}^{1/2}  \sqrt{\z^2( \eta)}\\
&\leq C (\|Z\|_{C^2}+1) \left(\sqrt{|\xi-\eta|} +   \sqrt{|\z(\xi) - \bar\z( \eta)|}\right)\\
&\leq 2C (\|Z\|_{C^2}+1) F^{1/2} |\z(\xi) - \bar\z( \eta)|^{1/2}.
\end{split}
\end{equation}
Then
\begin{equation}
\label{eq:temp6}
\begin{split}
T_3 (\xi,\eta)&\leq 2C F^{1/2} (\|Z\|_{C^2}+1)  |\z(\xi) - \bar\z( \eta)|^{-2\alpha-1/2}\\
&\leq 2 C F^{1+2\alpha} (\|Z\|_{C^2}+1)  |\xi-\eta|^{-2\alpha-1/2},
\end{split}
\end{equation}
which is integrable in $\eta$ (uniformly in $\xi$) when $\alpha<\tfrac 14$.  Plugging this into \eqref{eq:temp_i0} yields
\[
|I_0| \leq C(\alpha) F^{1+2\alpha} (\|\Z\|_{C^2}+1) \|\Z\|_{H^3}^2
\]
for all $\alpha<\tfrac14$.

For $j=1$ we notice that
\[
\begin{split}
I_1 &= -2\alpha  \int_{\mathbb{T}^2} \partial_\xi^3 \z(\xi) \cdot \left(\partial_\xi^3 \z(\xi) - \partial_\xi^3 \bar{\z}(\xi-\eta)\right)  \frac{(\partial_\xi \z(\xi) - \partial_\xi \bar{\z}( \xi-\eta))\cdot (\z(\xi)  - \bar{\z}(\xi-\eta))}{|\z(\xi)  - \bar{\z}(\xi-\eta)|^{2+2\alpha}}d\eta d\xi\\
&= -2\alpha  \int_{\mathbb{T}^2} \partial_\xi^3 \bar{\z}(\xi-\eta) \cdot \left(\partial_\xi^3 \bar{\z}(\xi-\eta) - \partial_\xi^3 \z(\xi)\right)  \frac{(\partial_\xi \z(\xi) - \partial_\xi \bar{\z}( \xi-\eta))\cdot (\z(\xi)  - \bar{\z}(\xi-\eta))}{|\z(\xi)  - \bar{\z}(\xi-\eta)|^{2+2\alpha}}d\eta d\xi,
\end{split}
\]
where we used a change of variables (switching $\xi$ and $\xi-\eta$) and $\bar u\cdot \bar v=u\cdot v$.  Thus
\[
I_1=-\alpha  \int_{\mathbb{T}^2} \left|\partial_\xi^3 \z(\xi) - \partial_\xi^3 \bar{\z}(\xi-\eta)\right|^2  \frac{(\partial_\xi \z(\xi) - \partial_\xi \bar{\z}( \xi-\eta))\cdot (\z(\xi)  - \bar{\z}(\xi-\eta))}{|\z(\xi)  - \bar{\z}(\xi-\eta)|^{2+2\alpha}}d\eta d\xi.
\]
So $|I_1|$ is bounded by twice the second line of \eqref{eq:temp_i0}, and it obeys the same bound as $|I_0|$.

The estimates for $I_2$ and $I_3$ will be slightly more involved.
For $j=2$ we have
\begin{equation}
\begin{split}
|I_2| &=\left| \int_{\mathbb{T}^2} \partial_\xi^3 \z(\xi) \cdot
\big(\partial_\xi^2 \z(\xi) - \partial_\xi^2 \bar{\z}( \xi-\eta)\big) \partial_\xi^{2} \big(|\z(\xi)  - \bar{\z}(\xi-\eta)|^{-2\alpha}\big) d\eta d\xi\right|
\\
&\leq 2\|\Z\|_{C^2} \, \|\Z\|_{H^3}\, \Big\| \underbrace{\int_{\mathbb{T}}  \big|\partial_\xi^{2} \big(|\z(\xi)  - \bar{\z}(\xi-\eta)|^{-2\alpha}\big) \big| d\eta}_{T_4(\xi)} \Big\|_{L^2},
\label{eq:I_3}
\end{split}
\end{equation}
so we need to bound $\|T_4\|_{L^2}$.  A simple change of variables and $\bar z=(z^1,-z^2)$ yield
\begin{equation}
\begin{split}
|T_4| &\leq C(\alpha) \bigg[ \int_{\mathbb{T}}\frac{|\partial_\xi \z(\xi) - \partial_\xi \bar \z(\eta)|^2}{|\z(\xi)-\bar{\z}(\eta)|^{2+2\alpha}} d\eta
+  \int_{\mathbb{T}}\frac{|\partial_\xi^2 \z(\xi) - \partial_\xi^2 \bar\z(\eta)|}{|\z(\xi)-\bar{\z}(\eta)|^{1+2\alpha}} d\eta\bigg]
\\[0.2cm]
&\leq C(\alpha) \bigg[ \int_{\mathbb{T}}\frac{|\partial_\xi \z(\xi) - \partial_\xi \z(\eta)|^2}{|\z(\xi)-\bar{\z}(\eta)|^{2+2\alpha}} d\eta + \underbrace{\int_{\mathbb{T}}\frac{|\partial_\xi \z^2(\eta)|^2}{|\z(\xi)-\bar{\z}(\eta)|^{2+2\alpha}} d\eta}_{T_5(\xi)}\\
&\quad+ \int_{\mathbb{T}}\frac{|\partial_\xi^2 \z(\xi) - \partial_\xi^2 \z(\eta)|}{|\z(\xi)-\bar{\z}(\eta)|^{1+2\alpha}} d\eta + \underbrace{\int_{\mathbb{T}} \frac{|\partial_\xi^2 \z^2(\eta)|}{|\z(\xi)-\bar{\z}(\eta)|^{1+2\alpha}}  d\eta}_{T_6(\xi)}\bigg].
\label{eq:T_4_breakdown}
\end{split}
\end{equation}
The first and third of the last four integrals can be controlled in the same
way as in \cite{g}.  Indeed, the numerator in the first term is bounded by $\|\Z\|_{C^2}^2\,|\xi-\eta|^2$, so the integral  is bounded uniformly in $\xi$ by $C(\alpha)F^{2+2\alpha} \|\Z\|_{C^2}^2$ due to $2\alpha<1$ (its $L^2$-norm, as a function of $\xi$, then satisfies the same bound).  As for the third term, let us change the $\eta$ variable back to $\xi-\eta$, so that the numerator equals $| \eta \int_0^1 \partial_\xi^3 z(\xi - s\eta) ds|$.  The Minkowski inequality for integrals then shows that that term's $L^2$-norm is bounded by
\[
  \left\| F^{1+2\alpha} \int_{\mathbb{T}\times[0,1]}  |\partial_\xi^3 \z(\xi - s\eta)|  \frac{ ds d\eta}{ |\eta|^{2\alpha} } \right\|_{L^2(\mathbb{T})}
 \leq F^{1+2\alpha}  \int_{\mathbb{T}\times[0,1]}  \left(\int_{\mathbb{T}} \left|\partial_\xi^3 \z(\xi - s\eta)\right|^2 d\xi \right)^{1/2}   \frac{ ds d\eta}{ |\eta|^{2\alpha} },
\]
that is, by $C(\alpha) F^{1+2\alpha} \|Z\|_{H^3}$ for  all $\alpha<\tfrac 12$.

To deal with $T_5$ and $T_6$, let us first define their regularizations
\[
T_5^\epsilon(\xi) := \int_{\mathbb{T}}\frac{|\partial_\xi \z^2(\eta)|^2}{(|\z(\xi)-\bar{\z}(\eta)|+\epsilon)^{2+2\alpha}} d\eta
\qquad\text{and}\qquad T_6^\epsilon(\xi) := \int_{\mathbb{T}} \frac{|\partial_\xi^2 \z^2(\eta)|}{(|\z(\xi)-\bar{\z}(\eta)|+\epsilon)^{1+2\alpha}}  d\eta.
\]
We will show that  $\|T_5^\epsilon\|_{L^2}$ and
$\|T_6^\epsilon\|_{L^2}$ are uniformly bounded as $0<\epsilon\to 0$, so that  the
monotone convergence theorem then yields the same bounds for
$\|T_5\|_{L^2}$ and $\|T_6\|_{L^2}$.

Let us start with $T_5^\epsilon$.  From \eqref{eq:temp} and $|\z(\xi)-\bar{\z}(\eta)|\geq \z^2(\eta)\ge 0$ we have
\begin{equation*}
T_5^\epsilon(\xi) \leq C(\alpha) \int_{\mathbb{T}} \left(F^{-1}|\xi - \eta| \right)^{-\frac{11}{12}-2\alpha} \underbrace{\frac{|\partial_\xi \z^2(\eta)|^2}{(\z^2(\eta)+\epsilon)^{\frac{13}{12}}} }_{T_7^\epsilon(\eta)} d\eta
\end{equation*}
Young's inequality for convolutions now yields
\[
\|T_5^\epsilon\|_{L^2} \leq C(\alpha)  F^{\frac{11}{12}+2\alpha} \left\|\xi^{-\frac{11}{12}-2\alpha}\right\|_{L^1} \|T_7^\epsilon\|_{L^2},
\]
with the $L^1$-norm finite provided $\alpha<\frac{1}{24}$.
The following will help us estimate $\|T_7^\epsilon\|_{L^2}$.

\begin{lemma}\label{cor:sobolevbis}
There is $C<\infty$ such that if $\beta\in[0,\frac{1}{6}]$ and $0< f\in H^3(\mathbb{T})$, then
\begin{equation}
\int_{\mathbb{T}} \frac{|f'(\xi)|^n}{f(\xi)^{ \beta+n/2}} d\xi \leq C^n \|f\|_{H^3}^{\tfrac{n}{2}-\beta}
\label{eq:sobolev_1bis}
\end{equation}
for any $n\ge 1$, as well as
\begin{equation}
\int_{\mathbb{T}} \frac{f''(\xi)^2}{f(\xi)^{\beta}} d\xi \leq C \|f\|_{H^3}^{2-\beta}.
\label{eq:sobolev_2bis}
\end{equation}
\end{lemma}

The proof is also presented in Section~\ref{sec:nonnegative}.
Now \eqref{eq:sobolev_1bis} with $n=4$ and $\beta=\tfrac 16$ yields
\[
\|T_7^\epsilon\|_{L^2}^2 = \int \frac{|\partial_\xi
  \z^2(\xi)|^4}{(\z^2(\xi)+\epsilon)^{\frac{13}{6}}} d\xi \leq
C^4 \|\Z\|_{H^3}^{\frac{11}{6}},
\]
so that
\[
\|T_5^\epsilon\|_{L^2} \leq C(\alpha)  F^{\frac{11}{12}+2\alpha} \|\Z\|_{H^3}^{\frac{11}{12}}
\]
for all $\alpha\in(0,\tfrac 1{24})$ and $\epsilon>0$.  Thus the same bound holds for $\|T_5\|_{L^2}$.

An almost identical argument for $T_6^\epsilon$ gives
\[
\|T_6^\epsilon\|_{L^2} \leq C(\alpha)  F^{\frac{11}{12}+2\alpha} \|\xi^{-\frac{11}{12}-2\alpha}\|_{L^1} \|T_8^\epsilon\|_{L^2},
\]
with $T_8^\epsilon(\xi):=|\partial_\xi^2 \z^2(\xi)|(\z^2(\xi)+\epsilon)^{-1/12}$.
From \eqref{eq:sobolev_2bis} we again obtain
\[
\|T_8^\epsilon\|_{L^2}^2 \leq C \|\Z\|_{H^3}^{\frac{11}{6}},
\]
which yields  for $\|T_6\|_{L^2}$ the same estimate as for $\|T_5\|_{L^2}$.  We therefore have
\begin{equation} \label{3.1s}
\|T_4\|_{L^2} \le  C(\alpha) F^{2+2\alpha} (\|Z\|_{C^2}+1)(\|Z\|_{H^3}+1)
\end{equation}
(also using $\|Z\|_{C^2}\le C\|Z\|_{H^3}$), so that \eqref{eq:I_3} finally yields for all $\alpha<\tfrac 1{24}$,
\[
|I_2| \leq C(\alpha)  F^{2+2\alpha}(\|\Z\|_{C^2}+1)^2(\|\Z\|_{H^3}+1) \|\Z\|_{H^3}.
\]

Finally, for $j=3$ we obtain after differentiating inside $I_3$ and changing variables,
\begin{equation}
\begin{split}
|I_3| \leq C(\alpha)\Big[ &\int_{\mathbb{T}^2}
|\partial_\xi \z(\xi) - \partial_\xi \bar{\z}(\eta)|  \frac{\left|\partial_\xi^3 \z(\xi)\right|^2 +\left| \partial_\xi^3 \bar{\z}(\eta)\right|^2}{|\z(\xi)  - \bar{\z}(\eta)|^{1+2\alpha}}d\eta d\xi\\
&+\int_{\mathbb{T}^2} |\partial_\xi^3 \z(\xi)| \,
\frac{|\partial_\xi^2 \z(\xi) - \partial_\xi^2 \bar{\z}(\eta)| \,|\partial_\xi \z(\xi) - \partial_\xi \bar{\z}(\eta)|^2}{|\z(\xi)  - \bar{\z}(\eta)|^{2+2\alpha}}d\eta d\xi\\
&+ \int_{\mathbb{T}^2} |\partial_\xi^3 \z(\xi) | \,
\frac{|\partial_\xi \z(\xi) - \partial_\xi \bar{\z}(\eta)|^4}{|\z(\xi)  - \bar{\z}(\eta)|^{3+2\alpha}}d\eta d\xi \Big].
\end{split}
\end{equation}
The first integral already appeared in \eqref{eq:temp_i0} and hence obeys the same bound as $I_0$. We then apply \eqref{eq:diff_derivative} squared to each of the other two integrals (removing $|\partial_\xi \z(\xi) - \partial_\xi \bar{\z}(\eta)|^2$ from the numerator and $|\z(\xi)-\bar \z(\eta)|$ from the denominator)  and find out that they are bounded by $\|\Z\|_{H^3}F(\|\Z\|_{C^2}+1)^2$ times the $L^2$-norm of the expression in the middle of \eqref{eq:T_4_breakdown}.  That latter norm is bounded by the right-hand side of \eqref{3.1s}, so that in the end we obtain
\[
|I_3| \leq C(\alpha) F^{3+2\alpha}(\|\Z\|_{C^2}+1)^3 (\|\Z\|_{H^3}+1) \|\Z\|_{H^3}
\]
for all $\alpha<\tfrac 1{24}$.

Thus for $i= k$, the integral on the right-hand side of \eqref{eq:third_derivative}  is bounded by
\begin{equation} \label{3.1t}
C(\alpha)  F^{3+2\alpha}(\|\Z\|_{C^2}+1)^3(\|\Z\|_{H^3}+1) \|\Z\|_{H^3}.
\end{equation}
This, the analogous estimate for $i=k$ and $m=1$ from the beginning of Step 2, the estimate from Step 1, and Lemma \ref{lemma:S}  now yield for any $\alpha\in(0,\tfrac 1{24})$ and $\Theta:=\sum_{k=1}^N |\theta_k|$,
\begin{equation}
\frac{d}{dt}\| \Z(t)\|_{H^3} \leq C(\alpha)N\Theta \left( \delta[\Z(t)]^{-1}+F[\Z(t)] \right)^{3+2\alpha} (\|\Z(t)\|_{C^2}+1)^3 (\|\Z(t)\|_{H^3}+1).
\label{evol_h3}
\end{equation}

\subsubsection*{The evolution of $F[\Z(t)]$}

For any $k=1,\dots, N$ and any $\xi, \lambda \in \mathbb{T}$ with $ \lambda \neq 0$, we have (again dropping $t$ from $z_k$)
\begin{equation}
\label{eq:F_growth}
\frac{d}{dt} \frac{|\lambda|}{|z_k(\xi)-z_k(\xi-\lambda)|} \leq \frac{|\lambda| \left |\partial_t z_k(\xi) -\partial_t z_k(\xi-\lambda)\right |}{|z_k(\xi)-z_k(\xi-\lambda)|^2}
\leq  F[Z(t)]^2 \frac{\left |\partial_t z_k(\xi) -\partial_t z_k(\xi-\lambda)\right |}{|\lambda|}.
\end{equation}
Using \eqref{eq:contour2} and the mean value theorem, we can estimate
\begin{equation*}
\begin{split}
\left |\partial_t z_k(\xi)  -\partial_t z_k(\xi-\lambda)\right | & \leq  \sum_{i=1}^N \sum_{m=1}^2 \frac{|\theta_i|}{2\alpha}
\int_{\mathbb{T}} |\lambda| \sup_{\xi\in \mathbb{T}} \left\{ \left| \partial_\xi \left(
 \frac{\partial_\xi \z_k(\xi) - \partial_\xi y_i^m(\xi - \eta)}{|\z_k(\xi)  - y_i^m(\xi-\eta)|^{2\alpha}} \right) \right| \right\} d\eta\\
 & \hskip -26mm \leq  \sum_{i=1}^N \sum_{m=1}^2 \frac{|\theta_i|   }{2\alpha} |\lambda|
\int_{\mathbb{T}}    \sup_{\xi \in \mathbb{T}}
\left\{ \frac{2\|Z(t)\|_{C^2}}{|\z_k(\xi)  - y_i^m(\xi-\eta)|^{2\alpha}} +  \frac{2\alpha\left |\partial_\xi \z_k(\xi) - \partial_\xi y_i^m(\xi - \eta)\right|^2}{|\z_k(\xi)  - y_i^m(\xi-\eta)|^{1+2\alpha}} \right\} d\eta
 \\
 & \hskip -26mm \leq C(\alpha)\Theta   |\lambda| \left( \delta[Z(t)]^{-1}+F[Z(t)] \right)^{1+2\alpha} (\|Z(t)\|_{C^2}+1)^2  ,
 \end{split}
\end{equation*}
where in the last inequality we used \eqref{eq:diff_derivative} to control the last term  on the second line for $i=k$ and $m=2$. Plugging this into \eqref{eq:F_growth} now yields
\begin{equation}
\frac{d}{dt} F[Z(t)] \leq C(\alpha)\Theta  \left( \delta[Z(t)]^{-1}+F[Z(t)] \right)^{3+2\alpha} (\|Z(t)\|_{C^2}+1)^2.
\label{evol_F}
\end{equation}

Finally, this, \eqref{evol_h3}, \eqref{3.11a}, and \eqref{3.51} imply for $\alpha\in(0,\tfrac 1{24})$,
\begin{equation}
\frac{d}{dt}\vertiii{\Z(t)} \leq C(\alpha)N\Theta   \vertiii{\Z(t)}^{7+2\alpha}.
\label{eq:apriori}
\end{equation}

\subsection{\hbox{Local $H^3$ well-posedness for the contour equation and small $\alpha$}}
\label{sec:lwp_h3}
\subsubsection*{Uniqueness of solutions}

Let $W=(w_1,\dots,w_N)=\Z-\tilde \Z$ for classical solutions $\Z$ and $\tilde \Z$ to
\eqref{eq:contour2} on some time interval $[0,T]$, with $\sup_{t\in[0,T]}(\vertiii{Z(t)}+\vertiii{\tilde Z(t)})<\infty$ and $\Z(0)=\tilde\Z(0)$  (here we require that $\partial_t Z$ is continuous in $(\xi,t)$ for classical solutions). Then for any  $k=1,\dots, N$ and $t\in[0,T]$  we have (with the argument $t$ again dropped)
\begin{equation}
\label{eq:temp2}
\begin{split}
\frac d{dt}\|w_k\|_{L^2}^2 &=
2 \int_{\mathbb{T}} w_k(\xi) \cdot \partial_t w_k(\xi) d\xi \\
&=  \sum_{i=1}^N \sum_{m=1}^2  \frac{\theta_i}{\alpha} \int_{\mathbb{T}^2} w_k(\xi) \cdot \left( \frac{\partial_\xi z_k(\xi) - \partial_\xi y_i^m(\xi - \eta)}{| z_k(\xi) -  y_i^m(\xi - \eta)|^{2\alpha}} -  \frac{\partial_\xi \tilde{z}_k(\xi) - \partial_\xi \tilde{y}_i^m(\xi - \eta)}{| \tilde{z}_k(\xi) -  \tilde{y}_i^m(\xi - \eta)|^{2\alpha}}\right) d\eta d\xi.
\end{split}
\end{equation}
The last integral equals $A_{k,i}^m + B_{k,i}^m$, where with $w_i^1 := w_i$ and $w_i^2 := \bar w_i $,
\[
A_{k,i}^m := \int_{\mathbb{T}^2} w_k(\xi) \cdot (\partial_\xi z_k(\xi) - \partial_\xi y_i^m(\xi - \eta)) \left( \frac{1}{| z_k(\xi) -  y_i^m(\xi - \eta)|^{2\alpha}} -  \frac{1}{| \tilde{z}_k(\xi) -  \tilde{y}_i^m(\xi - \eta)|^{2\alpha}}\right) d\eta d\xi,
\]
\[
B_{k,i}^m := \int_{\mathbb{T}^2} w_k(\xi) \cdot    \frac{\partial_\xi w_k(\xi) - \partial_\xi w_i^m(\xi - \eta)}{| \tilde{z}_k(\xi) -  \tilde{y}_i^m(\xi - \eta)|^{2\alpha}} d\eta d\xi.
\]

Let us first estimate $A_{k,i}^m$.
When $i\neq k$, the term inside the parentheses is easily bounded by
$C(\alpha) \min\{\delta[Z],\delta[\tilde Z]\}^{-1-2\alpha} (|w_k(\xi)|+|w_i(\xi - \eta)|)$, so
\[
|A_{k,i}^m| \leq C(\alpha) (\vertiii{Z}+ \vertiii{\tilde Z})^{2+2\alpha} \|W\|_{L^2}^2.
\]
For $i=k$ and $m=1$, $A_{k,i}^m$ is controlled the same way as in \cite[p. 13-14]{g}, yielding
\[
|A_{k,k}^1| \leq C(\alpha) (\vertiii{Z}+ \vertiii{\tilde Z})^{2+2\alpha} \|W\|_{L^2}^2
\]
for $\alpha<\tfrac 12$.
Finally, for $i=k$ and $m=2$, the following almost identical computation, \eqref{eq:diff_derivative}, and
$|x^{2\alpha}-1| \leq |x-1|$ for  $x\ge 0$ yield the same bound  for $\alpha< \tfrac 14$:
\begin{equation}
\label{Akk2}
\begin{split}
|A_{k,k}^2|
&\leq  \vertiii{Z}^{\tfrac{3}{2}} \vertiii{\tilde Z}^{2\alpha} \int_{\mathbb{T}^2} |w_k(\xi)|  | z_k(\xi) -  \bar{z}_k(\xi - \eta)|^{1/2} \left| \left(\frac{| \tilde{z}_k(\xi) -  \bar{\tilde{z}}_k(\xi - \eta)|}{| z_k(\xi) -  \bar{z}_k(\xi - \eta)|}\right)^{2\alpha} - 1 \right| |\eta|^{-2\alpha}  d\eta d\xi\\
&\leq  \vertiii{Z}^{\tfrac{3}{2}} \vertiii{\tilde Z}^{2\alpha}  \int_{\mathbb{T}^2} |w_k(\xi)|  | z_k(\xi) -  \bar{z}_k(\xi - \eta)|^{1/2}  \left| \frac{| \tilde{z}_k(\xi) -  \bar{\tilde{z}}_k(\xi - \eta)|}{| z_k(\xi) -  \bar{z}_k(\xi - \eta)|} - 1 \right| |\eta|^{-2\alpha} d\eta d\xi\\
&\leq  \vertiii{Z}^{2} \vertiii{\tilde Z}^{2\alpha}  \int_{\mathbb{T}^2} |w_k(\xi)| \, \Big| | \tilde{z}_k(\xi) -  \bar{\tilde{z}}_k(\xi - \eta)| - | z_k(\xi) -  \bar{z}_k(\xi - \eta)| \Big| \, |\eta|^{-\frac 12 -2\alpha} d\eta d\xi\\
&\leq \vertiii{Z}^{2} \vertiii{\tilde Z}^{2\alpha}   \int_{\mathbb{T}^2} |w_k(\xi)|  \Big( |w_k(\xi)| + |w_k(\xi - \eta)| \Big)  |\eta|^{-\frac{1}{2}-2\alpha} d\eta d\xi\\
&\leq  C(\alpha) (\vertiii{Z}+ \vertiii{\tilde Z})^{2+2\alpha} \|W\|_{L^2}^2.
\end{split}
\end{equation}

Next we control $B_{k,i}^m$. When $i\neq k$, we split it into two integrals: with $\partial_\xi w_k(\xi)$ and with $\partial_\xi w_i^m(\xi - \eta)$, respectively.  After integrating these by parts in $\xi$ and in $\eta$, respectively (similarly to \eqref{eq:I01} and \eqref{eq:I02}),  we obtain
\[
|B_{k,i}^m| \leq C(\alpha) \delta[\tilde Z]^{-1-2\alpha} \|\tilde Z\|_{C_2} \|W\|_{L^2}^2 \leq C(\alpha) \vertiii{\tilde Z}^{2+2\alpha}  \|W\|_{L^2}^2.
\]
For $i=k$, we symmetrize the integrand similarly to the equation before \eqref{eq:temp_i0} and obtain
\[
\begin{split}
B_{k,k}^m &= \frac{1}{4} \int_{\mathbb{T}^2}   \frac{\partial_\xi (| w_k(\xi) - w_k^m(\xi - \eta)|^2) }{| \tilde{z}_k(\xi) -  \tilde{y}_k^m(\xi - \eta)|^{2\alpha}} d\eta d\xi.
\end{split}
\]
Integration by parts now yields
\[
\begin{split}
|B_{k,k}^m| &\le \frac{\alpha}{2} \int_{\mathbb{T}^2}  | w_k(\xi) - w_k^m(\xi - \eta)|^2 \frac{ | \partial_\xi \tilde{z}_k(\xi) - \partial_\xi \tilde{y}_k^m(\xi - \eta)| }{| \tilde{z}_k(\xi) -  \tilde{y}_k^m(\xi - \eta)|^{1+2\alpha}} d\eta d\xi \\
&\le \alpha \int_{\mathbb{T}^2}  \Big( | w_k(\xi)|^2 + |w_k^m(\eta)|^2 \Big) \underbrace{ \frac{ | \partial_\xi \tilde{z}_k(\xi) - \partial_\xi \tilde{y}_k^m(\eta)| }{| \tilde{z}_k(\xi) -  \tilde{y}_k^m(\eta)|^{1+2\alpha}}}_{\tilde T_k^m(\xi,\eta)} d\eta d\xi \\[-0.5cm]
&\leq 2\alpha \|W\|_{L^2}^2 \max_{\xi\in\mathbb{T}} \int_{\mathbb{T}} \tilde T_k^m(\xi,\eta) d\eta.
\end{split}
\]
The bounds on $T_3$  from the discussion after \eqref{eq:temp_i0} (for both $m=1,2$) equally apply to $\tilde T_k^m$ and yield for $m=1$ and $\alpha<\tfrac 12$, as well as for $m=2$ and $\alpha<\tfrac 14$,
\[
|B_{k,k}^m| \leq C(\alpha) \vertiii{\tilde Z}^{2+2\alpha} \|W\|_{L^2}^2.
\]

Combining all the obtained bounds on $A_{k,i}^m$ and $B_{k,i}^m$ now yields
\begin{equation}
\label{eq:l2diff}
\frac{d}{dt}\|W( t)\|_{L^2}^2 \leq  C(\alpha)N\Theta (\vertiii {Z(t)}+\vertiii {\tilde \Z (t)})^{2+2\alpha} \|W( t)\|_{L^2}^2
\end{equation}
for $\alpha<\tfrac 14$.  Gronwall's inequality then shows $\|W(t)\|_{L^2}=0$ for $t\in[0,T]$, hence $Z=\tilde Z$.

\subsubsection*{Existence of  solutions}

Similarly to \cite[Chapter 3]{mb}, once we have the a priori control \eqref{eq:apriori} on the growth of $\vertiii{\Z(t)}$,  solutions to the contour equation \eqref {eq:contour2} will be obtained as limits of solutions to an appropriate family of mollified equations.  We will need to be careful, however, that the solutions of the latter do not exit $D$.

Consider any initial condition $Z_0 = \{\z_{0k}\}_{k=1}^N$ with $z_{0k}:\mathbb{T}\to \bar D$ for all $k=1,\dots, N$ and $M:=\vertiii{\Z_0}<\infty$ (then also $M\ge 2$). We let $\phi_\eps(\xi):=\eps^{-1}\phi(\eps^{-1}\xi)$ for some mollifier $\phi$ which is smooth, even, non-negative on $\mathbb T$, supported in $[-1,1]$, and satisfies $\int_{\mathbb T} \phi(\xi)d\xi=1$. For  $k=1,\dots, N$, we regularize \eqref{eq:contour2} to
\begin{equation}
\partial_t \z^\epsilon_k(\xi,t) = \sum_{i=1}^N \sum_{m=1}^2 \frac{\theta_i}{2\alpha}\,  \phi_\epsilon * \int_{\mathbb{T}}  \frac{\partial_\xi   (\phi_\epsilon * \z^\epsilon_k)(\xi,t) -  \partial_\xi(\phi_\epsilon * y_i^{m,\epsilon})(\xi - \eta, t)}{|(\phi_\epsilon*\z^\epsilon_k)(\xi,t)  - (\phi_\epsilon*y_i^{m,\epsilon})(\xi-\eta, t)|^{2\alpha}}  d\eta +  \tilde C(\alpha)\Theta \epsilon M^{3+2\alpha} e_2,
\label{eq:contour2_regular}
\end{equation}
with $e_2 := (0,1)$, a large constant $\tilde C(\alpha)>0$  to  be chosen later, and initial condition
\[
Z^\eps(0)= \{ z_k^\eps(\cdot,0)+\eps e_2\}_{k=1}^N  = \{ \phi_\eps* z_{0k}+\eps e_2\}_{k=1}^N = \phi_\epsilon *Z_0 +\eps e_2.
\]
The convolutions are all taken in the first variable only,
and the last term in \eqref{eq:contour2_regular} will ensure containment in $D$.

\noindent \textbf{Step 1.}  We now prepare the setup for an application of the Picard theorem to find a solution of \eqref{eq:contour2_regular}.  Consider the Banach space $B:= H^3(\mathbb{T})^N$ with the norm $\|\cdot \|_{H^3}$ and let $h[Z]:=\inf_{1\le k\le N\,\&\,\xi\in\mathbb{T}} z_k^2(\xi)$ be the infimum of the $x_2$-coordinates for $Z\in B$. Then
\[
O^A := \{Z \in B: \vertiii{Z} < A \text{ and } h[Z] > 0\}
\]
 (with $A > 2$) and its closure (in $B$) $\overline{O^A}$ satisfy the following claims.

\begin{lemma}
\label{lemma:open}
Each $O^A$ is an open set in $B$.
\end{lemma}

\begin{proof}
This follows from $\|Z-\tilde Z\|_{L^\infty} \leq C \|Z-\tilde Z\|_{H^3}$  and
\begin{equation}
\begin{split}
\left| F[Z]^{-1} - F[\tilde Z]^{-1} \right| &\le  \left|  \inf_{\substack{\xi, \eta \in \mathbb{T}\\ 1\leq k\leq N}} \frac{|z_k(\xi)-z_k(\eta)|}{|\xi - \eta|} -  \inf_{\substack{\xi, \eta \in \mathbb{T}\\ 1\leq k\leq N}} \frac{|\tilde z_k(\xi)-\tilde z_k(\eta)|}{|\xi - \eta|} \right| \\
& \leq \sup_{\substack{\xi, \eta \in \mathbb{T}\\ 1\leq k\leq N}}    \frac{\left|(z_k(\xi)-\tilde z_k(\xi))-(z_k(\eta)-\tilde z_k(\eta))\right|}{|\xi - \eta|} \\
&\leq \|Z-\tilde Z\|_{C^1} \leq C \|Z-\tilde Z\|_{H^3},
\label{F_diff}
\end{split}
\end{equation}
for some universal $C>0$.
\end{proof}

Notice that 
\[ \{Z \in B: \vertiii{Z} < A \text{ and } h[Z] \ge 0\} \subseteq \overline{O^A} \subseteq \{Z \in B: \vertiii{Z} \le A \text{ and } h[Z] \ge 0\}  \]
Indeed, the second inclusion follows from the proof of Lemma~\ref{lemma:open}. To see the first inclusion, notice that any $Z$ with
$\vertiii{Z} < A$ and $h[Z]=0$ can be approximated by $Z + \sigma e_2 \in O^A$ with $\sigma>0$, which converges to $Z$ in $B$ as $\sigma \rightarrow 0.$


\begin{lemma}\label{lem:mollify}
 If $Z \in \overline{O^A}$ and $\epsilon \in (0, c_0 A^{-2})$ (with a universal $c_0 >0$), then $\phi_\epsilon * Z \in  \overline{O^{2A}}$.
 \end{lemma}

\begin{proof} We obviously have $h[\phi_\epsilon *Z] \geq h[Z]$ and $\|\phi_\epsilon * Z\|_{H^3} \leq \|Z\|_{H^3}$ for all $\epsilon>0$. Since $\phi_\epsilon$ is supported in $[-\epsilon,\epsilon]$, we have also (with a universal $C >0$)
\[
\|\phi_\epsilon * Z - Z \|_{L^\infty} \le  \epsilon \|Z\|_{C^1}  \le  C \epsilon \|Z\|_{H^3}.
\]
Then $\delta[\phi_\epsilon * Z] \geq \delta[Z]-2C\epsilon A > \tfrac 12 \delta[Z]$ for $Z\in \overline{O^A}$ and  $\epsilon \in (0, \frac{1}{4CA^2})$.
Also, \eqref{F_diff}  yields
\begin{equation}
\left| F[\phi_\epsilon *Z]^{-1} - F[Z]^{-1} \right| \le \|\phi_\epsilon * Z - Z\|_{C^1} \leq  \epsilon \|Z\|_{C^2} \le  C\epsilon \|Z\|_{H^3},
\label{conv_diff_temp}
\end{equation}
so again
 $F[\phi_\epsilon * Z] < 2F[Z]$ for $Z\in \overline{O^A}$ and  $\epsilon \in (0, \frac{1}{2CA^2})$.
\end{proof}


Let us denote the right hand side of  \eqref{eq:contour2_regular} by $G_k^\epsilon[Z^\epsilon(t)]$. In general, for any $Z \in B$ with $\vertiii{Z}<\infty$ define
 \begin{equation}\label{Gdef125}
 G_k^\epsilon[Z] = \sum_{i=1}^N \sum_{m=1}^2 \frac{\theta_i}{2\alpha} \phi_\epsilon * H_{k,i}^{m}[\phi_\epsilon * Z] +  \tilde C(\alpha)\Theta \epsilon M^{3+2\alpha} e_2,
 \end{equation}
  where
 \begin{equation}
\label{def_H}
H_{k,i}^{m}[Z](\xi) := \int_{\mathbb{T}}  \frac{\partial_\xi   \z_k(\xi) -  \partial_\xi y_i^{m}(\xi - \eta)}{|\z_k(\xi)  - y_i^{m}(\xi-\eta)|^{2\alpha}}  d\eta.
\end{equation}
Note that the parameter $M$ in \eqref{Gdef125} is independent of $Z$ and is tied to the initial data for which we are trying to establish existence.

We have the following estimates for these operators.

 \begin{lemma}
 \label{diff_H}
There is $C(\alpha)>0$ such that for any $Z, \tilde Z \in \overline{O^{A}}$, any $k,i,m$, and $\alpha<\tfrac 14$,
 \begin{equation}
 \label{eq:diffH}
 \left\| H_{k,i}^m [Z] -H_{k,i}^m [\tilde Z] \right\|_{L^\infty} \leq C(\alpha) A^{2+2\alpha} \|Z-\tilde Z\|_{C^1}.
 \end{equation}
 \end{lemma}
 \begin{proof}
Let $w_k := z_k-\tilde z_k$, as well as $v_k^1:=w_k$  and $v_k^2:=\bar w_k$.
 Then
 \begin{equation*}
 \begin{split}
 \left| H_{k,i}^m [Z](\xi) -H_{k,i}^m [\tilde Z](\xi) \right| &\leq  \int_{\mathbb{T}} \frac{|\partial_\xi   w_k(\xi) -  \partial_\xi v_i^{m}(\xi - \eta)|}{|\tilde \z_k(\xi)  - \tilde y_i^{m}(\xi-\eta)|^{2\alpha}}   d\eta  \\
&+  \int_{\mathbb{T}}  \frac{|\partial_\xi   \z_k(\xi) -  \partial_\xi y_i^{m}(\xi - \eta)|}{|\tilde\z_k(\xi)  - \tilde{y}_i^{m}(\xi-\eta)|^{2\alpha}}  \left| \left( \frac{|\tilde \z_k(\xi)  - \tilde{y}_i^{m}(\xi-\eta)|}{|\z_k(\xi)  - y_i^{m}(\xi-\eta)|}\right)^{2\alpha} -1\right|d\eta \\
&\leq C(\alpha) A^{2\alpha} \|Z-\tilde Z\|_{C^1} +  C(\alpha)  A^{2+2\alpha} \|Z-\tilde Z\|_{L^\infty}\\
& \leq C(\alpha) A^{2+2\alpha} \|Z-\tilde Z\|_{C^1},
 \end{split}
 \end{equation*}
 where similarly to \eqref{Akk2}, we used  $|x^{2\alpha}-1| \leq |x-1|$ for  $x\ge 0$ and \eqref{eq:diff_derivative} (for $m=2$) to bound the second integral by
 \[
\int_{\mathbb{T}}  \frac{CA^{3/2} \|Z-\tilde Z\|_{L^\infty} }{|\tilde\z_k(\xi)  - \tilde{y}_i^{m}(\xi-\eta)|^{2\alpha}|\z_k(\xi)  - y_i^{m}(\xi-\eta)|^{1/2}} d\eta
\le  C(\alpha)  A^{2+2\alpha} \|Z-\tilde Z\|_{L^\infty}
 \]
for $\alpha<\tfrac 14$.
 \end{proof}


 \begin{lemma}
 \label{local_lip}
 If  $Z, \tilde Z \in \overline{O^A}$ and $\epsilon \in (0, c_0 A^{-2})$ (with $c_0 $ from Lemma \ref{lem:mollify}), then
\[
\|G_k^\epsilon[Z]-G_k^\epsilon[\tilde Z]\|_{C^n}\leq C(\alpha,n)\Theta \epsilon^{-n-1} A^{2+2\alpha}   \|Z-\tilde Z\|_{L^\infty}
\]
for any $k$, integer $n\ge 0$, and $\alpha\in(0,\tfrac 14)$.  In particular,  $G_k^\epsilon: \overline{O^A}\to B$ is Lipschitz.
\end{lemma}

\begin{proof}
It is easy to check that $G_k^\epsilon$ maps $\overline{O^A}$ to $B$ for any $\epsilon>0$. The properties of $\phi_\epsilon$ and
Lemma \ref{diff_H} now yield
\[
\begin{split}
\|G_k^\epsilon[Z]-G_k^\epsilon[\tilde Z]\|_{C^n} &\leq C(n) \epsilon^{-n} \sum_{i=1}^N \sum_{m=1}^2 \frac{\theta_i}{2\alpha} \|H_{k,i}^{m}[\phi_\epsilon * Z] -H_{k,i}^{m}[\phi_\epsilon * \tilde Z]  \|_{L^\infty}\\
&\leq C(\alpha,n)\Theta \epsilon^{-n} (2A)^{2+2\alpha}   \| \phi_\epsilon*Z - \phi_\epsilon * \tilde Z\|_{C^1}
\\
&\leq C(\alpha,n)\Theta \epsilon^{-n-1}  A^{2+2\alpha}   \|Z-\tilde Z\|_{L^\infty}.
\end{split}
\]
The last claim follows from taking $n=3$ and using $\|Z-\tilde Z\|_{L^\infty}\le \|Z-\tilde Z\|_{H^3}$.
\end{proof}

\noindent \textbf{Step 2.}    Let $\epsilon_0 := \min\{c_0  (4M)^{-2}$, 1\}, with $c_0 $ from Lemma \ref{lem:mollify} and $M = \vertiii{Z_0}\ge 1$. We then have $\vertiii{Z_0^\epsilon} \leq \vertiii{\phi_\epsilon*Z_0} + \epsilon < 3M$ for any $\epsilon \in (0,\epsilon_0)$, hence $Z_0^\epsilon \in O^{3M}$.
Also, Lemma \ref{local_lip} shows that $G^\epsilon_k$ is Lipschitz on $O^{4M}$ for any $k$ and $\epsilon \in (0, \epsilon_0)$.
Lemma~\ref{lemma:open} and Picard's Theorem applied in Banach space $B$ thus gives us a solution $Z^\epsilon(t)\in O^{4M}$ to \eqref{eq:contour2_regular} with initial data $Z_0^\epsilon$, on some short time interval $[0,t']$ and in the integral sense.  Then Lemma \ref{local_lip} with $n=0$ shows
\[
\sup_{t\in[0,t']}\|G_k^\epsilon[Z^\epsilon(t)]\|_{L^\infty} \le C(\alpha)\Theta \epsilon^{-1}  M^{3+2\alpha}
\]
for each $\eps>0$, so that $Z^\epsilon:[0,t']\to L^\infty(\mathbb T)$ is Lipschitz.  Another application of Lemma \ref{local_lip}, with $Z$ and $\tilde Z$ being $Z^\epsilon$ at different times shows that $\partial_\xi^n G_k^\epsilon[Z^\epsilon(\cdot)](\cdot)$ ($=\partial^\xi_n \partial_t z_k^\eps(\cdot,\cdot)$) is Lipschitz on $\mathbb T\times[0,t']$ for each $n\ge 0$.
Since $z_k^\epsilon(\xi,t)=z_k^\epsilon(\xi,0)+\int_0^t \partial_t z_k^\epsilon(\xi,s)ds$ and $Z^\epsilon(0)\in C^\infty$, we have that $Z^\epsilon \in C^{\infty,1}(\mathbb{T}\times [0,t'])$.

This $Z^\epsilon$ can then be continued in time as long as it stays in $O^{4M}$, and we let $t_\epsilon$ be the maximal such time.  We  have $Z^\epsilon \in C^{\infty,1}(\mathbb{T}\times [0,t_\epsilon])$ as above.  We will therefore be able to apply the a priori estimates from the previous sub-section to $Z^\epsilon$, and show that $t_\epsilon$ is bounded below by some  $T(\alpha,N,\Theta,M)>0$,  uniformly in $\epsilon \in (0,\epsilon_0)$ (for all small $\alpha>0$).

 Using $\int f(\psi*g)d\xi = \int (\psi*f)gd\xi$ for any even $\psi$, we now obtain (dropping $t$)
\begin{equation*}
\begin{split}
\frac{d}{dt}\|\partial_\xi^3 \z_k^\epsilon\|_{L^2}^2 =  \sum_{i=1}^N \sum_{m=1}^2 \frac{\theta_i}{\alpha} \int_{\mathbb{T}^2} \partial_\xi^3 (\phi_\epsilon* \z_k^\epsilon)(\xi) \cdot  \partial_\xi^3 
\left(\frac{\partial_\xi  (\phi_\epsilon* \z_k^\epsilon)(\xi) - \partial_\xi  (\phi_\epsilon* y_i^{m,\epsilon})(\xi-\eta)}{|(\phi_\epsilon*\z_k^\epsilon)(\xi)  - (\phi_\epsilon* y_i^{m,\epsilon})(\xi-\eta)|^{2\alpha}}
\right) d\eta d\xi
\end{split}
\end{equation*}
instead of \eqref{eq:third_derivative}.  The estimates from the previous sub-section thus apply to $\Z^\epsilon$ and lead to the analog of \eqref{evol_h3}.
 Namely, for $t\in[0,t_\epsilon]$ and  $\alpha\in(0,\tfrac 1{24})$ we obtain
\[
\frac{d}{dt}\|\Z^\epsilon(t)\|_{H^3} \leq C(\alpha)N\Theta  \vertiii{(\phi_\epsilon * \Z^\epsilon)( t)}^{7+2\alpha} \leq C(\alpha)N\Theta   (8M)^{7+2\alpha},
\]
where the last inequality holds due to Lemma \ref{lem:mollify} and $\epsilon<\epsilon_0$.  The estimates for
$\|\Z(t)\|_{L^\infty}$, $\delta[\Z(t)]^{-1}$, and $F[\Z(t)]$ also extend to $Z^\epsilon(t)$, with the first gaining an additional term due to the extra term in \eqref{eq:contour2_regular}:
\[
\frac{d}{dt} \|Z^\epsilon(t)\|_{L^\infty} \leq C(\alpha)\Theta  \vertiii{Z^\epsilon(t)}^{1+2\alpha} +  \tilde C(\alpha)\Theta \epsilon M^{3+2\alpha}.
\]
So for $\epsilon \in (0,\epsilon_0)$ (recall that $\eps_0\le 1\le M$) and $t\in[0,t_\epsilon]$ we have (with some $\bar C(\alpha)<\infty$ which also depends on our choice of $\tilde C(\alpha)$)
\begin{equation}\label{3.200}
\frac{d}{dt}\vertiii{\Z^\epsilon(t)} \leq \bar C(\alpha)N\Theta    M^{7+2\alpha}.
\end{equation}
Since $\vertiii{Z_0^\epsilon} < 3M$, it will follow that we have a uniform in $\epsilon \in (0,\epsilon_0)$ estimate $t_\epsilon> T(\alpha,N\Theta,M):= (\bar C(\alpha)N\Theta)^{-1}M^{-6-2\alpha}>0$,  as long as we can show that $h[Z^\epsilon(t)]=0$ cannot happen for $t\in(0,T(\alpha,N\Theta,M)]$.  This will be ensured by choosing $\tilde C(\alpha)$ large enough (keeping in mind that \eqref{3.200} yields $\vertiii{Z^\epsilon(t)}< 4M$ for these $t$).

Indeed, let us assume that $Z^\epsilon(t) \in \overline{O^{4M}}$ and $t$ is the minimal time for which we have $z_k^\epsilon(\xi,t)\in\partial D$ for some $k,\xi$.  Symmetry shows $(H_{k,i}^1[Z^\epsilon(t)](\xi) + H_{k,i}^2[Z^\epsilon(t)](\xi)) \cdot e_2=0$ for any $i$, thus (dropping $t$)
\[
\begin{split}
\partial_t z_k^\epsilon(\xi) \cdot e_2 &=\left( \sum_{i=1}^N \sum_{m=1}^2 \frac{\theta_i}{2\alpha} (\phi_\epsilon * H_{k,i}^m[\phi_\epsilon * Z^\epsilon])(\xi) \right)\cdot e_2 + \tilde C(\alpha)\Theta \epsilon M^{3+2\alpha} \\
&\geq -  \sum_{i=1}^N \sum_{m=1}^2  \frac{\theta_i}{2\alpha}   \left|(\phi_\epsilon * H_{k,i}^m[\phi_\epsilon * Z^\epsilon])(\xi) - H_{k,i}^m[\phi_\epsilon * Z^\epsilon](\xi) \right| \\
&\quad -  \sum_{i=1}^N \sum_{m=1}^2 \frac{\theta_i}{2\alpha}  \left|H_{k,i}^m[\phi_\epsilon * Z^\epsilon](\xi) - H_{k,i}^m[Z^\epsilon](\xi)  \right|+ \tilde C(\alpha)\Theta \epsilon M^{3+2\alpha}\\
&=: -T_1 -T_2 + \tilde C(\alpha)\Theta \epsilon M^{3+2\alpha}.
\end{split}
\]
We use $\|\phi_\epsilon * f - f\|_{L^\infty} \le  \epsilon \|f'\|_{L^\infty}$, \eqref{eq:diff_derivative}, and $\vertiii{\phi_\epsilon * Z^\epsilon}\leq 8M$ to bound
\[
\begin{split}
T_1 &\leq   \epsilon \sum_{i=1}^N \sum_{m=1}^2 \frac{\theta_i}{2\alpha}  \|\partial_\xi H_{k,i}^m[\phi_\epsilon * Z^\epsilon] \|_{L^\infty}\\
&\leq   \epsilon \sum_{i=1}^N \sum_{m=1}^2  \frac{\theta_i}{2\alpha} \int_\mathbb{T}\frac{\|\phi_\epsilon * Z^\epsilon\|_{C^2}}{|(\phi_\epsilon*z_k^\epsilon)(\xi) - (\phi_\epsilon*y_i^{m,\epsilon})(\xi-\eta)|^{2\alpha} } + \frac{|\partial_\xi (\phi_\epsilon*z_k^\epsilon)(\xi) -\partial_\xi (\phi_\epsilon*y_i^{m,\epsilon})(\xi-\eta)|^2 }{|(\phi_\epsilon*z_k^\epsilon)(\xi) - (\phi_\epsilon*y_i^{m,\epsilon})(\xi-\eta)|^{1+2\alpha} }  d\eta  \\
&\leq C(\alpha)\Theta \epsilon  M^{3+2\alpha}.
\end{split}
\]
On the other hand, Lemma \ref{diff_H} yields
\[
T_2 \leq C(\alpha)\Theta \|\phi_\epsilon * Z^\epsilon - Z^\epsilon\|_{C^1} \leq  C(\alpha)\Theta \epsilon  M^{2+2\alpha}  \|Z^\epsilon\|_{C^2} \leq  C(\alpha)\Theta \epsilon  M^{3+2\alpha}.
\]
Hence
\[
\partial_t z_k^\epsilon(\xi,t) \cdot e_2  \geq  \left( \tilde C(\alpha) -C(\alpha) \right) \Theta \epsilon  M^{3+2\alpha} ,
\]
which is positive if we choose $\tilde C(\alpha) >C(\alpha)$.  This yields a contradiction with our assumption that $t$ is the first time of touch, 
so this choice of $\tilde C(\alpha)$ indeed ensures $h[Z^\epsilon(t)]>0$  for $\eps\in(0,\eps_0)$ and  $t\le T(\alpha,N\Theta,M)$.
Thus $t_\epsilon> T(\alpha,N\Theta,M)$ for these $\epsilon$ (with $\alpha<\tfrac 1{24}$).

\noindent \textbf{Step 3.} To obtain local existence of solutions to \eqref{eq:contour2}, we take $\eps\to 0$ (with $\alpha<\tfrac 1{24}$).
Let $\eps, \tilde\eps \in (0,\eps_0),$ consider $Z:=Z^\eps$ and $\tilde Z:=Z^{\tilde\eps}$ solving \eqref{eq:contour2_regular}, and let $W:=Z-\tilde Z$.  We  have for $t\in[0,T]$ (with $T=T(\alpha,N\Theta,M)>0$ from Step 2, and dropping $t$ from the arguments)
\[
\frac d{dt}\|w_k\|_{L^2}^2 = 2 \int_{\mathbb{T}} w_k(\xi) \cdot \partial_t w_k(\xi) d\xi
= K_k +  \sum_{i=1}^N \sum_{m=1}^2  \frac{\theta_i}{\alpha} (I_{k,i}^m + J_{k,i}^m) ,
\]
where
\[
K_k := 2\tilde C(\alpha)\Theta (\epsilon-\tilde \epsilon) M^{3+2\alpha} \int_{\mathbb{T}} w_k(\xi) \cdot e_2 \,d\xi,
\]
while  (by again using  $\int f(\psi*g) d\xi= \int (\psi*f)g d\xi$ for even $\psi$)  $I_{k,i}^m$ equals
\begin{equation*}
\hskip -10mm
\begin{split}
\int_{\mathbb{T}^2} (\phi_{\tilde\eps}* w_k)(\xi) \cdot
\left( \frac{\partial_\xi (\phi_{\tilde\eps}* z_k)(\xi) - \partial_\xi (\phi_{\tilde\eps}* y_i^m)(\xi - \eta)}{| (\phi_{\tilde\eps}*z_k)(\xi) -  (\phi_{\tilde\eps}* y_i^m)(\xi - \eta)|^{2\alpha}} -
\frac{\partial_\xi (\phi_{\tilde\eps}*\tilde{z}_k)(\xi) - \partial_\xi (\phi_{\tilde\eps}*\tilde{y}_i^m)(\xi - \eta)}{| (\phi_{\tilde\eps}*\tilde{z}_k)(\xi) -  (\phi_{\tilde\eps}*\tilde{y}_i^m)(\xi - \eta)|^{2\alpha}}\right) d\eta d\xi
\end{split}
\end{equation*}
and $J_{k,i}^m$ equals
\begin{equation*}
\hskip -30mm
\begin{split}
\int_{\mathbb{T}} w_k(\xi) \cdot
\left( \phi_{\eps}* \int_{\mathbb{T}} \frac{\partial_\xi (\phi_{\eps}* z_k)(\xi) - \partial_\xi (\phi_{\eps}* y_i^m)(\xi - \eta)}{| (\phi_{\eps}*z_k)(\xi) -  (\phi_{\eps}* y_i^m)(\xi - \eta)|^{2\alpha}} d\eta
- \phi_{\tilde\eps}* \int_{\mathbb{T}} \frac{\partial_\xi (\phi_{\tilde\eps}* z_k)(\xi) - \partial_\xi (\phi_{\tilde\eps}* y_i^m)(\xi - \eta)}{| (\phi_{\tilde\eps}*z_k)(\xi) -  (\phi_{\tilde\eps}* y_i^m)(\xi - \eta)|^{2\alpha}} d\eta \right) d\xi.
\end{split}
\end{equation*}

We obviously have $|K_k|\le  C(\alpha)\Theta (\epsilon + \tilde\epsilon)  M^{3+2\alpha} \|W\|_{L^2}$ (with a new $C(\alpha)$).
As in the above uniqueness argument estimating the right hand side of \eqref{eq:temp2}, only with all functions mollified by $\phi_{\tilde\eps}$, we obtain
\[
|I_{k,i}^m| \leq C(\alpha) (\vertiii {\phi_{\tilde\eps}*Z}+\vertiii {\phi_{\tilde\eps}*\tilde \Z })^{2+2\alpha} \|\phi_{\tilde\eps}*W\|_{L^2}^2
\leq C(\alpha) M^{2+2\alpha} \|W\|_{L^2}^2
\]
for all $k,i,m,t$.
On the other hand, the bound
\[
\|\phi_\eps*f-\phi_{\tilde\eps}*g\|_{L^\infty}
\le \| (\phi_\eps-\phi_{\tilde\eps})*f\|_{L^\infty}  + \|\phi_{\tilde\eps}*(f-g)\|_{L^\infty}
\le  (\eps+\tilde\eps) \|f'\|_{L^\infty} + \|f-g\|_{L^\infty}
\]
can be repeatedly used to show that the term in the parentheses in the definition of $J_{k,i}^m$ is bounded uniformly in $\xi$ by
\[
 C(\alpha)(\eps+\tilde\eps) \left( \|Z\|_{C^1}\|Z\|_{C^2}\vertiii {\phi_{\eps}* \Z }^{1+2\alpha}
+  \| \Z\|_{C^2} \vertiii {\phi_{\eps}* \Z }^{2\alpha}
+ \|Z\|_{C^1}\| \Z\|_{C^2} (\vertiii {\phi_{\eps}* \Z } + \vertiii {\phi_{\tilde \eps}* \Z })^{1+2\alpha} \right)
\]
(here we also used \eqref{eq:diff_derivative} for $\phi_\eps*z_k$).  Hence
\[
|J_{k,i}^m| \leq C(\alpha)(\eps+\tilde\eps) M^{3+2\alpha} \|W\|_{L^2}
\]
for all $k,i,m,t$,
so we obtain
\[
\frac d{dt} \|W(t)\|_{L^2} \le C(\alpha) N\Theta M^{3+2\alpha}(\eps+\tilde\eps+ \|W(t)\|_{L^2})
\]
for all $t\in[0,T]$.
Since also $\|W(0)\|_{L^2}\le C(\eps+\tilde\eps) (\|Z_0\|_{C^1}+1)\le CM(\eps+\tilde\eps)$, we get for any $\epsilon, \tilde \epsilon \in (0,\epsilon_0)$,
\[
\sup_{t\in[0,T]}\|Z^{\epsilon}( t) - Z^{\tilde \epsilon}(t)\|_{L^2} \leq  C(\alpha,N\Theta,M) (\eps+\tilde\eps).
\]

Hence $Z^\eps$ converges in $L^\infty([0,T];L^2(\mathbb T)^N)$ to some $Z$ as $\eps\to 0$.  This and the estimate $\sup_{t\in[0,T]} \vertiii{Z^\epsilon(t)} \leq 4M$ for all  $\epsilon \in (0,\epsilon_0)$ then show that  $\sup_{t\in[0,T]} \vertiii{Z( t)} \leq 4M$ and
\[
\lim_{\epsilon\to 0} \sup_{t\in [0,T]} \|Z^{\epsilon}( t) - Z( t)\|_{H^s} = 0
\]
for $s<3$.  We also obtain the same convergence in $C^2$.  This and $\sup_{t\in[0,T]} \vertiii{Z^\epsilon(t)} \leq 4M$ yield that the integrands in \eqref{eq:contour2_regular} converge  to the integrands in \eqref{eq:contour2} as $\eps\to 0$, uniformly on compact subsets of  $\mathbb T\times(\mathbb T\setminus\{0\})\times[0,T]$ (i.e., with $\eta\neq 0$; note that that the denominators in \eqref{eq:contour2} are uniformly bounded below by $C(M)^{-1}\eta^{2\alpha}$ due to $\sup_{t\in[0,T]} \vertiii{Z( t)} \leq 4M$). Since
the integrands are also uniformly bounded above by $C(M)\eta^{-2\alpha}$ for all small $\eps$ (and $2\alpha<1$), it follows that as $\eps\to 0$, the integrals in \eqref{eq:contour2_regular} converge to those in  \eqref{eq:contour2} uniformly on $\mathbb T\times[0,T]$.  Lemma \ref{diff_H} applied to $Z(t)$ and its translate in $\xi$, together with $\sup_{t\in[0,T]} \vertiii{Z( t)} \leq 4M$, shows that the integrals in  \eqref{eq:contour2} are Lipschitz in $\xi$, uniformly in $t\in[0,T]$.  Hence the integrals in \eqref{eq:contour2_regular} convolved with $\phi_\eps$ also converge to them uniformly on $\mathbb T\times[0,T]$ as $\eps\to 0$.  Thus $\partial_t \z^\epsilon_k$ (which is continuous) converges to the right-hand side of \eqref{eq:contour2} as $\eps\to 0$, uniformly on $\mathbb T\times[0,T]$.  The latter is then also  continuous on $\mathbb T\times[0,T]$.  But since $\z^\epsilon_k\to z_k$ as $\eps\to 0$, we see that $\partial_t z_k$ exists
and equals the (continuous) right-hand side of \eqref{eq:contour2}.  Hence $Z$ is a classical solution to \eqref{eq:contour2} and obviously $Z(0)=Z_0$.

The above proves the following local regularity result for the contour equation \eqref{eq:contour2} corresponding to the half-plane case $D=\Rm\times\Rm^+$.

\begin{theorem}
  Let $\theta_1,\dots,\theta_N\in\mathbb R\setminus\{0\}$ and $\Z_0=\{\z_{0k}\}_{k=1}^N\in H^3(\mathbb{T}, \bar D)^N$ be a collection of (counter-clockwise) initial parameterizations of patch boundaries with $\vertiii{\Z_0}<  \infty$ (and $\vertiii{\cdot}$ from \eqref{eq:def_norm}).
  For any $\alpha\in (0,\tfrac{1}{24})$,  there is $T=T(\alpha,N\sum_{k=1}^N|\theta_k|,\vertiii{\Z_0})>0$ such that
  there exists a unique
solution $\Z = \{\z_k\}_{k=1}^N$ to \eqref{eq:contour2}  in $L^\infty([0,T];H^3(\mathbb T,\bar D)^N)\cap C^1([0,T];C(\mathbb T,\bar D)^N)$ with $\Z(0)=\Z_0$
and $\sup_{t\in[0,T]}\vertiii{\Z(t)} < \infty$.
This $T$ can be chosen to be decreasing in the last two arguments and so that $\sup_{t\in[0,T]}\vertiii{Z(t)} \le 4 \vertiii{\Z_0}$.
\label{thm:local}
\end{theorem}

We note that by using an argument similar to that in \cite[Chapter 3]{mb}, we could prove that $Z\in C([0,T];H^3(\mathbb T,\bar D)^N)$, 
but we will not need this.

\subsection{Lemmas on non-negative functions in $H^3(\mathbb{T})$}\label{sec:nonnegative}

We now prove the results about nonnegative $H^3$ functions, used in the proof of local well-posedness for the patch equation in $H^3$.



\begin{proof}[Proof of Lemma~\ref{lemma:f'bis}]
This is obvious for $\gamma=0$ so let us consider $\gamma\in(0,1]$.  Reflection $\xi\mapsto-\xi$ shows that it suffices to consider $\xi\in\mathbb{T}$ such that $f'(\xi)>0$.
For such $\xi$ let
\[
\xi' := \xi- \left(\dfrac{f'(\xi)}{2\|f\|_{C^{1,\gamma}}} \right)^{1/\gamma}.
\]
Then for any $\eta\in[\xi', \xi]$ we have
\[
|f'(\eta) - f'(\xi)| \leq \|f\|_{C^{1,\gamma}}\, |\xi-\eta|^\gamma \leq \frac{f'(\xi)}{2}.
\]
It follows that $f'(\eta)\ge \tfrac 12 f'(\xi)$ for all $\eta\in[\xi',\xi]$, so
\begin{equation*}
f(\xi) \geq f(\xi)-f(\xi') \geq \frac{f'(\xi)}{2} (\xi-\xi')  = \frac{f'(\xi)^{1+1/\gamma}}{2^{1+1/\gamma}\|f\|_{C^{1,\gamma}}^{1/\gamma}}.
\end{equation*}
The result follows.
\end{proof}



Let us now prove Lemma~\ref{cor:sobolevbis}.  We start with showing that if $0<f\in H^3(\mathbb{T})$, then
$f^{\frac{1}{2}-\beta}\in W^{1,1}(\mathbb{T})$ for small $\beta\ge 0$. We were unable to find this result in the literature.

\begin{lemma}\label{lemma:bv}
 If  $\beta\in[0,\tfrac 16]$ and $0< f\in H^3(\mathbb{T})$, then
\begin{equation}
\int_{\mathbb{T}} \frac{|f'(\xi)|}{f(\xi)^{\beta+1/2}} d\xi \leq 10 \|f\|_{H^3}^{\tfrac{1}{2}-\beta}.
\label{eq:bv}
\end{equation}
\end{lemma}

\begin{proof}
  The integral on the left side of (\ref{eq:bv}) is, up to a constant,
  the $BV$-norm of $f^{\frac{1}{2}-\beta}$, which we will bound as
  follows. Let $\xi_0\leq \xi_1\leq \ldots\leq \xi_{n-1}\leq \xi_n = \xi_0+2\pi$
  be a (finite) sequence of local extrema of $f$. If for any such
sequence we can show that
\[
\sum_{i=1}^n
  \left|f(\xi_i)^{\frac{1}{2}-\beta} -f(\xi_{i-1})^{\frac{1}{2}-\beta} \right|
\]
is bounded above by the right hand side of \eqref{eq:bv}, we will be
done.

Using first $f>0$ and concavity of the function $s^{1/2-\beta}$ on $(0,\infty)$, and then H\"older's inequality with $p = 2(\tfrac{1}{2}-\beta)^{-1}$ and $\frac{1}{q} =1-\frac{1}{p} = \frac{3}{4} + \frac{\beta}{2}$, we obtain
\begin{equation}
\begin{split}
\sum_{i=1}^n \Big|f(\xi_i)^{\tfrac{1}{2}-\beta}
-f(\xi_{i-1})^{\tfrac{1}{2}-\beta}\Big|
& \leq \sum_{i=1}^n |h_i|^{\tfrac{1}{2}-\beta} \\
&= \sum_{i=1}^n \left( |h_i|^{\tfrac{1}{2}-\beta} d_i^{-\tfrac{5}{2}(\tfrac{1}{2}-\beta)}\right) d_i^{\,\tfrac{5}{2}(\tfrac{1}{2}-\beta)}\\
&\leq \left( \sum_{i=1}^n h_i^2 d_i^{-5}\right)^{\tfrac{1}{p}} \,\left( \sum_{i=1}^n d_i^{\,\tfrac{5}{2}(\tfrac{1}{2}-\beta)q}\right)^{\tfrac{1}{q}},
\end{split}
\end{equation}
where $h_i := f(\xi_i)-f(\xi_{i-1})$ and $d_i := \xi_i - \xi_{i-1}$.
Since $\sum_{i=1}^n d_i = 2\pi$, the second sum in the last expression is bounded above
by $(2\pi)^{\frac{5}{2}(\frac{1}{2}-\beta)q}$ as long as
$\frac{5}{2}(\frac{1}{2}-\beta)q\geq 1$, which is equivalent
to~$\beta\leq \frac{1}{6}$. Since $\beta \geq 0,$ we have $\frac{5}{2}(\frac{1}{2}-\beta)\le \tfrac 54$, so it suffices to prove that
\begin{equation}
\sum_{i=1}^n h_i^2 d_i^{-5} \leq  \left[ 10(2\pi)^{-5/4} \right]^p \|f\|_{H^3}^2.
\label{eq:h2d-5}
\end{equation}

Since $\max_{\xi\in[\xi_{i-1},\xi_i]} |f'(\xi)| \ge h_id_i^{-1}$ and $f'(\xi_i)=0=f'(\xi_{i-1})$, we have
\[
\max_{\xi\in[\xi_{i-1},\xi_i]} |f''(\xi)| \ge 2h_id_i^{-2}.
\]
H\" older inequality and Rolle's theorem for $f'$ (i.e., $f''(\xi)=0$ for some $\xi\in [\xi_{i-1},\xi_i]$) yield
\[
\int_{\xi_{i-1}}^{\xi_i} f'''(\xi)^2 d\xi \ge \frac 1{d_i}  \left( \int_{\xi_{i-1}}^{\xi_i} |f'''(\xi)| d\xi \right)^2
\ge \frac 1{d_i} \left( \max_{\xi\in[\xi_{i-1},\xi_i]} |f''(\xi)| \right)^2 \ge \frac{4h_i^2}{d_i^5}.
\]
Summing this up over $i=1,\dots,n$ and using that $\tfrac 14<  \left[ 10(2\pi)^{-5/4} \right]^p$ gives \eqref{eq:h2d-5}.
\end{proof}




\begin{proof} [Proof of Lemma~\ref{cor:sobolevbis}]
Multiplying $(n-1)$-st power of  \eqref{eq:f'_sqrtfbis} with $\gamma=1$  by $|f'(\xi)| f(\xi)^{-\beta-n/2}$, then integrating and using Lemma \ref{lemma:bv} yields with a universal $C<\infty$,
\[
\int_{\mathbb{T}} \frac{|f'(\xi)|^n}{f(\xi)^{ \beta+n/2}} dx
\leq C^n\|f\|_{H^3}^{\tfrac{n-1}{2}} \int_{\mathbb{T}} \frac{|f'(\xi)|}{f(\xi)^{ \beta+1/2}} d\xi
\le 3C^n \|f\|_{H^3}^{\tfrac{n}{2}-\beta}.
\]
This is \eqref{eq:sobolev_1bis}.
As for \eqref{eq:sobolev_2bis},  we obtain via integration by parts
\begin{equation*}
\begin{split}
\int_{\mathbb{T}} \frac{f''(\xi)^2}{f(\xi)^{\beta}} dx  &\leq \left|\int_{\mathbb{T}}  \frac{f'(\xi) f'''(\xi)}{f(\xi)^{\beta}} d\xi\right|  + \beta\left|\int_{\mathbb{T}} \frac{ f'(\xi)^2 f''(\xi)}{f(\xi)^{\beta+1}} d\xi\right|\\
&\leq \|f\|_{H^3} \|f' f^{-\beta}\|_{L^2} + \beta \|f\|_{C^2} \left|\int_{\mathbb{T}} \frac{f'(\xi)^2}{f(\xi)^{\beta+1}} d\xi\right|.
\end{split}
\end{equation*}
For any $\beta \in [0,\frac{1}{6}]$ we have $\|f'f^{-\beta}\|_{L^\infty} \leq C\|f\|_{H^3}^{1-\beta}$  by \eqref{eq:f'_sqrtfbis} with $\gamma:=\beta(1-\beta)^{-1}$, and $\|(f')^2f^{-\beta-1}\|_{L^1} \leq C\|f\|_{H^3}^{1-\beta}$ by \eqref{eq:sobolev_1bis} with $n=2$.
Sobolev embedding now yields \eqref{eq:sobolev_2bis}.
\end{proof}

\section{Estimates on velocity fields and  $C^{1,\gamma}$ patches}
\label{sec:regularity}



In this section, we prove some basic estimates on the fluid velocities for general $\omega$, as well as on the geometry of $C^{1,\gamma}$ patches (we will always consider $\gamma\in(0,1]$).  Naturally, the latter also apply in the case of the more regular $H^3$ patches.

\begin{lemma}\label{lemma:uniform_u_bound}
For $D=\Rm\times\Rm^+$, $\alpha\in(0,\tfrac 12)$, and $u(\cdot,t)$ from \eqref{eq:velocity_law} with $\omega(\cdot,t) \in L^1(D) \cap L^\infty(D)$, we have
\begin{equation}\label{uLinfty}
\|u(\cdot,t)\|_{L^\infty} \leq \frac{2\pi}{1-2\alpha} \|\omega(\cdot, t)\|_{L^\infty}+ 2\|\omega(\cdot,t)\|_{L^1}
\end{equation}
and
\begin{equation}\label{uHold}
\|u(\cdot,t)\|_{C^{1-2\alpha}} \leq  \frac {8\pi}{\alpha(1-2\alpha)} \|\omega(\cdot, t)\|_{L^\infty} + 2\|\omega(\cdot,t)\|_{L^1}.
\end{equation}
Furthermore, if $\omega$ is weak-$*$ continuous as a function from some time interval $[a,b]$ to $L^\infty(D)$, and is supported inside some fixed compact subset of $\bar D$ for every $t\in [a,b]$, then $u$ is continuous on $\bar D \times [a,b]$.
\end{lemma}

\begin{proof}
Let  $\eta:\mathbb{R}^2\to\mathbb{R}$ be the odd extension of $\omega(\cdot,t)$ to
the whole plane.
The velocity law \eqref{eq:velocity_law} for $x\in D$  then becomes
\begin{equation}\label{genulaw}
u(x,t) =  \int_{\mathbb{R}^2} \frac{(x-y)^\perp}{|x-y|^{2+2\alpha}} \eta(y) dy,
\end{equation}
and \eqref{uLinfty} follows from
\begin{equation*}
\begin{split}
|u(x,t)| &\leq  \int_{|x-y|\leq 1} \frac{|\eta(y)|}{|x-y|^{1+2\alpha}}dy
+ \int_{|x-y|> 1} \frac{|\eta(y)|}{|x-y|^{1+2\alpha}} dy\\
&\leq  \|\eta\|_{L^\infty} \int_{|x-y|\leq 1}
\frac{1}{|x-y|^{1+2\alpha}}  dy +
\|\eta\|_{L^1} \\
& \leq \frac{2\pi}{1-2\alpha} \|\omega(\cdot,t)\|_{L^\infty}+ 2\|\omega(\cdot,t)\|_{L^1}.
\end{split}
\end{equation*}

To prove \eqref{uHold}, consider any $x,z \in \bar D$ with $r:=|x-z|$.
Then
\begin{equation*}
\begin{split}
|u(x,t)-u(z,t)| \leq & \int_{B(x,2r)} \frac{1}{|x-y|^{1+2\alpha}} \eta(y)\,dy + \int_{B(x,2r)} \frac{1}{|z-y|^{1+2\alpha}} \eta(y)\,dy \\
& + \int_{\Rm^2 \setminus B(x,2r)} \left| \frac{(x-y)^\perp}{|x-y|^{2+2\alpha}} - \frac{(z-y)^\perp}{|z-y|^{2+2\alpha}} \right| \eta(y)\,dy  \\
\le & 4\pi\|\eta\|_{L^\infty} \int_0^{3r} s^{-2\alpha}\,ds + 32\|\eta\|_{L^\infty}\int_{2r}^\infty r s^{-1-2\alpha}\,ds \\
\leq & \left( \frac {12\pi}{1-2\alpha}+\frac {32}{2\alpha} \right) \|\eta\|_{L^\infty} |x-z|^{1-2\alpha}.
\end{split}
\end{equation*}
Combining this with \eqref{uLinfty} yields \eqref{uHold}.


It remains to prove the last claim.  Since the kernel in \eqref{genulaw} is $L^1$ on any compact subset of $\bar D$, the assumptions show that $u$ is continuous in $t\in [a,b]$  for any fixed $x\in\bar D$.  The claim now follows from  uniform continuity of $u$ in $x\in\bar D$, see \eqref{uHold}.
\end{proof}

\it Remark. \rm As is clear from the proof, the lemma remains valid in the more general case where
$u$ is given by \eqref{genulaw} with $\omega(y,t)$ in place of $\eta(y)$ and $\omega$ satisfies the hypotheses of the lemma with $D$ replaced by $\mathbb{R}^2$.



The remaining results in this section hold for a single time $t$, so we drop the time variable from the notation.

\begin{lemma}
\label{lem:du_crude}
For $\alpha\in(0,\tfrac 12)$ and $\omega \in L^\infty (\mathbb{R}^2) \cap L^1 (\mathbb{R}^2)$, let
\begin{equation} \label {4.1}
v(x) := \int_{\mathbb{R}^2} \frac{(x-y)^\perp}{|x-y|^{2+2\alpha}} \omega(y) dy.
\end{equation}
Assume that $\omega \equiv c$ in $B(x,d)$, for some $x\in \mathbb{R}^2$, $d>0$, and $c\in\mathbb R$. Then
\[
|\nabla u(x)| \leq C(\alpha) \|\omega\|_{\infty} d^{-2\alpha},
\]
and more generally, for $D^n$ any spatial derivative of order $n$ we have
\[
|D^n u(x)| \leq C(\alpha,n) \|\omega\|_{\infty} d^{1-2\alpha-n}.
\]
\end{lemma}

\begin{proof}
Let $\phi:\mathbb{R}^+ \to [0,1]$ be a smooth  function with $\phi \equiv 0$ on $[0,\frac{1}{3}]$, $\phi\equiv 1$ on $[\frac{1}{2},\infty)$, and $0\leq \phi' \leq 10$. Let
\[
g(x) := \int_{\mathbb{R}^2} \frac{(x-y)^\perp}{|x-y|^{2+2\alpha}} \phi\left( \frac{|x-y|}{d}\right)\omega(y) dy.
\]
Then  $g\equiv v$ on $B(x,\tfrac d2)$ due to $\omega\equiv c$ in $B(x,d)$ and the mean zero property of the kernel. Hence
\[
\begin{split}
|\nabla v(x)| &= |\nabla g(x)| \leq C(\alpha)\|\omega\|_\infty \int_{\mathbb{R}^2} \left[\frac{1}{|x-y|^{2+2\alpha}} \phi\left( \frac{|x-y|}{d}\right) + \frac{1}{d|x-y|^{1+2\alpha}} \phi'\left( \frac{|x-y|}{d}\right)\right] dy\\
&\leq C(\alpha)\|\omega\|_\infty \left(\int_{d/3}^\infty r^{-(1+2\alpha)} dr + \int_{d/3}^{d/2} d^{-1} r^{-2\alpha}  dr \right)\\
&\leq C(\alpha) \|\omega\|_{\infty} d^{-2\alpha}.
\end{split}
\]
The proof of the higher derivatives case is analogous.
\end{proof}

Let us now turn to $C^{1,\gamma}$ patches.

\begin{definition} \label{D.4.1a}
For a bounded open  $\Omega\subseteq \mathbb{R}^2$ whose boundary  is a simple closed $C^{1,\gamma}$ curve with arc-length $|\partial\Omega|=:2\pi L$, let  $\vertiii{\Omega}_{1,\gamma} := \|z\|_{C^{1,\gamma}} + F[z]$, where $z$ is any constant speed parametrization of $\partial \Omega$ from Definition \ref{D.1.0} and $F[z] := \max\{\sup_{\xi,\eta \in \mathbb{T}, \xi \neq \eta} \frac{|\xi - \eta|}{|z(\xi)-z(\eta)|},1\}$.  We also denote by $n_p$ the outer unit normal vector for $\Omega$ at $P\in\partial \Omega$.
\end{definition}

{\it Remark.}  We clearly have $z'(\xi)=-Ln_{z(\xi)}^\perp$.

\begin{lemma}
\label{lem:regularity}
Let $\Omega\subseteq \mathbb{R}^2$ be as in Definition \ref{D.4.1a}, with
$\vertiii{\Omega}_{1,\gamma}\leq A$ for some $A\geq 1$. Let $R := L^{\frac{1}{\gamma}}(4A)^{-\frac{1}{\gamma}-1}$  and consider any $P \in \partial \Omega$.  Then we have:
\begin{enumerate}[(a)]
\item $\partial \Omega \cap B(P,R)$ is a simply connected curve.
\item In the coordinate system $(w_1,w_2)$ centered at $P$ and with axes $n_P^\perp$ and $n_P$,  the set $\partial \Omega \cap B(P,R)$ is a graph $w_2=f(w_1)$ with $|f(w_1)|\le 4L^{-1-\gamma}A |w_1|^{1+\gamma}$.
\item For any $Q \in \partial \Omega \cap B(P,R)$, we have $|n_P-n_Q| \leq 2L^{-1-\gamma}A|P-Q|^{\gamma}$.
\item If also $\Omega \subseteq D$, then for $P=(p_1, p_2) \in \partial \Omega$ we have $|n_P\cdot (1,0)| \leq 2L^{-1}A^{\frac{1}{1+\gamma}} p_2^{\frac{\gamma}{1+\gamma}}.$
\end{enumerate}
\end{lemma}

\begin{proof}
(a) Let $\xi_0\in\partial\Omega$ be arbitrary and let $P := z(\xi_0)$.  Since $\|z\|_{C^{1,\gamma}} \leq A$, we have $-z'(\xi) \cdot n_P^\perp \geq \frac{L}{2}$ for all $\xi$ such that $|\xi-\xi_0| \leq \left(\frac{L}{2A}\right)^{1/\gamma}$. Moreover, if $|\xi-\xi_0| > \left(\frac{L}{2A}\right)^{1/\gamma}$, then
$z(\xi) \not \in B(P,R)$ due to  $F[z]\leq A$ and the definition of $R$.

(b)
Let us write $z(\xi)=P+w_1(\xi) n_P^\perp + w_2(\xi) n_P$. Then for any $z(\xi)\in B(P,R)$, the discussion above gives $-w_1'(\xi) \geq \tfrac L2$ and thus $|w_1(\xi)| \geq \tfrac L2 |\xi-\xi_0|$, where $P=z(\xi_0)$.  Since $w_2'(\xi_0) = z'(\xi_0)\cdot n_P = 0$, we also have $|w_2'(\xi)| = |w_2'(\xi) - w_2'(\xi_0)| \leq A|\xi - \xi_0|^\gamma$. Hence when $z\in B(P,R)$ (i.e., when $w_1(\xi)^2+w_2(\xi)^2\leq R^2$), we have
\[
\left|\frac{dw_2}{dw_1}\right| = \left| \frac{w_2'(\xi)}{w_1'(\xi)} \right| \leq \frac{A|\xi-\xi_0|^\gamma}{L/2} \leq \frac{A|2w_1(\xi)/L|^\gamma}{L/2}\leq 4L^{-1-\gamma}A|w_1|^\gamma.
\]
Integrating this inequality gives $|w_2| \leq 4(1+\gamma)^{-1}L^{-1-\gamma}A|w_1|^{1+\gamma}$.

(c) Let $\xi_1\in \mathbb{T}$ be such that $Q = z(\xi_1)$. From $n_P^\perp =- \tfrac 1L z'(\xi_0)$ and $n_Q^\perp =- \tfrac 1L  z'(\xi_1)$ we obtain
\[
|n_P - n_Q| = |n_P^\perp - n_Q^\perp| = \frac{|z'(\xi_0) - z'(\xi_1)|}{L} \leq  \frac{A |\xi_0 - \xi_1|^\gamma}{L} \leq  \frac{A (2|P-Q|/L)^\gamma}{L}\leq 2L^{-1-\gamma}A|P-Q|^\gamma,
\]
where we also used $|z(\xi_1)-z(\xi_0)| \geq \tfrac L2 |\xi_1 - \xi_0|$, due to $-w_1'(\xi) \geq \tfrac L2$ when $z(\xi) \in B(P,R).$

(d) Let again $P = z(\xi_0)$, so that we need to show $|z_2'(\xi_0)| \le 2A^{\frac{1}{1+\gamma}} p_2^\frac{\gamma}{1+\gamma}$.  We have $|z_2'(\xi)-z_2'(\xi_0)|\le \tfrac 12|z_2'(\xi_0)|$ when $|\xi - \xi_0| \leq \left(\frac{|z_2'(\xi_0)|}{2A}\right)^{1/\gamma}$, so $\xi_\pm := \xi_0 \pm \left(\frac{|z_2'(\xi_0)|}{2A}\right)^{1/\gamma}$ satisfy
\[
0\le \min \{z_2(\xi_+),z_2(\xi_-)\} \le  z_2(\xi_0) -  \left(\frac{|z_2'(\xi_0)|}{2A}\right)^{1/\gamma}  \frac{|z_2'(\xi_0)|}{2},
\]
which is $|z_2'(\xi_0)| \le 2A^{\frac{1}{1+\gamma}} p_2^\frac{\gamma}{1+\gamma}$.
\end{proof}

\section{Local regularity for the patch equation and small~$\alpha$}
\label{sec:patch}

In this section we will prove that the solution of the contour equation which we constructed in Section \ref{sec:lwp} yields a (local) $H^3$ patch solution to \eqref{sqg}-\eqref{eq:velocity_law}, and that the latter is also the unique $H^3$ patch solution to \eqref{sqg}-\eqref{eq:velocity_law}.  This is achieved in Corollary \ref{C.5.7}  and  Theorem \ref{thm:unique_patch}, which together with the remark after Corollary \ref{C.5.7} prove Theorem \ref{T.1.1}.
We consider here the half-plane case $D=\Rm\times\Rm^+$, but the arguments are identical for the whole plane $D=\Rm^2$.

\subsection{Contour equation solution is a patch solution}
\label{sec:contpatch}

We start with using the results from the previous section to show that the solution of the contour equation is also a patch solution to \eqref{sqg}-\eqref{eq:velocity_law}.  The main result of this sub-section is the following proposition.

\begin{proposition} \label{P.5.1}
Consider the setting of Theorem \ref{thm:local} and let $\Omega_k(t)$ be the interior of the contour $z_k(\cdot,t)$.  Then $\omega(\cdot,t):=\sum_{k=1}^N \theta_k \chi_{\Omega_k(t)}$ is an $H^3$ patch solution to \eqref{sqg}-\eqref{eq:velocity_law} on $[0,T]$.
\end{proposition}

Since $z_k(\cdot,t)$ need not be a constant speed parametrization of $\partial\Omega_k(t)$, we first need to obtain a bound on the latter.

\begin{lemma}\label{lem:arcl_bd}
Let $Z=(z_1,\dots,z_N):\mathbb T\to (\Rm^2)^N$ and assume $\vertiii{Z}<\infty$, with $\vertiii{\cdot}$ from \eqref{eq:def_norm} (thus the $z_k$ are pairwise disjoint simple closed curves).  Let $\Omega_k$ be the interior of the curve $z_k$ and let $y_k$ be any constant speed  parametrization of $\partial\Omega_k$ from Definition \ref{D.1.0}.  There is a universal constant $C\ge 1$ such that $Y=(y_1,\dots,y_N)$ satisfies
 \[
 \vertiii{Y} \leq C\vertiii{Z}^8.
 \]
 \end{lemma}

\begin{proof}Since all constant speed parametrizations of $\partial \Omega_k$ are translations of each other on $\mathbb{T}$ (and such translation does not affect $\vertiii{Y}$), it suffices to prove the result for one of them. We will therefore assume $Y(0) = Z(0)$.  We can also assume without loss that $z_k$ is a counter-clockwise parametrization of $\partial\Omega_k$.

We obviously have $\delta[Y] = \delta[Z]$ and $\|y_k\|_{L^\infty} = \|z_k\|_{L^\infty}$ for each $k$.  Since $y_k$ and $z_k$ are both counter-clockwise parametrizations of $\partial \Omega_k$, with $|y_k'(\xi)|\equiv \frac 1{2\pi}|\partial \Omega_k|$, there is a bijection $f_k:\mathbb{T}\to\mathbb{T}$ with $f_k(0) = 0$ such that  $y_k(f_k(\xi)) \equiv z_k(\xi)$.  Then for $\xi \in \mathbb{T}$,
\begin{equation}
\label{eq:fk}
f_k'(\xi) = \frac{2\pi |z_k'(\xi)|}{|\partial \Omega_k|}.
\end{equation}

To simplify notation, let us now drop the index $k$, and denote $y=(y_1, y_2)$ and $z=(z_1, z_2)$.  For any distinct $\eta_1,\eta_2\in\mathbb{T}$, there are distinct $\xi_1, \xi_2\in \mathbb{T}$ such that $\eta_1 = f(\xi_1)$ and $\eta_2 = f(\xi_2)$. Then
\[
\frac{|\eta_1-\eta_2|}{|y(\eta_1)-y(\eta_2)|}
=   \frac{|f(\xi_1) - f(\xi_2)|}{|z(\xi_1) - z(\xi_2)|}
\leq \frac{|f(\xi_1) - f(\xi_2)|}{|\xi_1-\xi_2|} F[Z]  \leq  \|f'\|_{L^\infty} \vertiii{Z}.
\]
Since we have $\|z\|_{C^1}\le C\vertiii{Z}$ (with a universal $C<\infty$, which may change later) and $|\partial\Omega|\ge 2|z(\pi)-z(0)|\ge \frac{2\pi}{F[Z]}$, it follows from \eqref{eq:fk} that $\|f'\|_{L^\infty}  \le C \vertiii{Z}^2$, yielding $F[Y]\le C\vertiii{Z}^3$.

Since $\vertiii{Z}\ge 1$ by definition, it thus suffices to show $\|y'''\|_{L^2}\le C\|z\|_{H^3} \vertiii{Z}^7$.
 A direct computation yields
\[
y_1'(f(\xi)) = \frac{|\partial \Omega|}{2\pi} \frac{z_1'(\xi)}{|z'(\xi)|},
\]
\[
y_1''(f(\xi)) = \frac{|\partial \Omega|^2}{(2\pi)^2} \left( \frac{z_1''(\xi) }{|z'(\xi)|^2} - \frac{z_1'(\xi) [z_1''(\xi) z_1'(\xi) + z_2''(\xi) z_2'(\xi)]}{|z'(\xi)|^4}\right),
\]

\[
y_1'''(f(\xi)) = \frac{|\partial \Omega|^3}{(2\pi)^3} \left( \frac{z_1'''}{|z'|^3} - \frac{z_1''' (z_1')^2 + z_2''' z_1' z_2' +4 (z_1'')^2 z_1' +3 z_1'' z_2'' z_2' + (z_2'')^2 z_1'}{|z'|^5} + \frac{4z_1'[z_1''z_1' + z_2'' z_2']^2}{|z'|^7}\right),
\]
where for convenience we dropped $\xi$ in the last expression. Therefore,
\[
|y_1'''(f(\xi))| \leq C |\partial \Omega|^3 \left( \frac{|z'''(\xi)|}{|z'(\xi)|^3} + \frac{|z''(\xi)|^2}{|z'(\xi)|^4}\right).
\]
This gives the following estimate on $\|y_1'''\|_{L^2}$ (recall that $f$ is a bijection):
\[
\begin{split}
\|y_1'''\|_{L^2}^2 &= \int_{\mathbb{T}} [y_1'''(f(\xi))]^2 f'(\xi) d\xi \\
&\leq \int_{\mathbb{T}} C |\partial \Omega|^6 \left( \frac{|z'''(\xi)|^2 }{|z'(\xi)|^6} + \frac{|z''(\xi)|^4}{|z'(\xi)|^8}\right) \frac{|z'(\xi)|}{|\partial \Omega|} d\xi \\
&\leq C |\partial \Omega|^5 \left( \frac {\|z\|_{H^3}^2} {\min_{\xi \in \mathbb{T}} |z'(\xi)|^{5}} + \frac {\|z\|_{C^2}^4} {\min_{\xi \in \mathbb{T}} |z'(\xi)|^{7}} \right).
\end{split}
\]
Since $|\partial \Omega| \leq \frac 1{2\pi}\|z'\|_{C^1} \leq C\vertiii{Z}$ and $\min_{\xi \in \mathbb{T}} |z'(\xi)| \geq \frac 1{F[Z]} \geq \vertiii{Z}^{-1}$, it follows that $\|y_1'''\|_{L^2} \leq C\|z\|_{H^3}\vertiii{Z}^7$.  Since the same estimate holds for $y_2$, the proof is finished.
\end{proof}

\begin{proof}[Proof of Proposition \ref{P.5.1}]
First note that by Theorem \ref{thm:local}, the boundaries $\partial\Omega_k(t)$ are pairwise disjoint simple closed curves in $\bar D$ for each $t\in[0,T]$, which also have parametrizations $z_k(\cdot,t)$ that are uniformly-in-time bounded in $H^3$.  Due to Lemma \ref{lem:arcl_bd}, the latter then also holds for their constant speed parametrizations.  Lemma \ref{lemma:S} shows that each $\partial\Omega_k$ is continuous in time with respect to Hausdorff distance, so it remains to show \eqref{1.3}.

The derivation of \eqref{eq:contour2} shows that its  right-hand side $S_k[Z(t)](\xi)$ satisfies
\[
S_k[Z(t)](\xi)=u(z_k(\xi,t),t)+\beta_k(\xi,t) \partial_\xi z_k(\xi,t)
\]
for some $\beta_k(\xi,t)\in\Rm$, and Theorem \ref{thm:local} shows that $S_k[Z(\cdot)](\cdot)$ is continuous on $\mathbb T\times[0,T]$.  Since $z_k$ and $u$ are also continuous (the latter by Lemma \ref{lemma:uniform_u_bound}) and $\partial_\xi z_k(\xi,t)\ge F[Z(t)]^{-1}>0$, we will have that $\beta_k$ is continuous if we show that $\partial_\xi z_k$ is.  But the latter holds because $\sup_{t\in[0,T]}\|z_k(\cdot,t)\|_{C^2}<\infty$ and $z_k$ is continuous in $t\in[0,T]$.

This means that for each $\tau\in[0,T]$ and $x\in\partial\Omega_k(\tau)$, the ODE $\zeta'(t)=-\beta_k(\zeta(t),t)$ has a solution $\zeta:[0,T]\to\mathbb T$ with $\zeta(\tau)=\xi$, where $\xi\in\mathbb T$ is such that $z_k(\xi,\tau)=x$.
Then $\Psi_{x,\tau}(t):=z_k(\zeta(t),t)\in\partial\Omega_k(t)$ solves
\[
\frac d{dt} \Psi_{x,\tau}(t)=u(\Psi_{x,\tau}(t),t) \qquad\text{for $t\in[0,T]$ and}\qquad  \Psi_{x,\tau}(\tau)=x.
\]
Given any $t\in(0,T)$, for any small $h\in\Rm$ and $x\in\partial\Omega(t+h)$, let $y_{x,t,h}:=\Psi_{x,t+h}(t)\in\partial\Omega(t)$ and $\tilde y_{x,t,h}:=y_{x,t,h}+hu(y_{x,t,h},t)\in X_{u(\cdot,t)}^h[\partial\Omega(t)]$.  Then
\[
\sup_{x\in\partial\Omega(t+h)} |x-\tilde y_{x,t,h}|\le h w \left(2 |h|\|u\|_{L^\infty} \right)
\]
for all small $h$, where $w$ is the modulus of continuity of $u$ on some neighborhood of $\partial\Omega(t)\times\{t\}$.  Note that $w$ satisfies $\lim_{s\searrow 0} w(s)=0$ because $u$ is continuous and $\partial\Omega(t)$ is compact, which (together with \eqref{uLinfty})  yields
\[
\lim_{h\to 0} \sup_{x\in\partial\Omega(t+h)} \frac{{\rm dist} \left(x,  X_{u(\cdot,t)}^h[\partial\Omega(t)] \right)}h=0.
\]
Similarly, for small $h\in\Rm$ and $x\in X_{u(\cdot,t)}^h[\partial\Omega(t)]$, there is $y_{x,t,h}\in \partial\Omega(t)$ such that $x=y_{x,t,h}+hu(y_{x,t,h},t)$ and we let $\tilde y_{x,t,h}:=\Psi_{y_{x,t,h},t}(t+h)\in\partial\Omega(t+h)$.  Then again
\[
\sup_{x\in X_{u(\cdot,t)}^h[\partial\Omega(t)]} |x-\tilde y_{x,t,h}|\le h w \left(2 |h|\|u\|_{L^\infty} \right),
\]
so we obtain
\[
\lim_{h\to 0} \sup_{x\in X_{u(\cdot,t)}^h[\partial\Omega(t)]} \frac{{\rm dist} \left(x,  \partial\Omega(t+h) \right)}h=0.
\]
This proves \eqref{1.3}, and the proof is finished.
\end{proof}

\subsection{Independence from initial contour parametrization}

Next, we will show that the solution obtained in Proposition \ref{P.5.1} is independent of the chosen contour parametrization $Z_0$ for a given initial value $\omega(\cdot,0)$.
The main result of this sub-section is Corollary \ref{C.5.7}.

Hence, let us consider two families $\Omega(t) = \{\Omega_k(t)\}_{k=1}^N$ and $\tilde \Omega(t) = \{\tilde \Omega_k(t)\}_{k=1}^N$ of sets as in Definition~\ref{D.1.1}, but with $\Omega(t)$ now the sequence (rather than the union) of the $\Omega_k(t)$.  This notation will be more convenient in what follows.  We also drop the argument $t$ where we discuss results concerning a fixed time.

We start with an estimate on the area of the symmetric difference $\Omega\triangle\tilde\Omega := (\Omega \setminus \tilde \Omega) \cup (\tilde \Omega \setminus \Omega)$.  Recall the functional $\vertiii{\cdot}$ from \eqref{eq:def_norm}.

\begin{lemma}
\label{lem:l2_symm}
Let $\Omega = \{\Omega_k\}_{k=1}^N$ and $\tilde \Omega = \{\tilde \Omega_k\}_{k=1}^N$  be two families of bounded open subsets of  $\Rm^2$ whose boundaries are simple closed curves,
and let  $Z$ and $\tilde Z$ be some parametrizations of $\partial\Omega$ and $\partial\tilde \Omega$, respectively (that is, $Z=\{z_k\}_{k=1}^N$, with $z_k:\mathbb T\to \Rm^2$ a parametrization of $\partial\Omega_k$, and similarly for $\tilde Z$).  There exists a universal constant $C<\infty$ such that
\begin{equation}
\label{eq:gooal}
|\Omega \triangle \tilde \Omega| := \sum_{k=1}^N |\Omega_k \triangle \tilde\Omega_k| \leq C (\vertiii{Z}+\vertiii{\tilde Z}) \sum_{k=1}^N\|z_k-\tilde z_k\|_{L^2}.
\end{equation}
\end{lemma}

\begin{proof}
Let $A:=\vertiii{Z}+\vertiii{\tilde Z}$.  It obviously suffices to assume $A<\infty$, and to prove for each $k\in \{1,\dots, N\}$ that $|\Omega_k \triangle \tilde\Omega_k| \leq C A \|z_k-\tilde z_k\|_{L^1}$ (with some universal  $C$). We will now do this, dropping the index $k$ in the following.


We claim that  for any $x\in \Omega \triangle \tilde \Omega$, there exists some $(\xi,s) \in \mathbb{T}\times [0,1]$, such that $x= (1-s)z(\xi) + s\tilde z(\xi)$.  This is obvious for $x\in\partial\Omega\cup\partial\tilde\Omega$, so assume that $x\not\in\partial\Omega\cup\partial\tilde\Omega$.  Let $\Gamma_s(\xi) := (1-s) z(\xi) + s\tilde z(\xi)$ for $(\xi,s) \in \mathbb{T}\times [0,1]$, so that
$\Gamma_s:\mathbb T\to \Rm^2$ is a closed curve for each fixed $s\in [0,1]$, with $\Gamma_0 = z$ and $\Gamma_1 =  \tilde z$. Since $\Gamma_0$ and $\Gamma_1$ have different winding numbers with respect to  $x$ and $\Gamma$ is continuous in $(s,\xi)$, we must have $x\in\Gamma_s(\mathbb T)$ for some $s\in(0,1)$.

Consider now the quadrilateral $Q(\xi;h)$ with vertices at $z(\xi), z(\xi+h), \tilde z(\xi+h), \tilde z(\xi)$. The above discussion and both $z,\tilde z$ being $H^3$ yield
\begin{equation}
\label{eq:area_sum}
|\Omega\triangle\tilde\Omega| \leq \left| \{(1-s)z(\xi)+s\tilde z(\xi) \,:\, (\xi,s)\in\mathbb{T}\times[0,1] \}\right|
 \leq \lim_{n\to\infty} \sum_{j=0}^{n-1} \left|Q\left(\frac{2\pi j}{n}; \frac{2\pi}{n} \right)\right|.
\end{equation}
We also have, with  $C$ such that $\|f\|_{C^1}\leq C \|f\|_{H^3}$ for each $f\in H^3(\mathbb T)$,
\[
\begin{split}
|Q(\xi; h)| &\leq \max\left\{|z(\xi)-z(\xi+h)|,|\tilde z(\xi)-\tilde z(\xi+h)|\right\} \max\left\{|z(\xi)-\tilde z(\xi)|, | z(\xi+ h) - \tilde z(\xi+h)|\right\}\\
&\leq C A h \left(|z(\xi)-\tilde z(\xi)| + 2CA h\right).
\end{split}
\]
Hence the limit in \eqref{eq:area_sum} is bounded above by $C A \|z-\tilde z\|_{L^1}$, and an application of H\"older inequality leads to \eqref{eq:gooal}.
\end{proof}

\begin{definition}
For two families $\Omega = \{\Omega_k\}_{k=1}^N$ and $\tilde \Omega = \{\tilde \Omega_k\}_{k=1}^N$ of subsets of $\Rm^2$,
we define the {\it Hausdorff distance} of their boundaries to be
\[
d_H(\partial\Omega, \partial \tilde\Omega) := \max_{1\leq k\leq N}\max \left\{  \sup_{x\in \partial \Omega_k}\inf_{y\in \partial \tilde\Omega_k} |x-y|, \sup_{x\in \partial\tilde \Omega_k}\inf_{y\in \partial\Omega_k} |x-y| \right\}.
\]
\end{definition}

Next we prove that if we solve the contour equation \eqref{eq:contour2} with two families of $H^3$ initial curves which parametrize the same simple closed curves $\partial \Omega(0) := \{\partial \Omega_{k}(0)\}_{k=1}^N$, then the solutions parametrize the same curves $\partial \Omega(t) := \{\partial \Omega_{k}(t)\}_{k=1}^N$ throughout their common interval of existence.

\begin{proposition}
\label{prop:equal}
Let $\alpha \in (0,\frac{1}{24})$, let $\theta_1,\dots,\theta_N\in\Rm$, and let $Z,\tilde Z$ be both as $Z$ in Theorem \ref{thm:local}, with  initial conditions $Z_0,\tilde Z_0$, respectively.  Let $T'>0$ be the smaller of their maximal times of existence and for $t\in[0,T')$, let  $\Omega_k(t)$ and $\tilde \Omega_k(t)$ be the interiors of the contours $z_k(\cdot,t)$ and $\tilde z_k(\cdot,t)$, respectively.  If $\Omega_k(0)=\tilde \Omega_k(0)$ for each $k$, then $\Omega_k(t)=\tilde \Omega_k(t)$ for each $k$ and $t\in[0,T')$.
%
%
%
\end{proposition}

{\it Remark.}  Here   $T'$ is largest such that $\sup_{t\in[0,T]} (\vertiii{Z(t)}+\vertiii{\tilde  Z(t)})<\infty$ for each $T<T'$.

\begin{proof}
Due to the uniqueness claim in Theorem \ref{thm:local}, it suffices to prove this for the smaller of the $T>0$ (from the theorem) for $Z$ and $\tilde Z$, instead of for $T'$.  We then have $\sup_{t\in[0,T]} (\vertiii{Z(t)}+\vertiii{\tilde  Z(t)})\le 4 (\vertiii{Z(0)}+\vertiii{\tilde Z(0)})=:A$.

Our strategy here is to first prove the claim for a family of regularized equations, and then show that the solutions of the latter converge to those of the original equation (in an appropriate sense) as their parameter $\beta\to 0$.

Specifically, for any $\beta>0$, we regularize the Biot-Savart law in \eqref{eq:velocity_law} to
\begin{equation}
u^\beta(x, t) =  \int_D \left(\frac{(x-y)^\perp}{\left(|x-y|^2+\beta^2\right)^{1+\alpha}} -
\frac{(x-\bar y)^\perp}{\left(|x-\bar y|^2+\beta^2\right)^{1+\alpha}}\right) \omega(y,t) dy
\label{eq:velocity_law_eps}
\end{equation}
for $x \in \bar D$.  For \eqref{sqg} with $u=u^\beta$, following the same derivation as in Section \ref{sec:derivation}, we obtain the contour equation
\begin{equation}
\partial_t \z_k(\xi,t) = \sum_{i=1}^N \sum_{m=1}^2 \frac{\theta_i}{2\alpha}
\int_{\mathbb{T}}
 \frac{\partial_\xi \z_k(\xi,t)
- \partial_\xi y_i^m(\xi - \eta, t)}{\left(|\z_k(\xi,t)  - y_i^m(\xi-\eta, t)|^{2}+\beta^2\right)^{\alpha}}  d\eta
\label{eq:ct_eps}
\end{equation}
instead of \eqref{eq:contour2}.  Then  the same derivation as in Section \ref{sec:apriori} again yields the a priori estimate \eqref{eq:apriori} for the solutions to \eqref{eq:ct_eps}, with the same ($\beta$-independent) constant.
We can then use the same arguments as in Section \ref{sec:lwp_h3} to show that there exist unique (local) solutions $Z^\beta$ and $\tilde Z^\beta$ to \eqref{eq:ct_eps} with initial data $Z(0)$ and $\tilde Z(0)$, respectively,  and they again satisfy $\sup_{t\in[0,T]} (\vertiii{Z^\beta(t)}+\vertiii{\tilde  Z^\beta(t)})\le A$, for the above ($\beta$-independent) time $T$.

If now $\Omega_k^\beta(t)$ and $\tilde \Omega_k^\beta(t)$ are the interiors of the contours $z_k^\beta(\cdot,t)$ and $\tilde z_k^\beta(\cdot,t)$, respectively, then we can show as in Proposition \ref{P.5.1} that $\omega^\beta(\cdot,t):=\sum_{k=1}^N \theta_k \chi_{\Omega_k^\beta(t)}$ and $\tilde\omega^\beta(\cdot,t):=\sum_{k=1}^N \theta_k \chi_{\tilde\Omega_k^\beta(t)}$ are $H^3$ patch solutions to \eqref{sqg} with $u=u^\beta$ on $[0,T]$.
(Note also that since $u^\beta$ is smooth,
$\Phi_t(x)$ from \eqref{eq:alpha} is uniquely defined for any $(x,t)\in \bar D\times[0,T]$.)

One can now apply  standard estimates  for the 2D Euler equation (see, e.g., \cite[Theorems 8.1 and 8.2]{mb})
to \eqref{sqg} with the (smooth) velocity $u=u^\beta$ to show that there exists a unique weak solution
to it on $\Rm^2\times [0,\infty)$ with initial data $\omega^\beta(\cdot,0)$ ($=\tilde \omega^\beta(\cdot,0)$) extended oddly to $\Rm\times\Rm^-$.  Since $H^3$ patch solutions are also weak solutions (see Remark 3 after Definition \ref{D.1.1}), and obviously remain such when extended oddly to $x\in \Rm\times\Rm^-$, it follows that for any $\beta>0$ we have $\Omega_k^\beta(t)=\tilde \Omega_k^\beta(t)$ for each $k$ and $t\in[0,T]$.  (See the remark after this proof for an alternative argument.)

Next, we claim that with $\partial \Omega^\beta(t) := \{ \partial \Omega_k^\beta(t)\}_{k=1}^N$ we have
\begin{equation}
\label{eq:dh_beta}
\lim_{\beta\to 0} \sup_{t\in[0,T]}d_H(\partial \Omega^\beta(t), \partial \Omega(t)) = 0.
\end{equation}
Since the same result then holds with $\tilde \Omega$ in place of $\Omega$, this proves the proposition.
The key step in showing \eqref{eq:dh_beta} is the estimate
\begin{equation}
\label{eq:beta_diff}
\sup_{t\in[0,T]} \|Z^\beta(t) - Z(t)\|_{L^2} \leq  C(\alpha, N \Theta, A,T)\beta,
\end{equation}
with $\Theta:=\sum_{k=1}^N |\theta_k|$.
Then Lemma \ref{lem:l2_symm} yields $\sup_{t\in[0,T]} |\Omega^\beta(t)\triangle  \Omega(t)| \leq C(\alpha, N, \Theta, A) \beta$ for each $\beta>0$,  which together with the uniform $H^3$ bound on $Z,Z^\beta$ implies \eqref{eq:dh_beta}.  It therefore remains to prove \eqref{eq:beta_diff}.

We let $W:=Z- Z^\beta$ so that (after dropping the argument $t$)
\[
\frac d{dt}\|w_k\|_{L^2}^2 = 2 \int_{\mathbb{T}} w_k(\xi) \cdot \partial_t w_k(\xi) d\xi
= G_k +  \sum_{i=1}^N \sum_{m=1}^2  \frac{\theta_i}{\alpha} S_{k,i}^m,
\]
where $G_k$ is the right hand side of \eqref{eq:temp2} with $\tilde Z$ is replaced by $Z^\beta$, while  $S_{k,i}^m$ equals
\[
\hspace*{-0.5cm}
\int_{\mathbb{T}^2}w_k(\xi)\cdot \left[\partial_\xi \z_k^\beta(\xi)
- \partial_\xi y_i^{m,\beta}(\xi - \eta)\right]
\left[ \frac{1}{|\z_k^\beta(\xi)  - y_i^{m,\beta}(\xi-\eta)|^{2\alpha}} - \frac{1}{(|\z_k^\beta(\xi)  - y_i^{m,\beta}(\xi-\eta)|^{2}+\beta^2)^{\alpha}}  \right] d\eta d\xi.
\]
Note that the same derivation as in the uniqueness part of Section \ref{sec:lwp_h3} yields
\[
G_k \leq  C(\alpha)\Theta (\vertiii {Z(t)}+\vertiii { \Z^\beta (t)})^{2+2\alpha} \|W( t)\|_{L^2}^2 \leq C(\alpha)\Theta  A^{2+2\alpha}  \|W( t)\|_{L^2}^2.
\]

To control $S_{k,i}^m$, we first use that for any $x,\beta > 0$ and $\alpha \in [0,1]$ we have
\begin{equation}
\label{eq:temp4}
\left| \frac{1}{x^{2\alpha}} - \frac{1}{(x^2+\beta^2)^\alpha} \right| = x^{-2\alpha} \left| 1- \left(1+\left(\frac{\beta}{x}\right)^2\right)^{-\alpha} \right| \leq x^{-1-2\alpha} \beta,
\end{equation}
where in the last inequality we used that $|1-(1+b^2)^{-\alpha}| \leq b$ for $b>0$ and $\alpha \in [0,1]$.  Applying \eqref{eq:temp4} now yields
\begin{equation}
\label{eq:temp5}
\begin{split}
|S_{k,i}^m| &\leq  \beta \int_{\mathbb{T}^2} | w_k(\xi)| \frac{|\partial_\xi \z_k^\beta(\xi)
- \partial_\xi y_i^{m,\beta}(\xi - \eta)|}{|\z_k^\beta(\xi) - y_i^{m,\beta}(\xi - \eta)|^{1+2\alpha}} d\eta d\xi \\
& \leq \sqrt{2\pi}\beta \|w_k\|_{L^2} \underbrace{ \sup_{\xi\in \mathbb{T}} \int_{\mathbb{T}} \frac{|\partial_\xi \z_k^\beta(\xi)
- \partial_\xi y_i^{m,\beta}(\xi - \eta)|}{|\z_k^\beta(\xi) - y_i^{m,\beta}(\xi - \eta)|^{1+2\alpha}}d\eta}_{=: J}.
\end{split}
\end{equation}
For $i\neq k$, we immediately have $J \leq 4\pi \|Z^\beta\|_{C^1} \delta[Z^\beta]^{-1-2\alpha}$.
For $i=k$ and $m=1$, we have
\[
J \leq \int_{\mathbb{T}} \|Z^\beta\|_{C^2} F[Z^\beta]^{1+2\alpha} \eta^{-2\alpha} d\eta \leq C(\alpha) \|Z^\beta\|_{C^2} F[Z^\beta]^{1+2\alpha}
\]
for $\alpha<\frac 12$.  When $i=k$ and $m=2$, then the integrand in $J$ is the same as $T_3$ in \eqref{eq:temp_i0}, only with $\eta$ replaced by $\xi-\eta$.  Hence the same bound as in \eqref{eq:temp6} gives
\[
J \leq  \int_{\mathbb{T}} C (\|Z^\beta\|_{C^2}+1) F[Z^\beta]^{1+2\alpha}\eta^{-2\alpha-1/2}d\eta \leq C(\alpha) (\|Z^\beta\|_{C^2}+1) F[Z^\beta]^{1+2\alpha}
\]
for $\alpha<\frac 14$.
It now follows that
\[
|S_{k,i}^m|  \leq C(\alpha)  \vertiii{Z^\beta}^{2+2\alpha} \beta \|w_k\|_{L^2} \leq C(\alpha) A^{2+2\alpha} \beta \|w_k\|_{L^2}  .
\]

Combining the estimates for $G_k$ and $S_{k,i}^m$ gives
\[
\frac{d}{dt}\|W(t)\|_{L^2} \leq C(\alpha)N\Theta A^{2+2\alpha} (\|W(t)\|_{L^2} + \beta)
\]
for $t\in[0,T]$.
Solving this differential inequality with initial condition  $\|W(0)\|_{L^2} = 0$ yields $\|W(t)\|_{L^2} \leq (e^{C(\alpha)N\Theta A^{2+2\alpha} t }-1)\beta$ for $t\in[0,T]$, which implies \eqref{eq:beta_diff}.
\end{proof}

{\it Remark.}  We note that one can in fact prove $\Omega_k^\beta(t)=\tilde \Omega_k^\beta(t)$ (even uniqueness of $C^1$ patch solutions to \eqref{sqg} with $u=u^\beta$) for any $\beta>0$ without a reference to weak solutions.  Indeed, Lemma \ref{lem:udiff} below holds with power 1 (instead of $1-2\alpha$) in \eqref{eq:u_goal}, which follows from the first paragraph of its proof because $u^\beta$ is clearly smooth and \eqref{eq:utildeu} now holds with $(|x-y|^2+\beta^2)^{-\alpha-1/2}$ ($\le \beta^{-1-2\alpha}$) inside the integral. (The constant in \eqref{eq:u_goal} then becomes $C=C(\alpha,\beta, \sum_{k=1}^N|\theta_k|, |\partial\Omega(t)|)<\infty$.)  The argument in Lemma \ref{lem:onestep} below then yields $d'(t)\le Cd(t)$ for $ d(t):=d_H(\partial \Omega(t),  \partial \tilde \Omega(t))$ (as long as $d(t)\le 1$), so $d(0)=0$ implies $d(t)=0$ for all $t\in[0,T]$.

\begin{definition} \label{D.5.6}
For a family $\Omega = \{\Omega_k\}_{k=1}^N$ of bounded open subsets of $D$ whose boundaries $\partial\Omega_k$ are pairwise disjoint simple closed $H^3$ curves (i.e., $\|\Omega_k\|_{H^3}<\infty$ for each $k$), let us define $\vertiii{\Omega}_{H^3} := \vertiii{Z}$, where  $Z = \{z_k\}_{k=1}^N$ and each  $z_k$ is  a constant speed  parametrization of $\partial \Omega_k$ as in Definition \ref{D.1.0}.
\end{definition}

{\it Remarks.}  1. Since  $z_k\in H^3(\mathbb T)$ has constant speed and $\mathbb T$ is compact, it is not difficult to see that 
any $\Omega$ as in the definition  satisfies $\vertiii{\Omega}_{H^3}<\infty$. Indeed, looking at \eqref{eq:def_norm}, the only term that is not clearly finite due to the assumptions of the definition 
is $F[Z]$ from \eqref{F_def}.  It is clear that the constant speed of parametrization and $\|\Omega_k\|_{H^3}<\infty$  show that there exists $r>0$ such that if $|\eta|\in(0,r),$ then 
$\eta|z_k(\xi+\eta)-z_k(\xi)|^{-1} \leq 2L_k^{-1}$ for all $\xi,k.$ For $|\eta|\ge r$ and any $\xi,k$, the expression on the right hand side of \eqref{F_def} is bounded due to its continuity, compactness
of $\mathbb T$, and the assumption that all the $\partial \Omega_k$ are simple closed curves. 

2.  Similarly,  any $H^3$ patch solution $\omega(\cdot,t)=\sum_{k=1}^N \theta_k\chi_{\Omega_k(t)}$ to \eqref{sqg}-\eqref{eq:velocity_law} on $[0,T)$ must satisfy $\sup_{t\in[0,T']} \vertiii{\Omega(t)}_{H^3}<\infty$ for any $T'<T$
 (due to continuity of $\Omega$ in time and compactness of $\mathbb T\times[0,T']$).

Here is a corollary that summarizes the previous results in this section.

\begin{corollary} \label{C.5.7}
Let $\alpha\in(0,\frac 1{24})$, let $\theta_1,\dots,\theta_N\in\Rm\setminus\{0\}$, and let $\Omega(0) = \{\Omega_k(0)\}_{k=1}^N$ be as in Definition \ref{D.5.6}.
There exists $T_\omega>0$ and an $H^3$ patch solution $\omega(\cdot,t)=\sum_{k=1}^N \theta_k\chi_{\Omega_k(t)}$ to \eqref{sqg}-\eqref{eq:velocity_law} on $[0,T_\omega)$ which satisfies the following.
\begin{enumerate}[(a)]
\item
For any $T'\in[0,T_\omega)$ and any parametrization $ Z(T')$ of $\partial\Omega(T')$ with $\vertiii{ Z(T')}<\infty$, let $T=T(\alpha,N\sum_{k=1}^N|\theta_k|,\vertiii{Z(T')})>0$ be from Theorem \ref{thm:local}.  Then the corresponding $H^3$ patch solution  to \eqref{sqg}-\eqref{eq:velocity_law}  on $[T',T'+T]$ from Proposition \ref{P.5.1}, with initial value $Z(T')$ at  time $T'$, is equal to $\omega$ on the time interval $[T',T'+T]$.
\item
There is a universal $C\ge 1$ such that with $T=T(\alpha,N\sum_{k=1}^N|\theta_k|,\vertiii{\Omega(0)}_{H^3})>0$ from Theorem \ref{thm:local} we have $T_\omega\ge T$ and $\sup_{t\in[0,T]} \vertiii{\Omega(t)}_{H^3}\le C \vertiii{\Omega(0)}_{H^3}^8$.
\item
If $T_\omega<\infty$, then $\lim_{t\nearrow T_\omega} \vertiii{\Omega(t)}_{H^3}=\infty$.
\end{enumerate}
\end{corollary}
Basically, what the Corollary says is that the patch solution that we can obtain using the contour equation is unique. We cannot obtain different patch solutions by changing the parametrization of the initial data or at any other time. 
Also this solution may cease to exist only if its norm $\vertiii{\Omega(t)}_{H^3}$ blows up. On the other hand, the corollary does not rule out existence of other $H^3$ patch solutions with the same initial data, obtained not from the contour equation but in some 
other way. We will eliminate this possibility in the next section. 
\begin{proof}
Let $Z_0$ be a constant speed parametrization of $\partial\Omega(0)$ (which satisfies $\vertiii{Z_0}=\vertiii{\Omega(0)}_{H^3}<\infty$ due to Remark 1 after Definition \ref{D.5.6})
and consider the solution $\omega$ from Proposition \ref{P.5.1} on $[0,T_0]$, with $T_0:=T(\alpha,N\sum_{k=1}^N|\theta_k|,\vertiii{\Omega(0)}_{H^3})>0$ from Theorem \ref{thm:local}.  Then \hbox{$\sup_{t\in[0,T_0]} \vertiii{\Omega(t)}_{H^3}\le C(4\vertiii{\Omega(0)}_{H^3})^8$,} with $C$ from Lemma \ref{lem:arcl_bd},  so (b) holds with $T:=T_0$.  We can also extend $\omega$ up to time $T_0+T_1$, with $T_1:=T(\alpha,N\sum_{k=1}^N|\theta_k|,\vertiii{\Omega(T_0)}_{H^3})>0$, by  using Proposition \ref{P.5.1} with initial condition a constant speed parametrization of $\partial\Omega(T_0)$ at time $T_0$.  We can continue this way to obtain $T_\omega:=\sum_{j=0}^\infty T_j$, which then must satisfy either $\lim_{t\nearrow T_\omega} \vertiii{\Omega(t)}_{H^3}=\infty$ or $T_\omega=\infty$.   This proves (c), while (a) follows from Proposition \ref{prop:equal} (note also that $T$ in (a) must be less than $T_\omega-T'$ because $\sup_{t\in[T',\min\{T'+T,T_\omega\})} \vertiii{\Omega(t)}_{H^3
 }<\infty$ by Proposition \ref{P.5.1}).
\end{proof}

{\it Remark.} 
Now is the natural time to prove the last statement of Theorem~\ref{T.1.1} describing precisely how the blow up may manifest itself. 
Let us assume that $T_\omega<\infty$, and define $\partial\Omega(T_\omega):=\lim_{t\nearrow T_\omega} \partial\Omega(t)$  (the limit is taken with respect to Hausdorff distance and exists due to \eqref{uLinfty}).
Let us also assume that  $\min_{k\neq i}{\rm dist}(\partial\Omega_k(T_\omega),\partial\Omega_i(T_\omega))>0$ and that for each $k$, the limit $\lim_{t\nearrow T_\omega}  \|\Omega_k(t)\|_{H^3}$ is either finite or does not exist.  Then there must be some $k$ and $t_j\nearrow T_\omega$ such that for any constant speed parametrization $z_j$ of some $\partial \Omega_k(t_j)$ we have
\[
A:=\sup_j  \|z_j\|_{H^3}<\infty \qquad\text{and}\qquad \lim_{j\to\infty} \sup_{\xi,\eta\in \mathbb{T}, \eta\neq 0}\dfrac{|\eta|}{|\z_j(\xi) - \z_j(\xi-\eta)|} =\infty.
\]
The first of these statements, together with $|\partial\Omega_k(t)|$ being bounded below uniformly in $t$, due to $|\Omega_k(t)|$ being constant in time, and
along with the constant speed property of $z_j$, implies that there is $r=r(A)$ such that
\[
\sup_{j} \sup_{\xi,\eta\in \mathbb{T}, |\eta|\in(0,r)}\dfrac{|\eta|}{|\z_j(\xi) - \z_j(\xi-\eta)|} <\infty.
\]
Thus there must be two sequences of points $x_j,y_j\in\partial\Omega_k(t_j)$ with $\lim_{j\to\infty} |x_j-y_j|=0$ but distance of $x_j$ and $y_j$ along $\partial\Omega_k(t_j)$ uniformly bounded below by a positive number.  Continuity of $\partial\Omega_k$ in time (and its compactness) then implies that $\partial\Omega_k(T_\omega)$ cannot be a simple closed curve. This proves the last statement in Theorem \ref{T.1.1}.

\subsection{Uniqueness of $H^3$ patch solutions}

We will now prove (local) uniqueness of $H^3$ patch solutions to \eqref{sqg}-\eqref{eq:velocity_law}.  Hence the unique solution for a given initial value $\omega(\cdot,0)$ is the one from Corollary \ref{C.5.7}.  The main result of this sub-section is Theorem \ref{thm:unique_patch}.

The next lemma is a simple geometric result, concerning two $H^3$ patches whose boundaries are close to each other in Hausdorff distance.  It will be used in the following lemma to estimate the difference of the velocities from \eqref{eq:velocity_law} corresponding to two sets of $H^3$ patches whose boundaries are close to each other in Hausdorff distance.  As before, we denote by $n_P$ the outer unit normal vector for $\Omega$ at $P\in \partial \Omega$.

\begin{lemma}
\label{lem:func_diff}
Let $\Omega, \tilde \Omega \subseteq \mathbb{R}^2$ be two bounded open sets whose boundaries are simple closed curves, and let $\vertiii{\Omega}_{H^3}+\vertiii{\tilde \Omega}_{H^3}\le A$ for some $A\ge 1$.
Let $R := (4 C_0 A)^{-3}$, where $C_0\geq 1$ is a universal constant such that $\|f\|_{C_2} \leq C_0\|f\|_{H^3}$ for each $f\in H^3 (\mathbb{T})$, and let $P\in\partial\Omega$.
If $d_H(\partial\Omega, \partial\tilde\Omega) \le  \frac R{20}$,  then in the coordinate system $(w_1, w_2)$ centered at $P$ and with axes $n_P^{\perp}$ and $n_P$, both $\partial \Omega \cap B(P,R)$ and $\partial \tilde \Omega \cap B(P,\frac{19}{20}R)$ can be represented as graphs $w_2 = f(w_1)$ and $w_2 = g(w_1)$, respectively, and we have $|f'(w_1)|\leq 1$, $|g'(w_1)|\leq 1$, and $|f(w_1) - g(w_1)| \leq 2 d_H(\partial\Omega, \partial\tilde\Omega)$ for $|w_1| \leq \frac R2$.
\end{lemma}

\begin{proof}
Let $h:=d_H(\partial\Omega, \partial\tilde\Omega)$,
 let $\tilde P \in \partial \tilde \Omega$ be such that $|\tilde P-P| = {\rm dist}(P, \partial\tilde \Omega) \leq h$, and denote by $\tilde n_{\tilde P}$ the outer unit normal for $\tilde \Omega$ at $\tilde P$. (If  $\tilde P$ is not unique, we pick one such point.)  By Lemma \ref{lem:regularity}(a,c) with $\gamma=1$ and the definition of $R$ (note that if either $\partial\Omega$ or $\partial\tilde \Omega$ has arc-length $2\pi L$ and a constant speed parametrization $z$, then $L\ge \tfrac 1\pi |z(\pi)-z(0)|\ge F[z]^{-1}\ge \frac 1A$),  both $\partial \Omega \cap B(P,R)$ and $\partial \tilde \Omega \cap B(\tilde P,R)$ are simply connected curves whose (outer to $\Omega$ and $\tilde\Omega$) unit normal vectors lie in $B(n_P,\frac 1{32})$ and in $B(\tilde n_{\tilde P},\frac 1{32})$, respectively.

This implies that $\tilde n_{\tilde P} \cdot n_P \geq \cos \frac \pi 6$. Indeed, otherwise we could find $P' \in \partial \Omega\cap \partial B(P, \frac R2)$ such that ${\rm dist}(P', \partial \tilde \Omega) \geq \frac{R}{2} \sin(\frac{\pi}{6} - 2\cdot \arcsin\frac{1}{32}) - h > h$ (since we assume $h \leq \frac R{20}$), contradicting
  $d_H(\partial \Omega, \partial \tilde \Omega) = h$.

From $|\tilde P-P|\le h\le\frac R{20}$ and $\tilde n_{\tilde P} \cdot n_P \geq \cos \frac \pi 6$ (together with the normal vectors of $\partial \Omega \cap B(P,R)$ and $\partial \tilde \Omega \cap B(\tilde P,R)$ lying in $B(n_P,\frac 1{32})$ and in $B(\tilde n_{\tilde P},\frac 1{32})$, respectively), we have that in the coordinate system $(w_1, w_2)$, both $\partial \Omega \cap B(P,R)$ and $\partial \tilde \Omega \cap  B(P,\frac {19}{20}R)$
are graphs  $w_2 = f(w_1)$ and $w_2 = g(w_1)$, respectively, with $|f'(w_1)| \leq \tan\arcsin \frac 1{32}<1$ and $|g'(w_1)| \leq \tan(\frac \pi 6+\arcsin \frac 1{32}) < 1$.  Since $19^2>12^2+13^2$, it follows that the domains of $f,g$ both contain $[-\frac{12}{20}R,\frac{12}{20}R]$.

If now $|f(w_1) - g(w_1)| > 2 h$ for some $|w_1| \leq \frac R2$, then  $Q:=(w_1,f(w_1))\in\partial\Omega\cap[-\frac R2,\frac R2]^2$ has ${\rm dist}(Q, \Rm^2\setminus [-\frac {12}{20}R,\frac {12}{20}R]^2)\ge 2h$ and also ${\rm dist}(Q, \partial\tilde\Omega\cap [-\frac {12}{20}R,\frac {12}{20}R]^2)\ge \sqrt 2h$ (the latter because $|g'(w_1)|\le 1$ for $|w_1|\le\frac{12}{20}R$).  This again contradicts $d_H(\partial \Omega, \partial \tilde \Omega) = h$.
%
%
%
\end{proof}

\begin{lemma}
\label{lem:udiff}
Let $\alpha \in (0,\frac{1}{2})$, let $\theta_1,\dots\theta_N\in\Rm$, let $\Omega = \{\Omega_k\}_{k=1}^N$ and $\tilde \Omega = \{\tilde \Omega_k\}_{k=1}^N$ be as in Definition~\ref{D.5.6}, and let $u$ and $\tilde u$ be the velocity fields from \eqref{eq:velocity_law} corresponding to $\omega:= \sum_{k=1}^N \theta_k \chi_{\Omega_k}$ and $\tilde \omega:= \sum_{k=1}^N \theta_k \chi_{\tilde \Omega_k}$, respectively.  Let also $\vertiii{\Omega}_{H^3}+\vertiii{\tilde \Omega}_{H^3}\le A$ for some $A\ge 1$.
There exists a constant $C=C(\alpha, \sum_{k=1}^N|\theta_k|, A)<\infty$ such that if $d_H(\partial \Omega, \partial \tilde \Omega) \leq 1$, then for any $x,\tilde x \in \bar D$
we have
\begin{equation}
\label{eq:u_goal}
|u(x) - \tilde u(\tilde x)| \leq C \max\{|x-\tilde x|,d_H(\partial \Omega, \partial \tilde \Omega)\}^{1-2\alpha}.
\end{equation}
\end{lemma}

\begin{proof}
First note that \eqref{uHold} shows that it is sufficient to consider $\tilde x=x$.  From \eqref{uLinfty} it follows that it further suffices to restrict the proof to the case $h:=d_H(\partial \Omega, \partial \tilde \Omega)\le \frac R{20}$, with $R=R(A)$ from Lemma \ref{lem:func_diff}.
We then have
\begin{equation}
\label{eq:utildeu}
|u(x) - \tilde u(x)| \leq \sum_{k=1}^N 2|\theta_k|\underbrace{\int_{\Omega_k\triangle \tilde\Omega_k} |x-y|^{-1-2\alpha} dy}_{=:I_k},
\end{equation}
so it finally suffices to show $I_k\le C(\alpha,A) h^{1-2\alpha}$ for each $k$ and some $C(\alpha,A)<\infty$.

Let $P\in \partial \Omega_k$ be such that $|x-P| = d(x,\partial \Omega_k) =: d_k$.
Let us first assume that $d_k\geq \frac R4$. Since $\partial \Omega_k$ and $\partial \tilde \Omega_k$ both have arc-length bounded by $CA$ (for some universal $C<\infty$) and $A\ge 1$, we have $|\Omega_k \triangle \tilde\Omega_k| \leq CAh$ (with a different universal $C$).
Thus
\[
d(x,\Omega_k \triangle \tilde\Omega_k) \geq \frac{R}{4} - h \geq \frac{R}{5}
\]
because $h\leq \frac R{20}$. Since then also $h\le 1$, this indeed yields
\begin{equation}
\label{eq:tempIk}
I_k \leq  |\Omega_k \triangle \tilde\Omega_k|  \left(\frac{R}{5}\right)^{-1-2\alpha} \leq C(\alpha, A) h \leq C(\alpha,  A) h^{1-2\alpha}.
\end{equation}

Let us now assume that $d_k < \frac R4$, and split $I_k$ into
\[
I_k = \underbrace{\int_{(\Omega_k\triangle \tilde\Omega_k) \cap B(P,R/2)} |x-y|^{-1-2\alpha} dy}_{=: I_k^1} + \underbrace{\int_{(\Omega_k\triangle \tilde\Omega_k) \cap (D\setminus B(P,R/2)) } |x-y|^{-1-2\alpha} dy}_{=: I_k^2}.
\]
If $y \notin B(P,\frac R2)$, then $|x-y| \geq |y-P| - |P-x| \geq \frac{R}{2} - \frac{R}{4} \geq \frac{R}{4}$.  Hence $I_k^2$ can be bounded as $I_k$ in \eqref{eq:tempIk}, yielding $I_k^2 \leq C(\alpha,  A) h^{1-2\alpha}$.

To bound $I_k^1$, we apply Lemma~\ref{lem:func_diff} to $\Omega_k$ and $\tilde \Omega_k$. Thus in the coordinate system $(w_1, w_2)$ centered at $P$ and with axes $n_P^\perp$ and $n_P$ (the latter being the outer unit normal for $\Omega_k$ at $P$), both $\partial \Omega_k \cap B(P,R)$ and $\partial \tilde\Omega_k \cap B(P,\frac{19}{20}R)$ are graphs  $w_2 = f(w_1)$ and $w_2 = g(w_1)$, respectively, such that $|f(w_1) - g(w_1)| \leq 2h$ for $|w_1| \leq \frac R2$.
In this new coordinate system, $x$ is either $(0,d_k)$ or $(0,-d_k)$, and we can assume the former without loss.  Then
\[
\begin{split}
I_k^1 &\leq \int_{-R/2}^{R/2} \underbrace{\left |\int_{f(w_1)}^{g(w_1)} \left(w_1^2 + (w_2-d_k)^2\right)^{-\frac{1}{2}-\alpha} dw_2 \right|}_{=:T(w_1)} dw_1.
\end{split}
\]
For $|w_1| \geq h$ we have
\[
T(w_1) \leq |g(w_1)-f(w_1)| |w_1|^{-1-2\alpha} \leq 2h |w_1|^{-1-2\alpha},
\]
whereas for $|w_1| < h$ we have
\[
\begin{split}
T(w_1) &\leq \left| \int_{f(w_1)}^{g(w_1)} |w_1|^{-\frac{1}{2}-\alpha} |w_2-d_k|^{-\frac{1}{2}-\alpha} dw_2 \right| \leq 2|w_1|^{-\frac{1}{2}-\alpha} \int_0^{h} s^{-\frac{1}{2}-\alpha}ds
\le \frac 4{1-2\alpha} |w_1|^{-\frac{1}{2}-\alpha} h^{\frac{1}{2}-\alpha}.
\end{split}
\]
It follows that
\[
I_k^1 \leq  2\int_h^{R/2} 2h w_1^{-1-2\alpha} dw_1   
+ 2\int_{0}^h \frac 4{1-2\alpha} w_1^{-\frac{1}{2}-\alpha} h^{\frac{1}{2}-\alpha} dw_1  \leq C(\alpha, A) h^{1-2\alpha}.
\]
%
%
So we again have $I_k \leq C(\alpha, A) h^{1-2\alpha}$, and the proof is finished.
\end{proof}

We will now use Lemma \ref{lem:udiff} to show that the Hausdorff distance of two $H^3$ patch solutions with the same initial data grows (for a short time) at most as $t^{\frac{1}{2\alpha}}$.

\begin{lemma}
\label{lem:onestep}
Let $\alpha \in (0,\frac{1}{2})$, let $\omega(\cdot,t)= \sum_{k=1}^N \theta_k \chi_{\Omega_k(t)}$ and $\tilde \omega(\cdot,t) = \sum_{k=1}^N \theta_k \chi_{\tilde \Omega_k(t)}$ be two $H^3$ patch solutions to  \eqref{sqg}-\eqref{eq:velocity_law} on $[0,T]$, and let ${\rm sup}_{t \in [0,T]} ( \vertiii{\Omega(t)}_{H^3}+ \vertiii{\tilde \Omega(t)}_{H^3}) \leq A$ for some $A\geq 1$, where $ \Omega(t) := \{  \Omega_k(t)\}_{k=1}^N$ and $\tilde \Omega(t) := \{ \tilde \Omega_k(t)\}_{k=1}^N$. There is a constant $C=C(\alpha, \sum_{k=1}^N|\theta_k|, A)<\infty$ such that if  $\omega(\cdot,0) = \tilde \omega(\cdot,0)$, then for all $t\in[0,\min\{C^{-2\alpha},T\}]$ we have
\begin{equation}
\label{eq:d_bd}
 d_H\left(\partial \Omega(t),  \partial \tilde \Omega(t)\right) \leq C t^{1/{2\alpha}}.
\end{equation}
\end{lemma}

\begin{proof}
Let $u$ and $\tilde u$ be the velocity fields from \eqref{eq:velocity_law} corresponding to $\omega$ and $\tilde \omega$, respectively.  Let $C=C(\alpha, \sum_{k=1}^N|\theta_k|, A)$ be the constant from Lemma \ref{lem:udiff} and let $d(t) := d_H(\partial \Omega(t),  \partial \tilde \Omega(t))$.

We first claim that $d$ is Lipschitz on $[0,T]$, with some constant $\tilde C=\tilde C(\alpha, \sum_{k=1}^N|\theta_k|, A)<\infty$ which is  three times the right-hand side of \eqref{uLinfty}.  It is obviously sufficient to prove this for any $[a,b]\subseteq (0,T)$.  From \eqref{1.3} and \eqref{uLinfty} we have that for each $t\in[a,b]$ there is $h_t>0$ such that  $|d(t+h)-d(t)|\le \tilde Ch$ whenever $|h|<h_t$.  Thus we also have $|d(t)-d(s)|\le \tilde C|t-s|$ whenever $(t-h_t,t+h_t)\cap(s-h_s,s+h_s)\neq\emptyset$.  Since there is a finite sub-cover of $[a,b]$ from $\{(t-h_t,t+h_t)\}_{t\in[a,b]}$, it follows that $d$ is indeed $\tilde C$-Lipschitz on $[a,b]$.  It follows that $d$ is differentiable almost everywhere on $[0,T]$ and $d(t)-d(0)=\int_0^t d'(s)ds$ for $t\in[0,T]$.

Consider any $t>0$ such that $d(t)\in(0, 1]$ and $d'(t)$ exists.  Then \eqref{1.3} shows that for small $h>0$ and any $x\in\partial\Omega(t+h)$, there is $y_x\in \partial\Omega(t)$ such that $|y_x+hu(y_x,t)-x|\le o(h)$, with $o(h)$ uniform in $x$.  Then there are also $\tilde y_x\in \partial\tilde\Omega(t)$ and $\tilde x\in\partial\tilde\Omega(t+h)$ such that $|\tilde y_x-y_x|\le d(t)$ and $|\tilde y_x+h\tilde u(\tilde y_x,t)-\tilde x|\le o(h)$ (with a new uniform $o(h)$).  This and Lemma~\ref{lem:udiff} applied to $y_x$ and $\tilde y_x$ show that
\[
|\tilde x -x|\le d(t) + Cd(t)^{1-2\alpha}h+2o(h)
\]
 Since $o(h)$ is uniform in $x\in\partial\Omega(t+h)$, and since the same argument applies to $\partial\Omega$ and $\partial\tilde\Omega$ swapped, we obtain $d'(t)\le Cd(t)^{1-2\alpha}$ for each $t$ such that $d(t)\in(0,1]$ and $d'(t)$ exists.  Integrating this differential inequality (recall that $d$ is Lipschitz and $d(0)=0$) yields $d(t)\le (4\alpha Ct)^{1/2\alpha}$ on any time interval $[0,T']$ such that $\sup_{t\in[0,T']} d(t)\le 1$.  Hence the theorem holds with (the new) $C$ being (the old) $(4\alpha C)^{1/2\alpha}$.
\end{proof}

The next lemma, which is our last ingredient for the proof of uniqueness, says that the boundaries of two $H^3$ patches $\Omega,\tilde \Omega$ have constant speed parametrizations which differ (in $L^\infty$) by no more than $O(d_H(\partial\Omega, \partial\tilde\Omega))$, with the constant depending on $\vertiii{\Omega}_{H^3}+ \vertiii{\tilde \Omega}_{H^3}$.  

\begin{lemma}
\label{lem:arcl_diff}
Let $\Omega, \tilde \Omega \subseteq \mathbb{R}^2$ be two bounded open sets whose boundaries are simple closed curves, and let $\vertiii{\Omega}_{H^3}+\vertiii{\tilde \Omega}_{H^3}\le A$ for some $A\ge 1$.  There is a universal constant $C<\infty$ such that if $d_H(\partial\Omega, \partial\tilde\Omega) \leq 1$, then there exist some constant speed parametrizations $z$ and $\tilde z$ of $\partial\Omega$ and $\partial \tilde \Omega$, respectively, such that
\begin{equation}
\label{eq:goal_arcl}
\|z-\tilde z\|_{L^\infty} \leq CA^7 d_H(\partial\Omega, \partial\tilde\Omega).
\end{equation}
\end{lemma}

We postpone the proof of Lemma~\ref{lem:arcl_diff} until after the following theorem, which is our main uniqueness result.

\begin{theorem}
\label{thm:unique_patch}
Let $\alpha\in(0,\frac 1{24})$, let $\theta_1,\dots,\theta_N\in\Rm\setminus\{0\}$, and let  $\omega(\cdot,t)= \sum_{k=1}^N \theta_k \chi_{\Omega_k(t)}$ and $\tilde \omega(\cdot,t) = \sum_{k=1}^N \theta_k \chi_{\tilde \Omega_k(t)}$ be two $H^3$ patch solutions to  \eqref{sqg}-\eqref{eq:velocity_law}  on some interval $[0,T)$.  If $\omega(\cdot,0)=\tilde\omega(\cdot,0)$, then $\omega(\cdot, t) = \tilde \omega(\cdot, t)$ for all $t\in [0,T)$.
\end{theorem}

Let us first provide an overview of the argument (see also Figure \ref{fig}). Corollary~\ref{C.5.7} shows that there is a unique patch solution (with a given initial data) which can be obtained via the contour equation. 
It is then sufficient to prove the result with $\tilde \omega$ being this solution.
We will assume that there
exists another $H^3$ patch solution $\omega\neq\tilde\omega$
and arrive at a contradiction. The idea is to use a sequence of auxiliary $H^3$ patch solutions $\{\omega^{s_j}\}_{j=1}^J$, 
obtained via Corollary~\ref{C.5.7} with initial data $\omega^{s_j}(\cdot,s_j)=\omega(\cdot,s_j)$, where $s_j=\frac jJ T_1$ and $T_1\le 1$ is a fixed time. The first key step will be to apply Lemma~\ref{lem:onestep} to show that $\omega(\cdot,s_j)$ ($=\omega^{s_j}(\cdot,s_j)$) and $\omega^{s_{j-1}}(\cdot,s_j)$ are $J^{-1/2\alpha}$ close (in Hausdorff distance of their patch boundaries, and hence their constant speed parametrizations are also $J^{-1/2\alpha}$ close due to Lemma~\ref{lem:arcl_diff}). Next, since the $\omega^{s_j}$ were obtained via the contour equation, 
the $L^2$ stability estimate \eqref{eq:l2diff} applies to them up to time $T_1$ and allows us to show that $\omega^{s_j}(\cdot,T_1)$ and $\omega^{s_{j-1}}(\cdot,T_1)$ are  $J^{-1/2\alpha}$ close as well. The latter is in terms of the $L^2$ distance of some parametrizations of the curves, but Lemma~\ref{lem:l2_symm} allows us to transfer this into the same estimate for the area of the symmetric difference of the corresponding patches.  After telescoping the latter, we find that the   area of such a symmetric difference corresponding to $\omega(\cdot,T_1)$ and $\tilde \omega(\cdot,T_1)$ is bounded above by $O(J^{1-1/2\alpha})$. Taking $J\to\infty$ and using $2\alpha<1$ (and then applying this argument to arbitrary $T_1$), we find that $\omega=\tilde\omega$, which is a contradiction with our hypothesis.  
\begin{figure}[htbp]
\begin{center}
\includegraphics[scale=1.3]{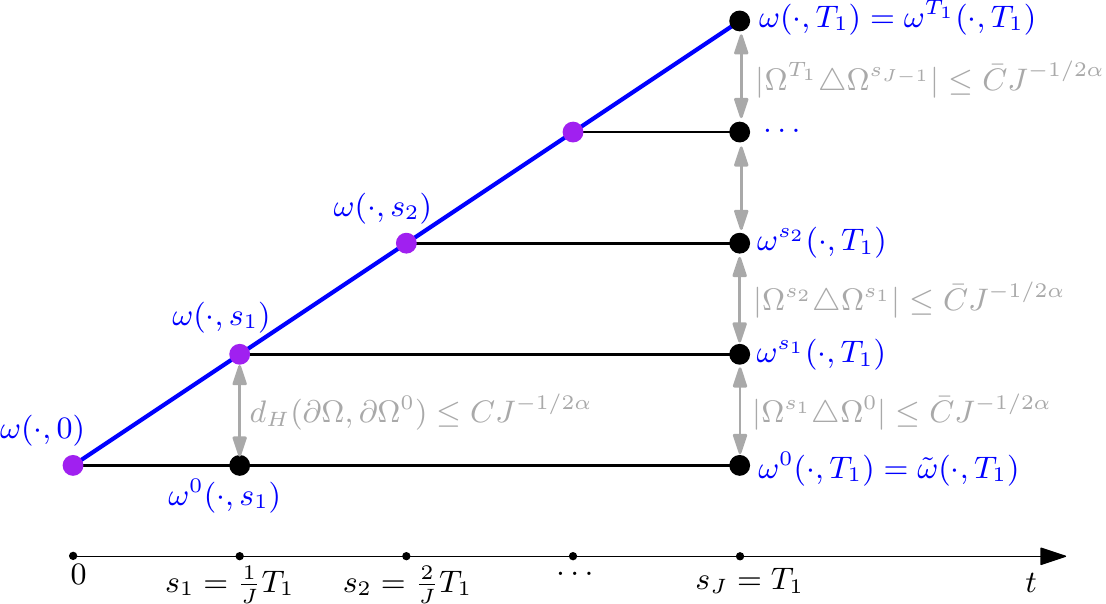}
\caption{An abstract phase space illustration of the proof of Theorem \ref{thm:unique_patch}. \label{fig}}
\end{center}
\end{figure}

\begin{proof}
%
%

%
Obviously, it suffices to prove the result in the case of $\tilde \omega$ being the solution from Corollary~\ref{C.5.7}.
Assume the contrary and let $T':=\inf \{t\in(0,T) \,:\, \omega(\cdot, t) \neq \tilde \omega(\cdot, t)\}<T$.  Without loss we can assume that $T'=0$.

With the notation from Definition \ref{D.5.6}, let $B:=\sup_{t\in[0,T/2]} \vertiii{\Omega(t)}_{H^3}$ ($\ge 1$), and for each $s\in[0,\frac T2]$, let $\omega^s(\cdot,t) = \sum_{k=1}^N \theta_k \chi_{\Omega^s_k(t)}$ be the (unique) solution from Corollary~\ref{C.5.7} with initial condition $\omega^s(\cdot,s):=\omega(\cdot,s)$ at time $s$.  In particular, $\omega^0=\tilde\omega$.  Let $A:=2CB^8\ge B$ (where $C\ge 1$ is from from Corollary \ref{C.5.7}(b)), then let $T_0:=T(\alpha,N\sum_{k=1}^N|\theta_k|,A)>0$ (which is decreasing in the last variable) be from Corollary \ref{C.5.7}(b)  and consider any   $T_1\in(0, \min\{T_0,\tfrac T2,1\}]$.  Thus Corollary \ref{C.5.7}(b) shows that  $\omega^s$ exists on $[s,T_1]$ for each $s\in[0,T_1]$ and satisfies $\sup_{t\in[s,T_1]} \vertiii{\Omega^s(t)}_{H^3}\le \frac A2$.

For any $J\in\mathbb N$, let $s_j:=\frac jJ T_1$ for $j=0,\dots,J$.  Let $C=C(\alpha, \sum_{k=1}^N|\theta_k|, A)<\infty$ be from Lemma \ref{lem:onestep} and also larger than the universal $C$ in Lemmas \ref{lem:l2_symm} and \ref{lem:arcl_diff}, and consider any $J\ge C^{2\alpha}$.  Then Lemma \ref{lem:onestep} applied to $\omega$ and $\omega^{s_{j-1}}$ (with starting time $s_{j-1}$) and $s_j-s_{j-1}= \frac {T_1}J\le \frac 1J$ imply
\[
 d_H\left(\partial \Omega(s_{j}),  \partial \Omega^{s_{j-1}}(s_{j})\right) \leq C J^{-1/{2\alpha}}
\]
for $j=1,\dots,J$. Since $\Omega(s_{j})=\Omega^{s_j}(s_{j})$, it follows from Lemma \ref{lem:arcl_diff} that the families $\partial\Omega^{s_j}(s_{j})$ and $\partial \Omega^{s_{j-1}}(s_{j})$ have constant speed parametrizations $Z_j(s_j)$ and $\tilde Z_j(s_j)$ satisfying
\[
\|Z_j(s_j)-\tilde Z_j(s_j)\|_{L^2} \le \sqrt{2\pi N} \|Z_j(s_j)-\tilde Z_j(s_j)\|_{L^\infty} \le \sqrt{2\pi N} C^2A^7 J^{-1/{2\alpha}}.
\]
We have
\[
\vertiii{Z_j(s_j)}+\vertiii{\tilde Z_j(s_j)}=\vertiii{\Omega^{s_j}(s_j)}_{H^3}+\vertiii{\Omega^{s_{j-1}}(s_j)}_{H^3}\le A,
\]
so that Theorem \ref{thm:local} yields solutions $Z_j$ and $\tilde Z_j$ to \eqref{eq:contour2} on the time interval $[s_j,s_j+T_0]\supseteq[s_j,T_1]$ and with initial data $Z_j(s_j)$ and $\tilde Z_j(s_j)$, respectively.  Theorem \ref{thm:local}  also shows that $\sup_{t\in[s_j,T_1]} (\vertiii{Z_j(t)}+\vertiii{\tilde Z_j(t)})\le 4A$, and then \eqref{eq:l2diff} with $W:=Z_j-\tilde Z_j$ yields
\[
\|Z_j(T_1)-\tilde Z_j(T_1)\|_{L^2} \le e^{C(\alpha) N\sum_{k=1}^N|\theta_k| (4A)^3T_1}\|Z_j(s_j)-\tilde Z_j(s_j)\|_{L^2}\le \tilde C J^{-1/{2\alpha}},
\]
where $\tilde C:= \sqrt{2\pi N} C^2A^7e^{C(\alpha) N\sum_{k=1}^N|\theta_k| (4A)^3T_1}$.  This and Lemma \ref{lem:l2_symm} show that with $\bar C:=C4A\sqrt N \tilde C$  we have
\[
|\Omega^{s_j}(T_1) \triangle \Omega^{s_{j-1}}(T_1)| \le \bar C J^{-1/{2\alpha}}.
\]
This holds for $j=1,\dots,J$, hence we obtain by telescoping,
\[
|\Omega(T_1) \triangle \tilde \Omega(T_1)|=|\Omega^{T_1}(T_1) \triangle \Omega^{0}(T_1)| \le \bar C J^{1-1/{2\alpha}}.
\]
Since $\bar C$ is independent of $J$ and $2\alpha<1$, we take $J\to\infty$ to get $|\Omega(T_1) \triangle \tilde \Omega(T_1)|=0$.  Hence $\omega(\cdot,T_1)=\tilde\omega(\cdot,T_1)$ for each $T_1\in(0, \min\{T_0,\tfrac T2,1\}]$, which is a contradiction with our hypothesis $\inf \{t\in(0,T) \,:\, \omega(\cdot, t) \neq \tilde \omega(\cdot, t)\}=0$.
\end{proof}

\begin{proof}[Proof of Lemma \ref{lem:arcl_diff}]
Let $C_0$ and $R := (4 C_0 A)^{-3}$ be from Lemma \ref{lem:func_diff}.
We can assume without loss that $h:=d_H(\partial\Omega, \partial\tilde\Omega) \le \frac{R^2}{4}$, because otherwise
the result holds with any $C\ge 8(4C_0)^6$ due to $\Omega,\tilde\Omega\subseteq B(0,A)$.
%

Since $ \frac{R^2}{4} < \frac{R}{20}$ due to $R\leq \frac 1{4^3}$, we can apply Lemma \ref{lem:func_diff} to $\Omega$ and $\tilde \Omega$. It shows that in the coordinate system $(w_1, w_2)$ centered at any given $P\in\partial\Omega$ and with axes $n_P^\perp$ and $n_P$, both $\partial \Omega \cap B(P,R)$ and $\partial \tilde \Omega \cap B(P,\frac{19}{20}R)$ are graphs $w_2 = f(w_1)$ and $w_2 = g(w_1)$, respectively, such that for any $|w_1| \leq \frac R2$ we have $|f'(w_1)|\leq 1$, $|g'(w_1)|\leq 1$, and
\begin{equation}\label{eq:3eps}
|f(w_1) - g(w_1)| \leq 2h.
\end{equation}

We also claim that $|f''(w_1)| \leq 2C_0 A^3$ and $|g''(w_1)| \leq 2C_0 A^3$ for all $|w_1| \leq \frac{R}{2}$.  Indeed, let $y(\xi)$ be $z(\xi)$ in the new coordinates $(w_1,w_2)$.   Then for any $\xi\in\mathbb T$ such that $y(\xi)\in B(0,R)$ (i.e., $z(\xi)\in B(P,R)$), we have $f(y_1(\xi)) = y_2(\xi)$.  Thus $f'(y_1(\xi))= \frac{y_2'(\xi)}{y_1'(\xi)}$ and
\[
\begin{split}
f''(y_1(\xi)) &= \frac{\left(y_2'(\xi)/y_1'(\xi)\right)'}{y_1'(\xi)} = \frac{y_2''(\xi) y_1'(\xi) - y_1''(\xi) y_2'(\xi)}{y_1'(\xi)^3} = \frac{y_2''(\xi) - y_1''(\xi) f'(y_1(\xi))}{y_1'(\xi)^{2}}.
\end{split}
\]
If, in addition, $|y_1(\xi)| \leq \frac{R}{2}$, then we have $y_1'(\xi)\geq \frac{1}{\sqrt{2}A}$ because $|y'(\xi)| \geq \frac{1}{A}$ (due to $F[y]=F[z]\le A$) and $|\frac{y_2'(\xi)}{y_1'(\xi)}| =|f'(y_1(\xi))|\leq 1$.  This, $|f'(y_1(\xi))|\leq 1$, and $\|y''\|_{L^\infty}=\|z''\|_{L^\infty}\leq C_0 A$ now yield $|f''(y_1(\xi))| \leq 2C_0 A^3$.  The bound for $g$ is obtained identically.

Next, we claim that for $|w_1| \leq \frac{R}{2}$ we have
\begin{equation}
\label{eq:dfdg}
|f'(w_1) - g'(w_1)| \leq 8 C_0 A^{3/2} \sqrt h.
\end{equation}
If this is violated for some $|w_1^0| \leq \frac{R}{2}$ (without loss we can assume $w_1^0\le 0$ as well as $f'(w_1^0) - g'(w_1^0) > 8 C_0 \sqrt{ A^3 h} $),  the estimate $|f''-g''| \leq 4C_0 A^3$ on $[-\frac{R}{2}, \frac{R}{2}]$ yields
\[
f'(w_1)-g'(w_1) > 4 C_0 A^{3/2} \sqrt h
\]
for  all $w_1 \in [w_1^0, w_1^1]$, where $w_1^1 := w_1^0 + A^{-3/2} \sqrt h$ ($\le \frac R2$ because $w_1^0\le 0$, $A\ge 1$,  and $h\le \frac{R^2}{4}$).  Then $C_0\ge 1$ shows
\[
f(w_1^1) - g(w_1^1) > f(w_1^0)-g(w_1^0) +4 C_0 A^{3/2} \sqrt h  A^{-3/2} \sqrt h \geq -2h + 4C_0h\ge 2h,
\]
contradicting \eqref{eq:3eps}.  Thus \eqref{eq:dfdg} holds.

For any $P \in \partial \Omega$, let $F(P)\in\partial \tilde \Omega\cap B(P,\frac{19}{20}R)$ be such that $(F(P)-P)\cdot n_P^\perp = 0$.  Lemma~\ref{lem:func_diff} shows that such $F(P)$ exists and is unique, $|F(P)-P| \leq 2h$, and $F(P)$ is continuous in $P$ (the latter because of continuity of $f'$ and the bound $|g'|\le 1$ on $[-\frac R2,\frac R2]$).

In addition,  $F$ is injective. Indeed, assume that $F(P) = F(Q)=:S$ for some distinct $P,Q\in \partial \Omega$, and also without loss that $|S-P|\ge |S-Q|$.  Then $|P-Q|\le 2|P-S|\le 4h<\tfrac R2$, so Lemma \ref{lem:regularity}(c) with $\gamma=1$ (together with $L\geq \frac{1}{A}$, as before) yields
\begin{equation}
\label{eq:nqnp}
\sin\angle PSQ = n_Q \cdot n_P^{\perp} = (n_Q-n_P)\cdot n_P^{\perp} \leq |n_Q-n_P| \leq 2C_0A^3|P-Q|.
\end{equation}
We also have $\angle PSQ\le\frac\pi 4$ (due to $|f'|\le 1$ on $[-\frac R2,\frac R2]$ in Lemma \ref{lem:func_diff}), so $|S-P|\ge |S-Q|$ and $|f'|\le 1$ imply $\angle PQS\in [\frac{3\pi}8,\frac{3\pi}4]$. The law of sines now yields
\[
|S-P| = \frac{|P-Q| \sin\angle PQS}{\sin\angle PSQ}  \geq \frac{\sin \angle PQS}{2C_0 A^3} \geq \frac{1}{4 C_0A^3} >R,
\]
a contradiction with $|S-P|=|F(P)-P| \leq 2h < R/4$.  Hence $F:\partial \Omega\to \partial \tilde \Omega$ is injective.  Since it is also continuous and $\partial\Omega,\partial\tilde\Omega$ are both simple closed curves, $F$ is a bijection.

 Next, we claim that  for any distinct $P,Q \in \partial \Omega$ with $|P-Q| \leq h$, we have with $C_1:=300C_0^2$,
 \begin{equation}
 \label{eq:gooool}
\left| \frac{|F(P)-F(Q)|}{|P-Q|} -1 \right| \leq C_1A^6 h.
 \end{equation}
 Without loss assume that $P$ is the origin and $n_P=(0,1)$ (so that $(F(P))_1=0$).  Let $f,g$ be from Lemma \ref{lem:func_diff} and recall that we proved above that  $|f''| \leq 2C_0 A^3$ on $[-\frac{R}{2},\frac{R}{2}]$.  This and $f'(0)=0$ yield
 $|f'|\le 2C_0 A^3|Q|$ on $[-|Q|,|Q|]$, hence
 $ \frac{|Q_2|}{|Q_1|}\leq 2C_0 A^3|Q|$
 and
 $\frac {|(n_Q)_1|}{|(n_Q)_2|}\leq 2C_0 A^3|Q|$.  Since $|Q_1|\le h$ and $|(n_Q)_2|\le 1$, it follows that
 \begin{equation}\label{qaux827}
(1-2C_0A^3h)|Q|\le  |Q_1|\le |Q| \qquad\text{and}\qquad |(n_Q)_1|\le 2C_0A^3|Q|.
 \end{equation}
 This and $|F(Q)-Q|\le 2h$ yield
 \[
 |(F(Q)-Q)_1|=|F(Q)-Q||(n_Q)_1| \le 4C_0A^3h|Q|.
 \]
By using $(F(P))_1=0$, an elementary inequality $||a|-|b||\leq |a-c|+||c|-|b||$, and the first bound in \eqref{qaux827}, we obtain 
\begin{equation} \label {5.99}
\big| |(F(P)-F(Q))_1| -|Q| \big| \le  |(F(Q)-Q)_1|+\big| |Q_1|-|Q| \big| \le 6C_0A^3h|Q|.
\end{equation}
 From $f'(0)=0$ and \eqref{eq:dfdg} we also have $|g'(0)|\le 8C_0A^{3/2}\sqrt h$, which together with $|g''| \leq 2C_0 A^3$ on $[-\frac{R}{2},\frac{R}{2}]$ (proved above) yields $|g'| \leq 18C_0 A^3\sqrt h$ on $[-5h,5h]$.  Since
\[
|F(P)-F(Q)|\le |F(P)-P|+|P-Q|+|Q-F(Q)|\le 2h+h+2h=5h,
\]
it follows that $\frac{|(F(P)-F(Q))_2|}{|(F(P)-F(Q))_1|}\le 18C_0A^3\sqrt h$.
Since $6C_0A^3h\le \frac 1{10}$ (due to $h\le \frac{R^2}4$, the definition of $R$, and $C_0,A\ge 1$), it follows from this and \eqref{5.99} that
\[
|(F(P)-F(Q))_2| \le 20C_0A^3\sqrt h|Q|.
\]
But this and \eqref{5.99} now yield (also using $6C_0A^3h\le \frac 1{10}$ and $\sqrt{1+b}\le 1+ \frac b2$ for $b\ge 0$)
\[
\big| |F(P)-F(Q)|-|Q| \big| \le \left|\left(1+13C_0A^3h+400 C_0^2A^6h\right)^{1/2}-1 \right||Q| \le 207C_0^2A^6h |Q|,
\]
so \eqref{eq:gooool} follows because $P$ is the origin.

%

For $P,Q \in \partial \Omega$ (or $\partial \tilde \Omega$), we now define $L(P,Q)$ (or $\tilde L(P,Q)$) to be the arc-length along $\partial\Omega$ (or $\partial\tilde \Omega$) from $P$ to $Q$, in the counter-clockwise direction. For any $P,Q \in \partial \Omega$, one can obtain $L(P,Q)$ as the  limit as $J\to\infty$  of lengths of polygonal paths $P=P_0\to P_1\to\cdots\to P_J$, with  each $P_{j+1}$ lying on the arc $P_jQ$ of  $\partial\Omega$ and all the segment lengths $|P_{j+1}-P_j|$ less than some $l_J$ which satisfies $\lim_{J\to\infty} l_J=0$.
   Then $\tilde L(F(P),F(Q))$ is the limit of the lengths of the paths $F(P)=F(P_0)\to F(P_1)\to\cdots\to F(P_J)=F(Q)$ because $|P_{j+1}-P_j|\le 2l_J$ for all large $J$, due to \eqref{eq:gooool} and $h\le\frac{R^2}4=\frac 1{4^7C_0^6A^6}<\frac 1{300C_0^2A^6}$.

It follows then from \eqref{eq:gooool} that for any $P,Q\in\partial\Omega$ we have
\[
\frac{\tilde L(F(P),F(Q))}{L(P,Q)} \in [1-C_1A^6 h, 1+C_1A^6 h].
\]
If $z$ is a constant speed parametrization of $\partial\Omega$, then
\[
L(P,Q)\le |\partial\Omega|=\|z'\|_{L^1}\le 2\pi \|z'\|_{L^\infty}\le 2\pi \|z\|_{C^2}\le 2\pi C_0 \|z\|_{H^3},
\]
which yields (with $C_2:=2\pi C_0C_1$)
\[
|\tilde L(F(P),F(Q)) - L(P,Q)| \leq C_1 A^6  h |L(P,Q)| \leq C_2 A^7 h.
\]
In particular, we have $\left||\partial \tilde \Omega| - |\partial \Omega| \right|\leq C_2 A^7h$.

 Finally, fix $z$ above and let $\tilde z$ be the (unique) constant speed parametrization of $\partial \tilde \Omega$ satisfying $\tilde z(0) = F(z(0))$. Then for any $\xi\in [0,2\pi)$, we have
\begin{equation*}
\begin{split}
|z(\xi) - \tilde z(\xi)| &\leq |z(\xi) - F(z(\xi))| + |F(z(\xi)) - \tilde z(\xi)|\\
&\leq 2h + |\tilde L(F(z(\xi)),\tilde z(0)) - \tilde L(\tilde z(0), \tilde z(\xi))|\\
&\leq 2h + |\tilde L(F(z(0)), F(z(\xi))) - L(z(0), z(\xi))| + | L(z(0), z(\xi))- \tilde L(\tilde z(0), \tilde z(\xi))|\\
&\leq 2h + C_2A^7 h + \frac{\xi}{2\pi} \left||\partial \tilde \Omega| - |\partial \Omega| \right|\\
&\leq (2+2C_2) A^7 h,
\end{split}
\end{equation*}
which yields \eqref{eq:goal_arcl}.
\end{proof}

\section{Proofs of Proposition \ref{P.1.3} and Theorem \ref{T.1.7}}\label{sec:prop13}


Let us start with some estimates on fluid velocities generated by $C^{1,\gamma}$ patches.  These results apply at a fixed time, hence we drop the argument $t$ in them.  And again, while we consider here the half-plane case $D=\Rm\times\Rm^+$, the arguments are identical for the whole plane $D=\Rm^2$.

We first consider the setting from Definition \ref{D.4.1a}, and will assume that $L \geq 1$ (note that since the area of each evolving patch stays constant, we only need to choose it to be at least $\pi$ initially so that the arc-length of the patch boundary will always be at least $2\pi$).
This is done to simplify our estimates but can be replaced by $L\ge \frac 1A$ for some $A<\infty$.  Since  $L\ge \tfrac 1\pi |z(\pi)-z(0)|\ge F[z]^{-1}\ge \vertiii{\Omega}_{1,\gamma}^{-1}$, this assumption can even be omitted (at the expense of changing the constants) because the results below assume $\vertiii{\Omega}_{1,\gamma}\le A$.

The following is a crucial bound on the gradient of the component of $v$ from \eqref{4.1} normal to $\partial\Omega$, that is, on $\nabla(v(x)\cdot n_P)=\nabla v(x)n_P$, with $x=P+rn_P$ for some small $r$.

\begin{lemma}
\label{lem:onepatch}
For $\gamma>\frac{2\alpha}{1-2\alpha}$, let $\Omega\subseteq \mathbb{R}^2$ be as in Definition \ref{D.4.1a}, with $L \geq 1$ and $\vertiii{\Omega}_{1,\gamma}\leq A$ for some $A\geq 1$.   Let also $R := (4A)^{-\frac{1}{\gamma}-1}$ and $v$ be given by \eqref{4.1}
with $\omega(x) = \chi_{\Omega}(x).$  Then for any $P \in \partial \Omega$ and any $x=P+rn_P$ with $|r|\in(0,\tfrac R2)$, we have $|\nabla v(x) n_P| \leq C(\alpha,\gamma)A$.
\end{lemma}


\begin{proof}
Let $S_P := \{y\in \mathbb{R}^2: (y-P)\cdot n_P\in(-R,0)\}$, and let $v_{S_P}$ be given by \eqref{4.1} with $\omega=\chi_{S_P}$, evaluated as principal value.
By symmetry we have $v_{S_P} \cdot n_P \equiv 0$.  Thus
\[
\begin{split}
|\nabla v(x) n_P| &= |\nabla (v(x)-v_{S_P}(x)) n_P| \leq  |\nabla (v(x)-v_{S_P}(x)) |\leq C(\alpha) \int_{\Omega\triangle S_P} \frac{1}{|x-y|^{2+2\alpha}} dy\\
& \leq C(\alpha) \left(\underbrace{\int_{(\Omega\triangle S_P)\cap B(P,R)} \frac{1}{|x-y|^{2+2\alpha}} dy}_{=:I_1} + \underbrace{\int_{\mathbb{R}^2 \setminus B(P,R)} \frac{1}{|x-y|^{2+2\alpha}} dy}_{=:I_2}\right),
\end{split}
\]
where as before $A\triangle B:=(A\setminus B)\cup (B\setminus A)$.
Using $|x-P|<\tfrac R2$, we obtain
\begin{equation}
\label{outside}
I_2 \leq 2\pi \int_{R/2}^\infty r^{-(1+2\alpha)}  dr \leq C(\alpha) R^{-2\alpha} \leq C(\alpha) A,
\end{equation}
where in the last step we used the definition of $R$, $A\ge 1$, and $2\alpha\frac{1+\gamma}{\gamma} < 1$.

To control $I_1$, we change coordinates to $(w_1, w_2)$ from Lemma \ref{lem:regularity}(b), which then implies that  $(\Omega \triangle S_P)\cap B(P,R)$ lies between the curves $w_2 = \pm 4Aw_1^{1+\gamma}$. Hence
\[
I_1 \leq 2\int_0^R w_1^{-(2+2\alpha)} 8Aw_1^{1+\gamma} dw_1 \leq C(\alpha, \gamma) A R^{\gamma-2\alpha} \leq C(\alpha, \gamma) A,
\]
where we first used that the $w_1$ coordinate of $x$ is 0, and then that $\gamma>2\alpha$ and $R<1$.
\end{proof}

{\it Remarks.}
 1. The estimate on $\nabla v(x) n_P$ holds not only on the line normal to $\partial \Omega$ at $P,$ but also non-tangentially.
Given any $\sigma >0,$ it is easy to see that we can replace the condition on $x$ in the statement of Lemma \ref{lem:onepatch} with $|x-P| < c(R,\sigma)$ (for some $c(R,\sigma)>0$) and $(x-P) \cdot n_P \geq \sigma |x-P|$, with the conclusion being
$|\nabla v(x)n_P| \leq C(\alpha,\gamma, \sigma)A$.

2. Note that $\nabla v$ is in general not defined at $P\in \partial \Omega$ due to a lack of regularity in the tangential component $v(x) \cdot n^\perp_P$ of $v$ at $P$.
However, the argument in the proof of Lemma~\ref{lem:onepatch} can be used to show that the normal component $v(x) \cdot n_P$ is sufficiently regular at $P$, and $\nabla (v(P) \cdot n_P)$ can in fact be defined.  We will make this more precise later.


Lemma~\ref{lem:onepatch} and Lemma \ref{lem:du_crude} now yield the following.

\begin{corollary}
\label{cor:du}
Let $\Omega$ and $v$ satisfy the hypotheses of Lemma \ref{lem:onepatch}.  Then for any $x\not \in \partial \Omega$ and any $P\in\partial\Omega$ such that  $|x-P|={\rm dist}(x,\partial\Omega)=:d(x)$ we have $|\nabla v(x) \frac{x-P}{|x-P|}| \leq C(\alpha, \gamma) A$ and $| (v(x) - v(P))\cdot \frac{x-P}{|x-P|} | \leq C(\alpha, \gamma) A |x-P|$.
\end{corollary}

\begin{proof}
Notice that $\frac{x-P}{|x-P|}\in\{n_P,-n_P\}$. If $d(x) < \tfrac R2$, then the first claim follows from Lemma \ref{lem:onepatch}. Otherwise, Lemma \ref{lem:du_crude} yields
$
|\nabla v(x) n_{P}| \leq |\nabla v(x)| \leq C(\alpha) R^{-2\alpha} \leq C(\alpha) A
$
because $A\ge 1$ and $2\alpha\frac{1+\gamma}{\gamma} < 1$.

To prove the second claim, note that if $y_s := x + s (P-x)$ for $s\in[0,1]$, then ${\rm dist}(y_s,P)={\rm dist}(y_s,\partial\Omega)$. Hence the first claim yields $|\nabla (v(y_s) \cdot n_{P})| \leq C(\alpha, \gamma) A$ for $s\in[0,1)$.  Integrating this in $s\in[0,1)$ and using continuity of $u$ yields the second claim.
\end{proof}

We now extend Lemma~\ref{lem:onepatch} and Corollary \ref{cor:du} to the case of $N$ patches $\Omega_k\subseteq D$ with disjoint boundaries.
We let $2\pi L_k:=|\partial\Omega_k|$, and if $P\in\partial\Omega_k$ (such $k$ is then unique), we denote by $n_P$ the outer unit normal vector for $\Omega_k$ at $P$.  We will again assume that $|L_k|\ge 1$, and also that $|\theta_k|\le 1$, both of which are not essential but simplify our formulas.
 Finally, recall that $\partial\Omega:=\bigcup_{k=1}^N \partial\Omega_k$.

\begin{proposition}
\label{lem:multiple_patch}
For $\gamma>\frac{2\alpha}{1-2\alpha}$, some $A\geq 1$,  and $k=1,\dots,N$, let $\Omega_k\subseteq D$ be as in Definition~\ref{D.4.1a}, with $L_k \geq 1$ and $\vertiii{\Omega_k}_{1,\gamma}\leq A$.   Assume  also $\textup{dist}(\partial\Omega_i,\partial\Omega_k) \geq \frac{1}{A}$ for all $i\neq k$ and let $R := (4A)^{-\frac{1}{\gamma}-1}$. Finally,  let $u$ be given by \eqref{eq:velocity_law} where $\omega = \sum_{k=1}^N \theta_k \chi_{\Omega_k}$ and $|\theta_k|\le 1$.
Then for any $P \in  \partial\Omega$ and any $x=P+rn_P\in\bar D$ with $|r|\in(0,\tfrac R2)$, we have $|\nabla u(x) n_P| \leq C(\alpha,\gamma)A$.
\end{proposition}

\begin{proof}
Denote by $\tilde \Omega_k$ the reflection of $\Omega_k$ with respect to the $x_1$-axis.
Since $P\in \partial \Omega_k$ for some $k$, $\min_{i\neq k}\textup{dist}(\partial\Omega_i,\partial\Omega_k) \geq \frac{1}{A} > R$, and $x\in B(P,\frac R2)\cap\bar D$,  we have ${\rm dist}(x,\Omega_i)>\frac R2$ and ${\rm dist}(x,\tilde \Omega_i)>\frac R2$ for all $i\neq k$. Due to Lemma \ref{lem:du_crude}, the total contribution to $|\nabla u(x)|$ from all the $\Omega_i$ and $\tilde\Omega_i$ with $i\neq k$ is bounded by $C(\alpha)R^{-2\alpha}$, and hence also by $C(\alpha)A$, by the definition of $R$, $A\ge 1$, and  $\frac{1+\gamma}{\gamma}2\alpha<1$.

Moreover, the contribution to $|\nabla u(x) n_{P}|$ from $\Omega_k$ is bounded by $C(\alpha,\gamma)A$ due to  Lemma \ref{lem:onepatch}. Thus it suffices to bound the contribution from $\tilde\Omega_k$.  Let
\[
\tilde v(x) := \int_{\mathbb{R}^2} \frac{(x-y)^\perp} {|x-y|^{2+2\alpha}} \chi_{\tilde\Omega_k}(y) dy,
\]
so that it suffices to show that $|\nabla \tilde v(x) n_{P}|\leq C(\alpha,\gamma)A$. Let $\tilde d_k(x) := \textup{dist}(x, \tilde{\Omega}_k)$, where the minimum is achieved at some $Q_x\in \partial \tilde \Omega_k$, and let $n_{Q_x}$ be the outer unit normal for $\tilde \Omega_k$ at $Q_x$. We then have
\[
\begin{split}
|\nabla \tilde v(x)  n_{P}| &\leq  |\nabla \tilde v(x)   n_{Q_x}|  + |\nabla \tilde v(x)  (n_{P} + n_{Q_x})| \\
&\leq C(\alpha, \gamma) A + C(\alpha) \tilde d_k(x)^{-2\alpha} |n_P + n_{Q_x}|,
\end{split}
\]
where we bounded the first term by Corollary \ref{cor:du} for $x$ and $\tilde \Omega_k$, and the second term by Lemma \ref{lem:du_crude}.

Note that if $\tilde d_k(x) \geq \frac R4$, then then the needed inequality holds because $R^{2\alpha} \geq \frac 1A$.
Hence it  suffices to show that $ |n_P + n_{Q_x}|\leq C(\alpha, \gamma) A \tilde d_k(x)^{2\alpha} $ if $\tilde d_k(x) \leq \frac R4$.   Let $\bar Q_x \in \partial\Omega_k$ be the reflection of $Q_x$ across the $x_1$-axis. Then we have
\[
|P-\bar Q_x| \leq |x-P| + |x-\bar Q_x| \leq 2 {\rm dist}(x,\partial\Omega_k)+\tilde d_k(x) \leq 3\tilde d_k(x),
\]
where in the second inequality we used that $|x-P|<\frac R2$ and Lemma \ref{lem:regularity}(b) imply $ {\rm dist}(x,\partial\Omega_k) > \frac{|x-P|}2$.
Now Lemma \ref{lem:regularity}(c), $R<1$, $L_k \geq 1$ and $\gamma > 2\alpha$ yield
\[
|n_P - n_{\bar Q_x}| \leq 2A|P-\bar Q_x|^{\gamma} \leq   6A \tilde d_k(x)^{2\alpha} .
\]
Symmetry, Lemma \ref{lem:regularity}(d), $\tilde d_k(x)=|x- Q_x|$, $x\in\bar D$,  and $\frac{\gamma}{1+\gamma}>2\alpha$ also give (with $\bar Q_x =: (q_1, q_2)$ and $Q_x = (q_1, -q_2)$)
\[
|n_{Q_x} +  n_{\bar Q_x}| = 2|n_{\bar Q_x}\cdot (1,0)| \leq 4A^{\frac{1}{1+\gamma}} q_2^{\frac{\gamma}{1+\gamma}} \leq 4A^{\frac{1}{1+\gamma}} \tilde d_k(x)^{\frac{\gamma}{1+\gamma}} \leq 4A\tilde d_k(x)^{2\alpha}.
\]
Thus $ |n_P + n_{Q_x}|\leq 10A \tilde d_k(x)^{2\alpha} $ and the proof is finished.
\end{proof}

We now obtain the following analog of Corollary \ref{cor:du} (with an identical proof).

\begin{corollary}
\label{cor:du_multiple}
Let $\omega,u$ satisfy the hypotheses of Proposition \ref{lem:multiple_patch}.  Then for any $x \in \bar D \setminus \partial\Omega$  and any $P\in\partial\Omega$ such that  $|x-P|={\rm dist}(x,\partial\Omega)=:d(x)$ we have
\begin{equation}
\label{eq:normal_v_bound}
\left| \nabla u(x) \frac{x-P}{|x-P|} \right| \leq C(\alpha, \gamma) A
\end{equation}
and
\begin{equation}
\left| (u(x) - u(P))\cdot \frac{x-P}{|x-P|} \right| \leq C(\alpha, \gamma) A |x-P| .
\label{eq:diff_v}
\end{equation}
\end{corollary}

%

Therefore,  the normal component of the velocity $u$ generated by such $\omega$ is Lipschitz in the normal direction relative to $\partial\Omega$.
We will also need this result for $\partial D$.

\begin{proposition}
\label{prop:vertical_v}
Let $\omega,u$ satisfy the hypotheses of Proposition \ref{lem:multiple_patch}.
Then for any $x = (x_1, x_2) \in \bar D$ we have
\begin{equation}
\label{eq:u2_bound}
|u_2(x)| \leq C(\alpha, \gamma) N A^{\frac{2+\gamma}{1+\gamma}} x_2.
\end{equation}
\end{proposition}

\begin{proof}
We have
\begin{equation}
\begin{split}
u_2(x)&= \sum_{k=1}^N \theta_k \int_{\Omega_k} \left(
\frac{y_1-x_1}{|x-y|^{2+2\alpha}} - \frac{y_1-x_1}{|x-\bar{y}|^{2+2\alpha}}\right)dy\\
&= -\sum_{k=1}^N \frac{\theta_k }{2\alpha} \int_{\Omega_k} \partial_{y_1} \left(\frac{1}{|x-y|^{2\alpha}} - \frac{1}{|x-\bar{y}|^{2\alpha}}\right)dy\\
&= -\sum_{k=1}^N \frac{\theta_k }{2\alpha} \int_{\partial \Omega_k} (n_y)_1 \left(\frac{1}{|x-y|^{2\alpha}} - \frac{1}{|x-\bar{y}|^{2\alpha}}\right)d\sigma(y)\\
&=  \sum_{k=1}^N \frac{\theta_k}{2\alpha L_k} \int_{\mathbb{T}}  (\z_k')_2(\xi) \left( \frac{1}{|x-\z_k(\xi)|^{2\alpha}} - \frac{1}{|x-\bar \z_k(\xi)|^{2\alpha}}\right) d\xi,
\end{split}
\end{equation}
where $(n_y)_1=n_y\cdot(1,0)$  for $y\in \partial \Omega_k$ and  $(z_k)_2=z_k\cdot(0,1)$, with $z_k$ a constant speed parametrization of $\partial\Omega_k$.
Hence for any $x\in \bar D$ we have
\begin{equation}
\label{eq:u2bound}
\begin{split}
|u_2(x)| &\leq   \sum_{k=1}^N \frac{\theta_k}{2\alpha L_k} \int_{\mathbb{T}}  \frac{|(\z_k')_2(\xi)|}{|x-\z_k(\xi)|^{2\alpha}} \left| 1-\frac{ |x-\z_k(\xi)|^{2\alpha}}{|x-\bar \z_k(\xi)|^{2\alpha}}\right| d\xi\\
&\leq \sum_{k=1}^N  \frac{\theta_k }{2\alpha L_k} \int_{\mathbb{T}}  \frac{|(\z_k')_2(\xi)| 2x_2}{|x-\z_k(\xi)|^{2\alpha}|x-\bar \z_k(\xi)|} d\xi\\
&\leq   \sum_{k=1}^N \frac{2\theta_k A^{\frac{1}{1+\gamma}} x_2 }{\alpha L_k} \underbrace{ \int_{\mathbb{T}}  \frac{1}{|x-\z_k(\xi)|^{2\alpha+\frac{1}{1+\gamma}}} d\xi}_{=:T_k},
\end{split}
\end{equation}
where in the second inequality we used that $|1-b^{2\alpha}| \leq |1-b|$ for $b\ge 0$ and $\alpha\in (0,\frac{1}{2})$, as well as that $0\leq |x-\bar \z_k(\xi)| - |x-\z_k(\xi)| \leq 2x_2$; and in the last inequality we used
Lemma \ref{lemma:f'bis} for $(\z_k)_2$
and also that $|x-\bar \z_k(\xi)| \geq \max\{ (z_k)_2(\xi),  |x-\z_k(\xi)|\}$.

It now remains to show  that $T_k \leq C(\alpha,\gamma) A$. Let $d_k(x) := d(x,\partial \Omega_k)$, and  consider only the case $d_k(x)\leq 1$ because otherwise clearly $T_k \leq 2\pi$.  Let $\xi_0\in\mathbb T$ be such that $|x-z_k(\xi_0)|=d_k(x)$.
Using $d_k(x)\le 1$ and $F[z_k]\leq A$ 
yields
\[
\begin{split}
T_k &= \int_{|\xi - \xi_0| \leq 2Ad_k(x)} {|x-z_k(\xi)|^{-2\alpha-\frac{1}{1+\gamma}}} d\xi +  \int_{|\xi - \xi_0| > 2Ad_k(x)} {|x-z_k(\xi)|^{-2\alpha-\frac{1}{1+\gamma}}} d\xi \\
&\leq 4Ad_k(x)^{\frac{\gamma}{1+\gamma} - 2\alpha} + \int_{|\xi-\xi_0| > 2Ad_k(x)} {\left(|z_k(\xi)-z_k(\xi_0)|- d_k(x)\right)^{-2\alpha-\frac{1}{1+\gamma}}} d\xi\\
&\leq 4A+ \int_{|\xi-\xi_0| > 2Ad_k(x)} {\left(\frac{1}{2}|z_k(\xi)-z_k(\xi_0)|\right)^{-2\alpha-\frac{1}{1+\gamma}}} d\xi\\
&\leq 4A + (2A)^{2\alpha+\frac{1}{1+\gamma}} \int_{\mathbb{T}} |\xi - \xi_0|^{-2\alpha-\frac{1}{1+\gamma}} d\xi,
\end{split}
\]
which is bounded by $C(\alpha,\gamma)A$ due to $2\alpha < \frac{\gamma}{1+\gamma} $ and $A\geq 1$.
\end{proof}

The above results lead to the following lemma, from which Proposition \ref{P.1.3} will follow.

\begin{lemma} \label{L.5.2}
Consider the setting of Proposition \ref{P.1.3} and assume that $\omega$ is a $C^{1,\gamma}$ patch solution to \eqref{sqg}-\eqref{eq:velocity_law} on  $[0,T)$.  Then $\Phi_t(x)$ is unique for each $(x,t)\in(\bar D\setminus\partial\Omega(0))\times[0,T)$, and for each $T'\in(0, T)$, there is $B<\infty$ such that $d_t(x):={\rm dist} (x,\partial\Omega(t))$ satisfies
\begin{equation} \label{5.1}
d_t(\Phi_t(x)) \ge e^{-Bt}d_0(x) \qquad \text{and} \qquad (\Phi_t(x))_2 \ge e^{-Bt}x_2
\end{equation}
for each $(x,t)\in(\bar D\setminus\partial\Omega(0))\times[0,T']$.
\end{lemma}

\begin{proof}
Let $A\ge 1$ be such that
\[
A\ge \sup_{t\in[0,T']} \left[ \max_{k} \vertiii{\Omega_k(t)}_{1,\gamma}+\max_{k\neq i} {\rm dist} \left( \partial\Omega_k(t),\partial\Omega_i(t) \right)^{-1} \right],
\]
and let $B:=C(\alpha,\gamma)NA^{\frac{2+\gamma}{1+\gamma}}$, with $C(\alpha,\gamma)$ from Corollary \ref{cor:du_multiple} and Proposition \ref{prop:vertical_v}.
To satisfy the hypotheses of Proposition \ref{lem:multiple_patch} on $[0,T']$, we will assume that $|\theta_k| \leq 1$ and that the  lengths $2\pi L_k(t)$ of $\partial\Omega_k(t)$ satisfy $L_k(t)\ge 1$ for all $k$ and $t\in[0,T']$.  As we remarked before Lemma \ref{lem:onepatch} and before Proposition \ref{lem:multiple_patch},
 these two assumptions are not essential and can be removed by adjusting the constants involved in the bounds with extra factors of $A$ and $\Theta:=\sum_{k=1}^N |\theta_k|$. One can also see this by a scaling argument.
 Specifically,  the scaling $\tilde\theta_k:=\theta_k\Theta^{-1}$ and $\tilde \Omega_k(t):= \lambda \Omega_k(\lambda^{-2\alpha}\Theta^{-1} t)$  yields a patch solution $\tilde\omega$ on $[0,A^{2\alpha}\Theta T)$.  Choosing  $\lambda:=A$ makes any constant speed parametrization $Z(t)$ of $\partial \tilde \Omega(t)$ satisfy $\inf_{t\in[0,A^{2\alpha}\Theta T']}|\tilde z_k(\pi,t)-\tilde z_k(0,t)|\ge \pi$ for each $k$ because $F[Z(t)]\le A$.  Hence $|\partial\tilde\Omega_k(t)|\ge 2\pi$
 for each $k$ and $t\in[0,T']$, and of course $|\tilde \theta_k|\le 1$. If the result holds for $\tilde \theta_k$ and $\tilde \Omega_k$ on $[0,A^{2\alpha}\Theta T']$, then it also holds for $\theta_k$ and $\Omega_k$ on $[0,T']$, but with $B$ replaced by $A^{2\alpha}\Theta B$.

Fix now any $x\in \bar D\setminus\partial\Omega(0)$ and let $T'<T$ be any time such that $\Phi_t(x)\notin\partial\Omega(t)$ for all $t\le T'$. To prove the lemma, it suffices to show \eqref{5.1} with this $T'$ and the $B$ from above.

  The second claim in \eqref{5.1} now follows directly from Proposition \ref{prop:vertical_v}, so let us consider the first.
 Notice that the function $f(t):=d_t(\Phi_t(x))$ is Lipschitz on $[0,T']$ (this is proved via the argument from Lemma \ref{lem:onestep}). Note that while $f$ depends on $x$, we will suppress this in the notation. 
 Hence $f'$ exists almost everywhere and $f(t)-f(0)=\int_0^t f'(s)ds$ for $t\in[0,T']$. Gronwall's inequality now shows that if the first claim in \eqref{5.1} does not hold for some $t\in[0,T']$, then there must be $s\in[0,t]$ such that $f(s)>0$ and $f'(s)<-Bf(s)$.  Let $a:=- \frac 13(Bf(s)+f'(s))>0$ and let $\delta>0$ be such that
\begin{equation} \label{5.2}
\left| \big[u(x'',s'') - u(x',s)\big]\cdot \frac{\Phi_s(x)-x'}{|\Phi_s(x)-x'|} \right| \leq B f(s)+a
\end{equation}
whenever $|x''-\Phi_s(x)|\le \delta$, $|s''-s|\le \delta$, and $|x'-P|\le \delta$ for some $P\in\partial\Omega(s)$ such that $|\Phi_s(x)-P|=f(s)$.  Existence of such $\delta$ follows from continuity of $u$ (which holds by the last claim in Lemma \ref{lemma:uniform_u_bound}) and Corollary~\ref{cor:du_multiple}, because \eqref{eq:diff_v} shows \eqref{5.2} with $(x'',s'',x'):=(\Phi_s(x),s,P)$ and $Bf(s)$ instead of $Bf(s)+a$.

Let now $\delta':=\inf_{x'\in S}  | \Phi_s(x) - x'|  - f(s)$ (with $\delta':=\infty$ if $S=\emptyset$), where
\[
S:= \{ x'\in\partial\Omega(s) \,:\, B(x',\delta)\cap \partial\Omega(s) \cap \overline{B(\Phi_s(x),f(s))}=\emptyset\},
\]
and notice that $\delta'>0$ since $\partial\Omega(s)$ is compact and so is $S$.  Because of this, the distance of the points in $S$ ($\subseteq \partial\Omega(s)$) from $\Phi_s(x)$ exceeds $f(s)$ by more than a positive constant, and thus their dynamics will not affect $f'(s).$   
Also let $h:=C^{-1}\min\{\delta',\delta\}$ $(\le\delta)$, where $C:=2\|u\|_{L^\infty}  +a+1$.
Since $f'(s)=-(Bf(s)+3a)$, there are $s'\in(s,s+h)$ with arbitrarily small $s'-s$ such that
\begin{equation} \label{7.1}
f(s')< f(s)-(Bf(s)+2a)(s'-s).
\end{equation}
  Pick such $s'$ so that we also have $ d_H \left(\partial\Omega(s'),X_{u(\cdot,s)}^{s'-s}[\partial\Omega(s)] \right) \le a(s'-s)$ (which is possible by \eqref{1.3}),
 and let $Q\in\partial\Omega(s')$ be such that $|\Phi_{s'}(x)-Q|=f(s')$. There exists $\tilde Q\in\partial\Omega(s)$ such that 
\begin{equation}\label{qtildeq}
|\tilde Q+(s'-s)u(\tilde Q,s)-Q|\le a(s'-s).
\end{equation} 
 Therefore 
\[
|\Phi_{s}(x)-\tilde Q| \le |\Phi_{s'}(x)-Q|+|\Phi_{s'}(x)-\Phi_{s}(x)|+|\tilde Q-Q| \le |\Phi_{s'}(x)-Q|+(2\|u\|_{L^\infty}+a)h < f(s)+\delta',
\]
which implies that $\tilde Q\in \partial\Omega(s)\setminus S$.  Hence $|\tilde Q-P|\le \delta$ for some
$P\in\partial\Omega(s)$ with $|\Phi_s(x)-P|=f(s)$.
Let us now write 
\begin{align*}
\Phi_{s'}(x)-Q &= \Phi_s(x)-\tilde Q  +\Phi_{s'}(x)-\Phi_s(x) + \tilde Q -Q
\\ &= \Phi_s(x)-\tilde Q +
\int_s^{s'} [u(\Phi_{s''}(x),s'')-u(\tilde Q,s)]\,ds''  +(s'-s)u(\tilde Q,s) +\tilde Q - Q.
\end{align*}
Multiplying this equality by $\frac{\Phi_s(x)-\tilde Q}{|\Phi_s(x) - \tilde Q|},$ and using \eqref{qtildeq}, we obtain
\begin{align*}
|\Phi_{s'}(x)-Q|
 & \geq |\Phi_s(x) - \tilde Q | -\int_s^{s'} \left| \big[ u(\Phi_{s''}(x),s'') - u(\tilde Q,s)\big] \cdot \frac{\Phi_s(x)-\tilde Q}{|\Phi_s(x) - \tilde Q|} \right|\,ds'' - a(s'-s)
\\ & \geq f(s) - (Bf(s)+2a)(s'-s),
\end{align*}
where in the last step we used that $|\tilde Q-P|\le \delta$, and that from $\|u\|_{L^\infty}h \le \delta$ we have $|\Phi_{s''}(x)-\Phi_{s}(x)|\le \delta$ for any $s''\in[s,s']$, so  \eqref{5.2} applies. 
The obtained inequality contradicts \eqref{7.1}, and the proof is finished.
\end{proof}

{\it Remark.} We note that the argument above can be extended in a straightforward manner to show that for $(x,t)\in(\bar D\setminus \partial\Omega(0))\times[0,T]$ we in fact have
\begin{equation} \label{5.3}
\frac{d^+}{dt} d_t(\Phi_t(x)) =  \inf_{P\in\partial\Omega(t)\cap \partial B(\Phi_t(x),d_t(\Phi_t(x)))} \left\{ \big[u(\Phi_t(x),t) - u(P,t)\big]\cdot \frac{\Phi_t(x)-P}{|\Phi_t(x)-P|} \right\},
\end{equation}
where $\frac{d^+}{dt}$ is the right derivative.  The left derivative has $\sup$ in place of $\inf$.

\begin{proof}[Proof of Proposition \ref{P.1.3}]
(a)
This follows from Lemma \ref{L.5.2} and smoothness of $u$ away from $\partial\Omega$ (see Lemma \ref{lem:du_crude}).    Indeed, these show that $\Phi_t:[\bar D\setminus \partial\Omega(0)]\to [\bar D\setminus \partial\Omega(t)]$ is injective, and it is surjective by solving the ODE in \eqref{eq:alpha} backwards in time, with any given terminal condition $\Phi_t(x)=y\in \bar D\setminus \partial\Omega(t)$. Note that all estimates of Lemma~\ref{L.5.2} still apply in this case. 

(b)  First note that $\Phi_t:[\bar D\setminus \partial\Omega(0)]\to [\bar D\setminus \partial\Omega(t)]$
 is measure preserving because it is such when restricted to any closed subset of $\bar D\setminus \partial\Omega(0)$ (due to $\nabla \cdot u\equiv 0$, compactness of $\partial\Omega(t)$, and its continuity in time).  Continuity of $\Phi_t(x)$ and $\partial\Omega(t)$ in time also shows that $\Phi_t$ must preserve connected components of $\bar D\setminus \partial\Omega$.

In addition, since the ODE in \eqref{eq:alpha} has unique backwards-in-time solutions with terminal conditions $\Phi_t(x)=y\in \bar D\setminus \partial\Omega(t)$, and they satisfy $x\in\bar D\setminus \partial\Omega(0)$ (due to $\Phi_t:[\bar D\setminus \partial\Omega(0)]\to [\bar D\setminus \partial\Omega(t)]$ being a bijection), any $\Phi_t(x)$ for $(x,t)\in\partial\Omega_k(0)\times[0,T)$ must be in $\partial\Omega(t)$ (and hence in $\partial\Omega_k(t)$ by continuity).
We then also have that for each $t\in[0,T)$ and $y\in\partial\Omega_k(t)$ there is  a solution of \eqref{eq:alpha} such that $\Phi_t(x)=y$ (obtained by solving \eqref{eq:alpha} backwards), and $\Phi_t$ being a bijection  shows that we must have $x\in\partial\Omega_k(0)$.

Finally,  \eqref{1.32} together with uniform continuity of $u$ on a neighborhood of the compact set $\partial\Omega(t)\times\{t\}$  shows that \eqref{1.3} holds for each $t\in(0,T)$.  Hence $\omega$ is a patch solution to \eqref{sqg}-\eqref{eq:velocity_law} on  $[0,T)$.
\end{proof}

\begin{proof}[Proof of Theorem \ref{T.1.7}]
This is  identical to the proofs of Proposition \ref{P.1.3} and Theorem \ref{T.1.1} in the case $D=\Rm\times\Rm^+$, but with Theorem \ref{thm:local} being valid for all $\alpha\in(0,\frac 12)$ when $D=\Rm^2$ \cite{g}, so that Proposition \ref{prop:equal}, Corollary \ref{C.5.7}, and Theorem \ref{thm:unique_patch} then also hold with $\alpha\in(0,\frac 12)$.
\end{proof}


\begin{thebibliography}{99}

\bibitem{bc} A. Bertozzi and P. Constantin, \it Global regularity for vortex patches, \rm Comm. Math. Phys., {\bf 152} (1993), 19--28

    \bibitem{Butt} T.~Buttke, \it The observation of singularities in the boundary of patches of constant vorticity, \rm
        Physics of Fluids A: Fluid Dynamics {\bf 1}(1989), 1283--1285

%
%

\bibitem{CCCGW}
D.~Chae, P.~Constantin, D.~C\' ordoba, F.~Gancedo and J.~Wu, \it Generalized surface quasi-geostrophic equations with singular velocities, \rm Comm. Pure Appl. Math., {\bf 65} (2012), no. 8, 1037--1066

\bibitem{c} J.-Y. Chemin, \it
Persistance de structures geometriques dans les fluides
incompressibles
bidimensionnels, \rm
Annales de l'\'Ecole Normale Sup\'erieure, {\bf 26} (1993), 1--26.


\bibitem{CIW} P.~Constantin, G.~Iyer and  J.~Wu,
\it Global regularity for a modified critical dissipative quasi-geostrophic equation,
\rm Indiana Univ. Math. J. {\bf 57} (2008), 2681--2692

\bibitem{CMT} P.~Constantin, A.~Majda and E.~Tabak. \textit{Formation
of strong fronts in the 2D quasi-geostrophic thermal active scalar}.
Nonlinearity, {\bf 7} (1994), 1495--1533

%
%
%
%

\bibitem{CFMR}
D. C\'ordoba, M.A. Fontelos, A.M. Mancho, and J.L. Rodrigo,
\it Evidence of singularities for a family of contour dynamics
equations, \rm Proc. Natl. Acad. Sci. USA {\bf 102} (2005), 5949-–5952


\bibitem{Den1} S. Denisov, \it Infinite superlinear growth of the gradient for the two-dimensional Euler equation,
\rm Discrete Contin. Dyn. Syst. A, {\bf 23} (2009), no. 3, 755--764

\bibitem{Den2} S. Denisov, \it
Double-exponential growth of the vorticity gradient for the two-dimensional Euler equation, \rm  to appear in Proceedings of the AMS


\bibitem{DM} D.G.~Dritschel and M.E.~McIntyre, \it
Does contour dynamics go singular? \rm
Phys. Fluids A {\bf  2} (1990), 748-–753

\bibitem{DZ} D.G.~Dritschel and N.J.~Zabusky, \it
A new, but flawed, numerical method for vortex patch evolution in
two dimensions, \rm
J. Comput. Phys. {\bf 93} (1991), 481–484

%

\bibitem{g} F.~Gancedo, \it Existence for the $\alpha$-patch model
and the QG sharp front in Sobolev spaces,
\rm Adv. Math., {\bf 217} (2008), 2569--2598

\bibitem{GS}
F.~Gancedo and R.~M.~Strain, \it Absence of splash singularities for SQG sharp fronts and the Muskat problem, \rm Proc. Natl. Acad. Sci., {\bf 111} (2014), no. 2, 635--639.

%

\bibitem{Holder} E.~H\"older, \it \"Uber die unbeschr\"ankte Fortsetzbarkeit einer stetigen ebenen Bewegung in einer
unbegrenzten inkompressiblen Fl\"ussigkeit, \rm Math. Z. {\bf 37} (1933), 727--738

\bibitem{Yud1} V. I. Judovic, \it The loss of smoothness of the solutions of Euler equations with time (Russian), \rm
Dinamika Splosn. Sredy, Vyp. {\bf 16}, Nestacionarnye Problemy Gidrodinamiki, (1974), 71--78,
121

%
%

\bibitem{KRYZ2}
A. Kiselev, L. Ryzhik, Y. Yao, and A. Zlato\v s,
\it Finite time blow-up for the modified SQG patch equation,
\rm preprint.

\bibitem{KS}
A. Kiselev and V. Sverak, \it Small scale creation for solutions of the
incompressible two dimensional Euler equation, \rm Annals of Math. {\bf 180} (2014), 1205--1220


\bibitem{Majda} A.~Majda, \it Vorticity and the mathematical theory of incompressible fluid
flow, \rm  Comm. Pure Appl. Math. {\bf 39} (1986), 187-–220

\bibitem{mb} A. Majda and A. Bertozzi,
\emph{Vorticity and Incompressible Flow}, Cambridge University Press, 2002.

\bibitem{Mancho} A.~Mancho, \it Numerical studies on the self-similar collapse of the alpha-patches problem,
\rm preprint arXiv:0902.0706

\bibitem{MP} C.~Marchioro and M.~Pulvirenti, {\it Mathematical Theory of Incompressible Nonviscous Fluids}, Springer-Verlag, New York-Heidelberg, 1994.

\bibitem{Nad} N.S. Nadirashvili, \it Wandering solutions of the two-dimensional Euler equation, (Russian) \rm Funktsional.
Anal. i Prilozhen., {\bf 25} (1991), 70--71; translation in \emph{Funct. Anal. Appl.}, 25 (1991), 220--221 (1992)

\bibitem{Ped}
J. Pedlosky, \it Geophysical Fluid Dynamics, \rm Springer, New York, 1987

\bibitem{PHS}
R.T. Pierrehumbert, I.M. Held, and K.L. Swanson,
\it  Spectra of local and nonlocal two-dimensional turbulence,
\rm  Chaos, Solitons Fractals {\bf 4} (1994), 1111--1116

\bibitem{Pulli} D.I.~Pullin, \it Contour dynamics methods, \rm
Annu. Rev. Fluid Mech. {\bf 24} (1992), 89--115

\bibitem{Resnick}
S. Resnick, \it Dynamical problems in nonlinear advective partial differential equations,
\rm Ph.D. thesis, University of Chicago,  1995.

\bibitem{Rodrigo}
J.L. Rodrigo, \it On the evolution of sharp fronts for the
quasi-geostrophic equation, \rm
Comm. Pure Appl. Math., {\bf 58},
(2005), 821--866.


\bibitem{Wolibner} W. Wolibner, \it Un theor\`eme sur l'existence du mouvement plan d'un
uide parfait, homog\`ene,
incompressible, pendant un temps infiniment long (French), \rm Mat. Z., {\bf 37} (1933), 698--726.

%

\bibitem{Yud2} V. I. Yudovich, \it On the loss of smoothness of the solutions of the Euler equations and the
inherent instability of
ows of an ideal
fluid, \rm  Chaos, {\bf 10} (2000), 705--719

\bibitem{Yudth} V.~I.~Yudovich, \it Non-stationary flows of an ideal incompressible fluid, \rm Zh. Vych. Mat., {\bf 3} (1963), 1032--1066

\bibitem{ZlaEuler}
A.~Zlato\v s,
\it Exponential growth of the vorticity gradient for the Euler equation on the torus,
\rm Adv. Math. {\bf 268} (2015), 396--403.

\end{thebibliography}
\end{document}